\documentclass[oneside,english]{amsart}
\usepackage{cmap}%makes pdf searchable
\usepackage[latin1]{inputenc} 
\usepackage{amssymb}
\usepackage[]{epsf,epsfig,amsmath,amssymb,amsfonts,latexsym}
\usepackage{comment}
\usepackage{enumerate}

\usepackage{tikz}
\usetikzlibrary{arrows,decorations.markings}
\usetikzlibrary{calc}

\usepackage{imakeidx}%for making indx
\usepackage{color}
\usepackage{cancel}
\usepackage{amsmath}
\usepackage[normalem]{ulem}
\usepackage{amsfonts}
\usepackage{bbm}
\usepackage{amsthm}
\usepackage[mathscr]{eucal}
\usepackage[active]{srcltx}
\usepackage{fancyhdr}
\usepackage{verbatim}
\usepackage{fullpage}
\usepackage{hyperref}
\usepackage{mathtools}
\hypersetup{
colorlinks = true, %Colours links instead of ugly boxes
urlcolor = blue, %Colour for external hyperlinks
linkcolor = red, %Colour of internal links
citecolor = blue %Colour of citations
}

\def\N{\mathbb N}

\newcommand{\Z}{\mathbb{Z}}

\def \P{\mathbb P}
\def \E{\mathbb E}
\def \R{\mathbb R}
\def \G{\mathcal G}
\def \V{\mathcal V}
\def \E1{\mathcal E}

\def \L{\mathcal L}
\def \T{\mathcal T}

\def\t{\tilde}
\def \m{\vec}
\def \H{\mathcal H}

\def \C{\mathcal C}
\def\mi{{\vec{i}}}
\def\mj{{\vec{j}}}

\newcommand{\Mod}[1]{\ (\text{mod}\ #1)}

\newcommand{\ZD}{\mathbb{Z}^d}
\newcommand{\ZZ}{\mathbb{Z}}
\newcommand{\A}{\mathcal{A}}

\newcommand{\NN}{\mathbb{N}}
\newcommand{\RR}{\mathbb{R}}

\newcommand{\Hom}{\mathit{Hom}}
\newcommand{\Mor}{\mathit{Mor}}
\newcommand{\Aut}{\mathit{Aut}}

\newcommand{\X}{\mathcal{X}}
\newcommand{\Y}{\mathcal{Y}}
\newcommand{\parity}{\mathit{parity}}

\newcommand{\Prob}{\mathit{Prob}}

\newcommand{\ldens}{\underline{d}}
\newcommand{\udense}{\overline{d}}

\newcommand{\Clopen}{\mathit{Clopen}}
\newcommand{\Borel}{\mathit{Borel}}

\newcommand{\Inter}{\mathit{Spec}}

\newcommand{\Ext}{\mathit{Ext}}

\newcommand{\sep}{\mathit{sep}}

\newcommand{\COV}{\mathit{COV}}
\newtheorem{thm}{Theorem}[section]
\newtheorem{lem}[thm]{Lemma}
\newtheorem{prop}[thm]{Proposition}
\newtheorem{cor}[thm]{Corollary}

\newtheorem{lemma}[thm]{Lemma}
\newtheorem{defn}[thm]{Definition}
\newtheorem{remark}[thm]{Remark}

\def\N{\mathbb N}

\def \P{\mathcal P}
\def \QQ{\mathcal Q}
\def \E{\mathbb E}
\def \R{\mathbb R}
\def \G{\mathcal G}
\def \V{\mathcal V}
\def \E1{\mathcal E}

\def \L{\mathcal L}
\def \T{\mathcal T}

\def\t{\tilde}
\def \m{\vec}
\def \H{\mathcal H}

\def \C{\mathcal C}
\def\mi{{\vec{i}}}
\def\mj{{\vec{j}}}

\title{Borel subsystems and ergodic universality for compact $\ZD$-systems via specification and beyond}

\author{Nishant Chandgotia and Tom Meyerovitch}
\makeindex

\begin{document}
	
\maketitle
\begin{abstract}   
	A Borel $\ZD$ dynamical system $(X,S)$ is ``almost Borel universal'' if any free Borel $\ZD$ dynamical system $(Y,T)$ of strictly lower entropy is isomorphic to a Borel subsystem of $(X,S)$, after removing a null set. We obtain and exploit a new sufficient condition for a topological $\ZD$ dynamical system to be almost Borel universal. We use our main result to deduce various conclusions and answer a number of questions. Along with additional results, we prove that a ``generic'' homeomorphism of a compact manifold of topological dimension at least two can  model any ergodic transformation, that non-uniform specification implies almost Borel universality, and that $3$-colorings in $\ZD$ and dimers in $\ZZ^2$ are almost Borel universal.
\end{abstract}
\section{Introduction and statement of results}
In this paper we extend Krieger's generator theorem in several directions. We obtain and exploit a new sufficient condition for a topological $\ZD$ dynamical system (or just a topological dynamical system) $(X,S)$ to be universal with respect to embedding in the ``almost Borel'' category. 
The new results are used to  solve a number of  questions  proposed by Boyle-Buzzi, Lind-Thouvenot, Robinson-\c{S}ahin, Quas-Soo, Weiss and others. Our results are in the interface of measurable, Borel and topological dynamics.

Let $(Y,T)$ be a Borel $\ZD$ dynamical system (or just Borel dynamical system), that is, $T$ is  a $\ZD$ action on $Y$ by Borel bijections on a standard Borel space $Y$. We say that a Borel subset $N$ of $Y$ is null if it has zero measure with respect to any $T$-invariant Borel probability measure.  We say that $(Y,T)$ almost Borel embeds in another Borel dynamical system $(X,S)$ if there is a $T$-invariant Borel subset $Y_0 \subset Y$ so that $Y \setminus Y_0$ is null and an equivariant Borel embedding of $Y_0$ into $X$.

Our  main theorem is the following:

\begin{thm}\label{thm:spec_sequence_implies_univesality}
	Let $(X,S)$ be a topological $\ZD$ dynamical system that admits a flexible marker sequence $\C=(C_n)_{n=1}^\infty \in (2^X)^{\NN}$.
	Then any free  Borel $\ZD$ dynamical system having entropy less than $h(\C)$ almost Borel embeds in $(X,S)$.
\end{thm}
 
The definition of ``a flexible marker sequence  for a topological dynamical system $(X,S)$'' is new, and requires a bit of space to state precisely so we defer it to Section \ref{subsection:Flexible_sequences}. For now, we can say that it is a sort of ``specification property''. Some examples of systems admitting  a flexible marker sequence include generic homeomorphisms of two dimensional compact manifolds, proper $k$-colorings of the $\ZD$ for $k\geq 3$ and domino tilings of $\Z^2$.

Regarding the infinite entropy case, we additionally obtain the following theorem: 

\begin{thm}\label{thm:infinite_entropy}
	\label{thm:infinite_entropy_universal}Let $(X,S)$ be topological $\ZD$ dynamical system with a flexible marker sequence $\C=(C_n)_{n=1}^\infty \in (2^X)^{\NN}$ 
	such that $h(\C)=\infty$.
	Then any free  Borel $\ZD$ dynamical system almost Borel embeds in $(X,S)$. If in addition $\bigcup_{k\in \N} C_k$ is dense in $X$, we can choose the almost Borel embedding so that the pushforward of any $T$-invariant measure has full support in $X$.
\end{thm}

Note that systems satisfying the assumptions of Theorem \ref{thm:infinite_entropy} can in particular almost Borel embed any Borel dynamical system having infinite entropy, whereas in Theorem \ref{thm:spec_sequence_implies_univesality} we require a strict inequality of entropies. 

\subsection{Brief historical survey of previous results}

Krieger's well known finite generator theorem  \cite{MR259068,MR294603} says that 
any (invertible) free,  ergodic measure preserving transformation having entropy less than $\log(N)$ admits an $N$-set generator. Equivalently, Krieger's theorem asserts that the full shift  on a finite number of symbols is universal in the ergodic sense. Krieger also proved that mixing subshifts of finite type are universal \cite{MR0422576}. Lind and Thouvenot proved that hyperbolic (two-dimensional) toral automorphisms are  universal, and used this together with an older result of Oxtoby and Ulam to deduce that any free ergodic measure preserving transformation having \emph{finite entropy} can be realized by a homeomorphism of the two-dimensional torus preserving  Lebesgue measure \cite{MR0584588}. To the best of our knowledge, prior to our new result it remained an open question if the finite entropy assumption can be dropped. Quas and Soo \cite{MR3453367}  proved that non-uniform  specification implies ergodic universality  under two additional hypothesis:  ``Asymptotic $h$-expansiveness'' and the ``small boundary property''.
As a corollary, they deduced that ergodic toral automorphisms are  universal, also in the non-hyperbolic case.

In the same paper, Quas and Soo conjectured that both additional assumptions  ``asymptotic $h$-expansiveness'' and  ``small boundary property'' can be removed. Benjy Weiss  came up with a new proof of the Quas-Soo universality result that avoids the  ``asymptotic $h$-expansiveness'' hypothesis. It is well known that the  ``small boundary property'' hypothesis can fail for some automorphisms of compact abelian groups that have infinite entropy and infinite topological dimension\footnote{One simple example is the  group $(\mathbb{R}/\mathbb{Z})^\mathbb{Z}$ with the shift map. In this particular case  Borel universality  is rather trivial, as this system contains the shift on  $(\{0,1\}^\mathbb{N})^\mathbb{Z}$ as a subsystem, see Proposition \ref{prop:Borel_universaility_of_shift_on_cantor_Set}. Furthermore, this system is universal in the Borel category for rather trivial reasons.}.

Another direction in which Krieger's finite generator theorem has been vastly extended is by replacing the acting group from $\ZZ$ to more general group. The latest breakthrough in this direction  is Seward's theorem \cite{MR3904452}, which is a generalization of Krieger's generator theorem for actions of an arbitrary countable group.

Shifting attention to topological $\ZD$ dynamical systems, Robinson and \c{S}ahin \cite{MR1844076} have shown that any $\ZD$ subshift of finite type $X$  that satisfies a certain property called the \emph{uniform filling property (UFP)} plus a technical condition on periodic points is universal in the ergodic sense. In contrast to  the  case $d=1$, where a shift of finite type is universal if and only it is topologically mixing, uniform filling property is not necessary and topological mixing is not sufficient for universality of $\ZD$ subshifts of finite type. Some examples of $\ZD$ subshifts that satisfy the assumptions of the Robinson-\c{S}ahin result and are thus universal include the hard-square model and $k$-colorings of $\mathbb{Z}^d$ for sufficiently large $k$. Pavlov used the  Robinson-\c{S}ahin result  to deduce universality for $\ZD$ subshifts of finite type with ``nearly full entropy'' \cite{MR3162822}. 

The notions of almost Borel embedding and almost Borel isomorphism  have been introduced by Hochman \cite{MR3077948}, who proved (among other results) an ``almost Borel'' version  of Krieger's theorem. This approach turned out rather fruitful in several aspects \cite{MR3692886}.

  David Burguet has recently obtained  a proof that non-uniform  specification implies universality in the almost Borel sense, confirming a conjecture of Quass and Soo \cite[Conjecture $1$]{MR3008405}. Burguet's theorem 
 also provides an affirmative answer to a question of Boyle and Buzzi \cite[Problem $9.1$]{MR3692886}.

\subsection{Some corollaries of our results}
It has been known for some time that a generic  homeomorphism of a compact manifold of dimension at least $2$ has infinite topological entropy.   
Catsigeras and Troubetzkoy \cite{catsigeras2019ergodic} recently proved  that a generic homeomorphism of a compact manifold of dimension at least $2$   admits an ergodic measure having infinite entropy. 

 Using Theorem \ref{thm:infinite_entropy} together with previously known results about generic homeomorphism of a compact $d$-manifold  we prove that ``most''  homeomorphisms of $M$  in fact admit an invariant measure isomorphic to  any given  free measure-preserving transformation, in particular those having infinite entropy:
 
 Call a topological  $\ZD$ dynamical system $(X,S)$ \emph{fully $\infty$-universal} if any free measure preserving system can be realized as a fully-supported $S$-invariant probability measure on $X$.
\begin{thm} \label{thm:inf_universal_generic}
	Let $M$ be a compact connected  topological manifold (with or without boundary) 
	of dimension $ d \ge 2$. Then there exists a fully $\infty$-universal homeomorphism $h:M \to M$.
	In fact, for any   fully supported, non-atomic probability measure $\mu \in \Prob(M)$ for which the measure of the boundary is zero, 
	there is a dense $G_\delta$ set of fully $\infty$-universal homeomorphisms in the space of homeomorphisms that preserve $\mu$.
\end{thm}

The above result  answers a question of Lind and Thouvenot  \cite{MR0584588}, which had been brought to our attention by Benjy Weiss.

Using Theorem \ref{thm:inf_universal_generic}, we can apply the old argument of  Lind and Thouvenot from \cite{MR0584588} to deduce:
\begin{thm}\label{thm:lebesgue_universal}
	If $k >1$ then 
	for any  free measure preserving $\Z$-system $(Y,\mu,T)$ there is a  homeomorphism $\psi$ of $\mathbb{T}_k= \R^k/\Z^k$ that preserves Lebesgue measure, denoted by $m_{\mathbb{T}_k}$, and so that $(Y,\mu,T)$ is isomorphic to $(\mathbb{T}_k,m_{\mathbb{T}_k},\psi)$ as a measure preserving system.
\end{thm}

We remark  that a fully $\infty$-universal homeomorphism of the torus has infinite topological entropy, thus it cannot be topologically conjugate to a  smooth or even Lipschitz homeomorphism of the torus.

Our main result also recovers the following theorem of David Burguet:
\begin{thm}[Burguet \cite{burguet2019topological}]\label{thm:spec_implies_universal}
	Let $(X,S)$ be a topological $\ZZ$ dynamical system that has non-uniform  specification.
Then $(X,S)$ almost Borel embeds any free  Borel $\ZZ$ dynamical system whose Gurevich entropy is strictly smaller than the topological entropy of $(X,S)$. 

	Furthermore, any free ergodic transformation having entropy less than the topological entropy of $(X,S)$ can be realized as a fully supported invariant probability measure on $(X,S)$.
\end{thm}

A topological dynamical system $(X,S)$ is called \emph{almost Borel universal} if it almost Borel embeds any free Borel $\ZZ^d$ dynamical system whose Gurevich entropy is strictly smaller than the topological entropy of $(X,S)$. The initial motivation that led us to our main result, concerned certain topological dynamical systems that do not have specification. In particular, Robinson and \c{S}ahin \cite{MR1844076} had asked whether proper $3$-colorings and domino tilings of $\Z^2$ are universal; these systems do not have specification. We answer the question positively.
\begin{thm}\label{thm:universal_subhisfts_examples}
	The following $\ZD$-subshifts are almost Borel universal:
	\begin{enumerate}[(i)]
		\item Proper  $k$-colorings of $\ZD$, for all $k \ge 3$ and all $d \ge 1$. (Theorem \ref{thm: universality of hom-shifts})\label{item:theorem_colouring}
		\item Domino tilings in $d=2$. (Theorem \ref{thm: domino universal})\label{item:domino}
	\end{enumerate}
\end{thm}

We prove  \eqref{item:theorem_colouring} of Theorem \ref{thm:universal_subhisfts_examples} as a particular case of a more general result about universality for the space of graph-homomorphisms  from the standard Cayley graph of $\ZD$ to an arbitrary non-bipartite finite graph.
This has some consequences for the Borel structure of a graph generated by   a finite set of commuting measure preserving transformations: After removing an invariant null set and the periodic points, the Borel chromatic number coincides with a basic spectral invariant and is always equal to  $2$ or $3$ (Corollary \ref{cor:Borel_3_coloring}). In  recent work  Gao, Jackson, Krohne and Seward have announced a stronger result namely \cite{gao2018continuous} 
that in fact the Borel chromatic number of such graphs is at most $3$ (so actually the part about  removing a null set in Corollary \ref{cor:Borel_3_coloring} is superfluous). 

Another application of our main result concerns equivariant measurable tiling of free $\ZD$-actions by rectangular shapes:  
\begin{thm}\label{thm:rec_tiling}
	Let $(Y,T)$ be a free Borel $\ZD$ dynamical system and let $F$ be a set of rectangular shapes in $\ZD$ such that the projection of $F$ onto each of the $d$ coordinates is a set of intervals in $\ZZ$ having coprime lengths.
	Then after removing a null set, there exists an  equivariant measurable map from $Y$ to
	the space of tilings of $\Z^d$ by shapes from $F$.
 Furthermore, if the entropy of $(Y,T)$ is sufficiently small (as a function of the set $F$), the map can be chosen to be injective.  
\end{thm}

We prove Theorem \ref{thm:rec_tiling} in Section \ref{section:for Tilings and Flexibility for Dominoes}. 
	This relates to the ``$\ZD$-Alpern Lemma''  \cite{MR1716239,MR2573000}.
	In particular, from the case when $F$ consists of rectangular shapes of size two, it follows that the graph associated with a  free Borel $\ZD$ dynamical system admits a Borel perfect matching, after removing a null set.
	Gao, Jackson, Krohne and Seward obtained remarkable results about equivariant tilings of free $\ZD$ actions, both in the Borel and in the continuous category \cite{gao2018continuous}. In view of these results, it might be possible to avoid removing a null set in the statement Theorem \ref{thm:rec_tiling}, but this goes beyond the scope of this paper.

\subsection{Acknowledgments:}
 We are deeply indebted to Benjy Weiss for many valuable discussions. It was suggested to us by Jerome Buzzi to pursue the almost Borel version of our result. We thank Ron Peled for suggesting the use of reflection positivity and Yinon Spinka for simplifying several arguments regarding graph colorings and graph homomorphisms in Section \ref{section:universal_hom}. We thank Dominik Kwietniak for pointing out the recent publications of Guih\'eneuf and Lefeuvre. We also thank Jon Aaronson, Tim Austin,
Mike Boyle, Raimundo Brice\~no, Mike Hochman, Brian Marcus, Ronnie Pavlov, Brandon Seward, Ay\c se \c Sahin, Terry Soo and Spencer Unger for several discussions and encouragement. We thank the Pacific Institute for Mathematical Studies and the department of mathematics at the University
of British Columbia for warm hospitality. Finally, we thank the anonymous referee for numerous comments and suggestions which improved the presentation significantly and Sebasti\'an Barbieri Lemp for having look at parts of the revised version.
Most of the work in this paper was done when the first author was a postdoctoral fellow at Tel Aviv university and the Hebrew University of Jerusalem.
We gratefully acknowledge funding from European Research Council starting grant 678520 (LocalOrder), and from the Israeli Science foundation (ISF grant no. 1289/17, grant 1702/17, 1052/18 and ISF-Moked grants 2095/15 and 2919/19). 
 
\subsection{Main idea and sketch of proof}
The proof of our main result proceeds by constructing a sequence of ``approximate embeddings'' that converge pointwise to an embedding on a Borel set which has full measure with respect to any invariant measure. This basic idea goes a long way back. Burton and Rothstein defined a certain notion of $\epsilon$-approximate embedding and used it in conjunction with the  Baire category framework to reprove Krieger's generator Theorem.
At this level of generality, this is similar to  the basic approach of Quas and Soo from \cite{MR3453367}. We do not use the Baire category framework, and instead prove convergence of a sequence of approximate embeddings directly, somewhat along the lines of \cite{MR3008405}. The basic difference is that we use a slightly different notion for an ``approximate-embedding''.
Roughly speaking, a ``good approximate embedding'' is a map $\rho:Y \to X$ which is:
\begin{itemize}
	\item
	``Approximately injective'': Informally, this means that there is a  ``big subset $Y_0  \subseteq Y$'' such that  two points in $Y_0$ have ``sufficiently close''
	images under the map if and only if the two points are ``very close''
	(this can be made precise by imposing a topological structure on $Y$ or by using finite Borel partitions to indicate ``closeness'').
	\item 
	``Approximately equivariant'': Informally, this means that on a ``big part of $Y$'' $\rho(T^{\mi}(y))$  is ``very close'' to $S^{\mi}(\rho(y))$, as long as $\mi$ is in a given bounded set (the ``largeness'' of this set  is one of the parameters for the ``quality'' of the approximate embedding).
\end{itemize}
The ``approximate equivariance'' property above can be reformulated by saying that the map is  truly equivariant, but the target space is not $X$ itself, but rather $X^{\ZD}$, which we think of as the ``space of approximate orbits of the system $(X,S)$''.

Using ``Rokhlin towers'' (a classical tool in ergodic theory, which we recall later), and a relatively well-known version of the Shannon-McMillan theorem it turns out that it is quite easy to produce ``approximate embeddings'' which are ``arbitrarily good'', and that this does not require any assumptions on the target system $(X,S)$, except that its topological entropy has to exceed  the entropy of the source system $(Y,T)$.

Our precise definition of an ``approximate embedding'' appears in Section \ref{section:ergodic_universal}. The actual definition we use is slightly more complicated than the above, in particular  because to carry out the full proof we need to assure that a ``sufficiently small modification'' of an approximate embedding  retains its good properties. This seemingly minor issue is perhaps the reason why previous proofs for universality needed to assume an additional property called ``the small boundary condition'', which roughly speaking means that from a dynamical point of view the space $X$ is in some sense ``almost totally disconnected''. Essentially, we overcome this by constructing ``approximate embeddings'' whose image is already totally disconnected (in the space of approximate orbits), in a manner which is ``tolerant to small perturbations''.  
Our construction of an approximate embedding is described in  Lemma \ref{lem:approximate_embedding_exists}.

Construction of an initial approximate embedding is the first step. 
The goal is to produce a sequence of approximate embeddings that converge to a ``genuine''  embedding on a full set.
For this we prove that for systems satisfying our sufficient   condition, a ``small modification'' of a given approximate embedding  can produce a much better one. 
This is the main part of the proof. It is carried out in Lemma \ref{lem:improve_approx_emb}.
This is where we assume a  special property of the system $(X,S)$. Basically, the property we assume is a certain kind of ``specification property'': It is possible to shadow a bunch of ``sufficiently spaced'' orbit segments by a single orbit segment. 
We do not require the ability of being able to ``shadow'' any collection of ``sufficiently spaced'' orbit segments: We only need a ``sufficiently big'' supply of ``good'' orbit segments. By ``sufficiently big'' we roughly mean that these orbit segments are sufficient to ``witness enough entropy in the system''. It should be stressed that the collection of points that constitute ``good orbit segments'' need not be a compact subsystem. For the ``full universality'' result we actually need them to be a dense subset.

We point out that in contrast to \cite{MR3453367} and other works that follow the Burton-Keane-Rothstein-Serafin paradigm, 
our approach is indifferent to the existence of measures which ``locally'' maximize the entropy. For some of the systems that motivate our result the question about existence of measures which locally maximize entropy (but not globally) seems to be a subtle issue.    

To go from ``ergodic universality'' (equivalently embedding a set which is full with respect to a fixed ergodic measure) to ``almost Borel universality'' (equivalently embedding a set which is full with respect to all ergodic measures), we first check that our ``procedure'' for constructing an embedding with respect to a given ergodic measure is  ``Borel'' as a function of the given ergodic measure on $(Y,T)$. As observed in  \cite{MR3077948}, it is possible to apply ``the embedding procedure'' separately on the set of generic points for each ergodic measure and obtain an equivariant  Borel function from a full subset of $Y$ into $X$ that induces an embedding on a full set.
To ensure that the resulting function is actually injective on a full set, we take care so that a pair of points which are generic with respect to distinct ergodic measures will be mapped into different points in $X$. We do this by making sure that the image of a generic $y$  also ``encodes'' the empirical measure associated with $y$.
The part where we go from  ``ergodic universality'' to ``almost Borel universality'' is described in Section \ref{sec:almost_borel_universality}.

We have attempted to make our proof reasonably self contained and refrained from using complicated theorems without providing a  proof. The main results that we do use are the $\ZD$ version of Rokhlin's lemma, the $\ZD$ version of the mean ergodic theorem, and the $\ZD$ version of the Shannon-McMillan theorem (in its weak form stating only convergence in measure). 
For one particular lemma we also  use  a theorem  of Downarowicz and Weiss \cite[Theorem $3$]{MR2048214} (instead of using a slightly longer but self-contained argument).

\subsection{Organization of the paper}
In Section \ref{sec:prelim} we recall some background results and introduce some notation and terminology
that is used in later sections. 
 In Section \ref{sec:SMB_towers} there is a formulation and proof for certain variants of the Shannon-McMillan theorem and of the mean ergodic theorem, adapted to a sequence of Rokhlin towers.
 Section \ref{sec:approx_cov} introduces ``approximate covering numbers''  and formulates related inequalities.  These ``approximate covering numbers''
 provide a ``coding theory'' interpretation for  the entropy of a process.
 Together Sections \ref{sec:prelim} to \ref{sec:approx_cov} can be considered as background and preparation for the main part of the proof.
 In Section \ref{sec:flexible_sequences} we introduce flexible sequences and flexible marker sequences that appear in  the formulation of our main result. Also in Section \ref{sec:flexible_sequences} we  prove that systems having non-uniform specification satisfy our sufficient condition for universality.  
 In Section \ref{section:ergodic_universal} we prove an ergodic version of Theorem \ref{thm:spec_sequence_implies_univesality}  which is slightly weaker and less technical than the ``Almost-Borel'' version, and also Theorem \ref{thm:infinite_entropy} regarding full $\infty$-universality.
 Section \ref{sec:almost_borel_universality} contains a proof of our main result (Theorem \ref{thm:spec_sequence_implies_univesality}). The proofs in  Section  \ref{sec:almost_borel_universality} extend and rely on the previous section. In Section \ref{sec:universlity_generic} we prove Theorems \ref{thm:inf_universal_generic} and \ref{thm:lebesgue_universal} regarding homeomorphisms of manifolds and universality.
 In Section \ref{section:universal_hom} we introduce hom-shifts and prove that they satisfy the assumptions of Theorem \ref{thm:spec_sequence_implies_univesality}. This includes the case of $3$-colorings.
 In Section \ref{section:for Tilings and Flexibility for Dominoes} we prove the result about equivariant rectangular tilings and universality for dimers in $\ZZ^2$. In Section \ref{sec:Not_fully_universal} we exhibit an example of a fully ergodic universal subshift that admits a topological factor which is not universal, providing a negative answer to an old question of Lind and Thouvenot \cite{MR0584588}. In Section \ref{section:SI_subhshifts} we show some non-trivial restrictions regarding  subshifts that can be continuously embedded in the space of $3$-colorings of $\ZZ^2$. This shows that the universality result for the $3$-colorings cannot be deduced by applying 
the  Robinson-\c{S}ahin  universality  result on subsystems. In the last section, we conclude with some further questions.

\section{Preliminaries and notation}\label{sec:prelim}
\subsection{Borel and topological $\ZD$ dynamical systems}
In this paper a \emph{topological $\ZD$ dynamical system} (or topological dynamical system)\index{Definitions and notation introduced in Section 2!topological $\ZD$ dynamical system} is an action of $\ZD$ on a compact metric space by homeomorphisms. Throughout the first few sections $(X,S)$ will denote a topological $\ZD$ dynamical system. To be precise, $X$ will be a compact space with a compatible metric $d_X:X \times X \to \mathbb{R}_+$, and for every $\mi \in \ZD$, $S^{\mi}:X \to X$ will be a homeomorphism so that
$$\forall \mi,\mj \in \ZD,~ S^{\mi + \mj}= S^{\mi} \circ S^{\mj}.$$

A \emph{Borel $\ZD$ dynamical system} (or Borel dynamical system)\index{Definitions and notation introduced in Section 2!Borel $\ZD$ dynamical system} is an action of $\ZD$ on a standard Borel space by Borel automorphisms.
Throughout the first few sections  $\Y = (Y,T)$ will denote a Borel $\ZD$ dynamical system. This means that  $Y$ will be a standard Borel space and that the maps $T^{\mi}:Y \to Y$ are Borel bijections such that $T^{\mi + \mj}= T^{\mi} \circ T^{\mj}$. We generally assume that  $\Y$ is a \emph{free Borel dynamical system}\index{Definitions and notation introduced in Section 2!free Borel dynamical system}. This means that   $T^{\mi}(y) \ne y$ for any $y \in Y$ and $\mi \in \ZD\setminus \{\vec{0}\}$.

The Borel $\sigma$-algebra of $Y$ will be denoted by $\Borel(Y)$. When we say that $A \subset Y$ is measurable without further adjectives, we will mean that $A \in \Borel(Y)$.  

Note that any topological $\ZD$ dynamical system is also a Borel $\ZD$ dynamical system. 

A  Borel probability measure $\mu$ is $T$-invariant if $\mu(T^{-\mi}(A))=\mu(A)$ for every $\mi \in \ZD$ and $A \in \Borel(Y)$.
We denote the space of $T$-invariant probability measures on $Y$ by
\begin{equation}
\index{Definitions and notation introduced in Section 2!$\Prob(\Y)$}\Prob(\Y) = \left\{ T \mbox{-invariant Borel probability measures on } Y\right\}.
\end{equation}
We also denote:
\begin{equation}
\index{Definitions and notation introduced in Section 2!$\Prob_e(\Y)$}\Prob_e(\Y) = \left\{\mbox{ergodic } T \mbox{-invariant Borel probability measures on } Y\right\}.
\end{equation} 
Following \cite{MR3077948}, we say that a Borel set $Y_0 \subset Y$ is \emph{null}\index{Definitions and notation introduced in Section 2!null set} if $\mu(Y_0) = 0$ for all $\mu \in \Prob(\Y)$. A set
is \index{Definitions and notation introduced in Section 2!full set}\emph{full} if its complement is null.

\subsection{Boxes and other subsets  in $\ZD$}
For $n \in \NN$  denote 
\begin{equation}
\index{Definitions and notation introduced in Section 2!$F_n$}F_n = \left\{-n,\ldots,n \right\}^d \subset \ZD.
\end{equation}
Also, for $t \in (0,\infty)$ we denote
\begin{equation}
tF_n  = \left\{-\lfloor t n \rfloor, \ldots ,\lfloor t n \rfloor \right\}^d.
\end{equation}

Let $K \subset \ZD$ be a finite set. We say that $F \subset \ZD$ is \emph{$K$-spaced} \index{Definitions and notation introduced in Section 2!$K$-spaced subsets of $\ZD$} if 
\begin{equation}\label{eq:K_spaced}
(\mi + K) \cap (\mj + K) =\emptyset \mbox{ for distinct } \mi,\mj \in F. 
\end{equation}

Later on we will fix a sequence of positive numbers $(\delta_n)_{n=1}^\infty$ that tends monotonically to $0$. 
With such a sequence fixed, for integers $n>n_0$ we will denote:
\begin{equation}\label{eq:S_n_n0_def}
\index{Definitions and notation introduced in Section 2!$S_{n,n_0}$}S_{n,n_0}= \left\{ A \subset (1-2\delta_n) F_{n}:~ A \mbox { is } (1+\delta_{n_0})F_{n_0} \mbox{-spaced}\right\}.
\end{equation}

\subsection{The space of approximate orbits}
For a compact metric space $X$, the space  $X^{\ZD}$ of $X$-valued functions on $\ZD$ with the product topology is again a compact metrizable space. For $w \in X^{\ZD}$ we  denote by $w_\mi$ the value of $w$ at $\mi \in \ZD$. For $F \subset \ZD$, $w \mid_F \in X^F$ will denote the restriction of $w$ to $F$. 
The group $\ZD$ acts on $X^{\ZD}$ by translations. 
For $\mi \in \ZD$ and $w \in X^{\ZD}$ we write
$(S^{\mi}(w))_{\mj}= w_{\mi+\mj}$.
The resulting topological dynamical system $(X^{\ZD},S)$ is sometimes called the \emph{full shift over $X$}, and the action is called the \emph{shift action}. 
There is a natural embedding of a topological dynamical system  $(X,S)$ into $(X^{\ZD},S)$ given by  $x \mapsto (S^{\mi}(x))_{\mi \in \ZD}$. In other words, each point in $X$  can be identified with its  \emph{$S$-orbit}. This embedding is equivariant with respect to $S$ and the shift. Thus, we identify $X$ with its image in  $X^{\ZD}$ under the orbit map. This justifies the  abuse of notation when we denote by $S$ both the shift action $X^{\ZD}$ and the original action on $X$.  
In this context  we  refer to  $X^{\ZD}$ together with the shift action as the space of \emph{approximate orbits}\index{Definitions and notation introduced in Section 2!space of approximate orbits} for $(X,S)$. This embedding will be useful in the proof of our main result, where we obtain an almost Borel embedding of a Borel dynamical system $(Y,T)$ as a limit of equivariant maps into $X^{\ZD}$.
We will denote the space of approximate orbits by
\begin{equation*}
\index{Definitions and notation introduced in Section 2!$\X$}\X= (X^{\ZD},S).
\end{equation*}
\subsection{Bowen metrics and topological entropy}
For a finite subset $F \subset  \ZD$ and  $x,x' \in X$, we denote
\begin{equation}
\index{Definitions and notation introduced in Section 2!$d_X^F$|(}d_X^F(x,x') = \max_{\mi \in F}\left(d_X(S^{\mi}(x),S^{\mi}(x'))\right).
\end{equation}
For every non-empty finite $F \subset \ZD$,  $d_X^F:X \times X \to \mathbb{R}_+$ defines a metric on $X$ that is compatible with the original one. These metrics are known as ``Bowen metrics''.

We say that a subset  $C \subset X$ is \index{Definitions and notation introduced in Section 2!$(\epsilon,F)$-separated}$(\epsilon,F)$-separated if the $d_X^F$-distance between any pair of distinct points in $C$ is greater than $\epsilon$.
Let
\begin{equation}
\index{Definitions and notation introduced in Section 2!$\sep_\epsilon(A,F)$}\sep_\epsilon(A,F) = \max \left\{ |C| :~ C \subseteq A \mbox { is } (\epsilon,F)\mbox{-separated}\right\}.
\end{equation}

The Bowen metrics can be viewed as  restrictions of the  \emph{pseudo-metrics} on $X^{\ZD}$ given by:
If $\omega,\omega' \in X^{\ZD}$we denote
\begin{equation}
d_X^F(\omega,\omega')= \max_{\mi \in F}d_X(\omega_\mi,\omega'_\mi).
\end{equation}
We will also use the above formula when $\omega,\omega'\in X^F$.
Similarly, if $x \in X$, and $\omega \in X^{\ZD}$ or $\omega \in X^F$ we denote
\begin{equation}
d_X^F(\omega,x)= \max_{\mi \in F}d_X(\omega_\mi,S^{\mi}(x)).
\end{equation}
This is consistent with the natural embedding of $(X,S)$ into $\X$.

For $w \in X^{\Z^d}$, $F\subset \Z^d$ and $A\subset X$, we denote
$$d^F_X(w, A)=\inf_{x\in A}d_X^F(\omega,x).\index{Definitions and notation introduced in Section 2!$d_X^F$|)}$$
Let $h(X,S)$ denote the topological entropy of $(X,S)$. Recall that the topological entropy  of $(X,S)$ is given by:
\begin{equation}\label{eq:h_top_sep}
\index{Definitions and notation introduced in Section 2!Entropy!$h(X,S)$ for a topological dynamical system $(X,S)$}h(X,S)= \lim_{\epsilon \to 0}\limsup_{n\to \infty}\frac{1}{|F_n|}\log \sep_\epsilon(X,F_n).
\end{equation}

\subsection{Non-uniform specification}\label{subsec:non_uniform_spec}
 We say that $(X,S)$ satisfies \index{Definitions and notation introduced in Section 2!non-uniform specification}\emph{non-uniform specification} (as in \cite{MR3488036}) if 
there exists a sequence of increasing functions $g_n:(0,1) \to (0,\infty)$ satisfying the following conditions:
\begin{itemize}
	\item For every $\epsilon>0$, $\lim_{n \to \infty}g_n(\epsilon)= 0$,
	\item For  every $n_1,\ldots,n_s \in \NN$ , $\mi_1,\ldots,\mi_s \in \ZD$ and  $\epsilon>0$ such that 
$$\{\mi_1 + (1+g_{n_1}(\epsilon))F_{n_1},\ldots,\mi_s + (1+g_{n_s}(\epsilon))F_{n_s}\}$$
are pairwise disjoint, 
and any $x_1,\ldots,x_s \in X$ there exists $x \in X$ such that $d_X^{\mi_j + F_{n_j}}(x,x_j) < \epsilon$ for all $1 \le j \le s$.
\end{itemize}

For $d=1$  the property we defined above is a ``symmetric'' version and an easy consequence of the property that Dateyama named \emph{almost weak specification} \cite{MR1041229}. Quas and Soo used Dateyama's terminology in the context of sufficient conditions for ergodic universality \cite[P. $4138$]{MR3008405}. ``Almost weak specification''  also goes under the name \emph{weak specification property}  \cite{MR3546668}.
See \cite{MR3546668}  for an overview of specification-like 
properties and historical background.

\subsection{Morphisms for Borel $\ZD$ dynamical systems}
A morphism between two Borel dynamical systems is a measurable map that intertwines the actions. 
We denote the collection of morphisms from $Y$ to $\X$ by $\Mor(\Y,\X)$\index{Definitions and notation introduced in Section 2!$\Mor(\Y,\X)$}.
An injective morphism $\rho \in \Mor(\Y,\X)$ is a  Borel embedding of $(Y,T)$ into  $\X$.
A bijective morphism gives a Borel isomorphism, as the inverse is necessarily Borel by Souslin's theorem.

\subsection{Borel partitions}
By a Borel partition $\mathcal{P}$ of $Y$ we will mean a partition of $Y$ into finitely or countably many Borel subsets. 
We follow the convention of identifying a  Borel partition $\mathcal{P}$ of $Y$ with the function that maps $y \in Y$ to the unique element of $\mathcal{P}$ that contains $y$, which we denote by  $\index{Definitions and notation introduced in Section 2!$\mathcal{P}(y)$}\mathcal{P}(y)$.

A partition $\mathcal{P}$  refines another partition $\mathcal{Q}$ of $X$ if every partition element $P \in \mathcal{P}$ is contained in some partition element $Q \in \mathcal{Q}$. In this case we write $\mathcal{Q} \preceq \mathcal{P}$.

The least common refinement of two partitions $\mathcal{P}$ and $\mathcal{Q}$ is given by
$$\mathcal{P} \vee \mathcal{Q} = \left\{ P \cap Q :~ P \in \mathcal{P},\; Q \in \mathcal{Q} \right\}.$$

For a Borel partition $\mathcal{P}$ of $Y$   and  a finite subset $F \subset \ZD$, we write  
$$ \index{Definitions and notation introduced in Section 2!$\mathcal{P}^F$} \mathcal{P}^F = \bigvee_{\mi \in F}T^{\mi}(\mathcal{P}).$$

\subsection{Shannon entropy, information and Kolmogorov-Sinai entropy}
The \emph{information function} \index{Definitions and notation introduced in Section 2!$\mathcal{I}_\mu(\mathcal{P})$}$\mathcal{I}_\mu(\mathcal{P}):Y \to \mathbb{R}_+$ for a measurable partition $\mathcal{P}$ is defined to be
\begin{equation}
\mathcal{I}_\mu(\mathcal{P})(y) := -\sum_{P \in \mathcal{P}}1_P(y)\log(\mu(P))= -\log\left(\mu(\mathcal{P}(y))\right).
\end{equation}
More generally, the \emph{relative information function} of the partition $\mathcal{P}$ given a sub-$\sigma$-algebra $\mathcal{F} \subset \Borel(Y)$ is given by
\begin{equation}
\index{Definitions and notation introduced in Section 2!$\mathcal{I}_\mu(\mathcal{P}\mid \mathcal{F})(y)$}\mathcal{I}_\mu(\mathcal{P}\mid \mathcal{F})(y) :=
-\sum_{P \in \mathcal{P}} 1_P(y)\log\left(\mu(P\mid \mathcal{F})(y)\right).
\end{equation}

The Shannon entropy of a measurable partition $\mathcal{P}$ of $Y$ with respect to a Borel probability measure $\mu$ is given by
$$ \index{Definitions and notation introduced in Section 2!Entropy!$H_\mu(\mathcal{P})$}H_\mu(\mathcal{P}) = \int \mathcal{I}_\mu(\mathcal{P}) d\mu.$$
The entropy of a  measurable partition  $\mathcal{P}$  relative to a sub-$\sigma$-algebra  $\mathcal{F} \subset \Borel(Y)$ with respect to a Borel probability measure $\mu$ is given by
$$ \index{Definitions and notation introduced in Section 2!Entropy!$H_\mu(\mathcal{P}\mid \mathcal{F})$} H_\mu(\mathcal{P}\mid \mathcal{F}) = \int \mathcal{I}_\mu(\mathcal{P}\mid \mathcal{F}) d\mu.$$
When $\mathcal{F}=\{\emptyset,Y\}$ is the trivial $\sigma$-algebra, this coincides with the ``non-relative'' case,
For $p \in (0,1)$ we denote
\begin{equation}\label{eq:H_shannon}
\index{Definitions and notation introduced in Section 2!Entropy!$\H(p)$}\H(p) = p\log(\frac{1}{p})+(1-p)\log(\frac{1}{1-p}).
\end{equation} 
$\H(p)$
is the Shannon entropy of a two set partition, where the measure of one of the parts is $p$.

If $\mu \in \Prob(\Y)$, $\mathcal{P}$ is a finite measurable partition and $\mathcal{F}$ is a $T$-invariant sub-$\sigma$-algebra $\mathcal{F} \subset \Borel(Y)$, the Kolmogorov-Sinai entropy of the partition  $\mathcal{P}$ relative to $\mathcal{F}$ is given by
\begin{equation}
\index{Definitions and notation introduced in Section 2!Entropy!$h_\mu\left(Y,T; \mathcal{P} \mid \mathcal{F} \right)$}h_\mu\left(Y,T; \mathcal{P} \mid \mathcal{F} \right) =  \lim_{n \to \infty}\frac{1}{|F_n|}  H_\mu(\mathcal{P}^{F_n} \mid \mathcal{F}).
\end{equation}
For the non-relative case we denote:
\begin{equation}
\index{Definitions and notation introduced in Section 2!Entropy!$h_\mu\left(Y,T; \mathcal{P} \right)$}h_\mu\left(Y,T; \mathcal{P} \right) =  \lim_{n \to \infty}\frac{1}{|F_n|} H_\mu(\mathcal{P}^{F_n} ).
\end{equation}
The  Kolmogorov-Sinai entropy of $(Y,\mu,T)$ (relative to $\mathcal{F}$) is given by:
$$\index{Definitions and notation introduced in Section 2!Entropy!$h_\mu(Y,T \mid \mathcal{F})$} h_\mu(Y,T \mid \mathcal{F}) = sup_{\mathcal{P}}h_\mu(Y,T; \mathcal{P}\mid \mathcal{F}),$$
where supremum is over all finite measurable partitions $\mathcal{P}$.
\subsection{Ergodic universality and Almost Borel universality}\label{subsec:ergodic_almost_borel_universality}
We say that a topological $\ZD$ dynamical system $(X,S)$ is \emph{$t$-universal in the ergodic sense}\index{Definitions and notation introduced in Section 2!$t$-universal in the ergodic sense} if for every free Borel dynamical system $\Y=(Y,T)$ and $\mu \in \Prob_e(\Y)$ such that 
$h_\mu(Y,T) < t$ the ergodic dynamical system $(Y,\mu,T)$ can be realized as an invariant measure on $(X,S)$, in the following sense:
There is a Borel $T$-invariant subset $Y_0 \subset Y$ with $\mu(Y\setminus Y_0)=0$ and $\rho \in \Mor\left((Y_0,T),(X,S)\right)$ so  that $\rho$ is injective on $Y_0$.
We say $(X,S)$ has ergodic universality if it is  $t$-universal with $t=h(X,S)$.
If furthermore under the above conditions we can find $\rho \in \Mor\left((Y_0,T),(X,S)\right)$ as above so that in addition $\mu(\rho^{-1}(U)) >0$ for any open subset $U \subset X$, then we say that 
$(X,S)$ is \index{Definitions and notation introduced in Section 2!fully ergodic $t$-universal}\emph{fully} ergodic $t$-universal.   
In other words, $(X,S)$ is fully ergodic $t$-universal if any free ergodic dynamical system $(Y,T,\mu)$ with entropy less than $t$ can be realized as a fully supported  invariant measure on $(X,S)$.

An \emph{almost Borel embedding} of $\Y$ into $(X,S)$ is an injective morphism from a full $S$-invariant subset of $Y$ to $(X,S)$.
We say that $(X,S)$ is \index{Definitions and notation introduced in Section 2!$t$-universal in the almost Borel sense}\emph{$t$-universal in the almost Borel sense} if every free Borel dynamical system $\Y$ with $h(\Y) <t$ admits an almost Borel embedding into $(X,S)$. 
By $h(\Y)$ we refer to the  \index{Definitions and notation introduced in Section 2!Gurevich entropy}\emph{Gurevich entropy of $\Y$},  given by:
$$\index{Definitions and notation introduced in Section 2!Entropy!$h(\Y)$ for a Borel dynamical system $\Y$}h(\Y) = sup_{\mu \in \Prob(\Y)}h_\mu(Y,T),$$
where $h_\mu(Y,T)$ is the Kolmogorov-Sinai entropy of the measure-preserving system $(Y,\mu,T)$.

It follows from the variational principle  that the Gurevich entropy coincides with the topological entropy for compact topological dynamical systems. See  \cite{MR0430213} for a short proof.

\subsection{Rokhlin towers}\label{subsec:Rokhlin towers}
We now recall the notion of Rokhlin towers and  certain versions of Rokhlin's lemma for $\ZD$ actions. Rokhlin towers are instrumental in the proof of fundamental  results in ergodic theory.

From now on $\Y=(Y,T)$ will denote a free Borel $\ZD$ dynamical system. For $F \subset \ZD$ and $Z \subset Y$, we will use the notation
\begin{equation}
\index{Definitions and notation introduced in Section 2!$T^F Z$}T^F Z= \bigcup_{\mi \in F}T^{\mi}(Z).
\end{equation}
 
 Given a finite $F \subset \ZD$, $\mu \in \Prob(\Y)$ and $\epsilon >0$, we say that $Z \in \Borel(Y)$ is the base of  an \index{Definitions and notation introduced in Section 2!$(F,\epsilon,\mu)$-tower}\emph{$(F,\epsilon,\mu)$-tower}  if $\{T^{\mi}Z\}_{\mi \in F}$ are  pairwise disjoint and
	\begin{equation}\label{eq:Tower_epsilon}
	\mu(T^{F}  Z) >  1- \epsilon.
	\end{equation}
If \eqref{eq:Tower_epsilon} holds for every $\mu \in \Prob_e(\Y)$, we say that $Z$ is the base of an \index{Definitions and notation introduced in Section 2!$(F,\epsilon)$-tower}$(F,\epsilon)$-tower for $\Y$.

\begin{prop}[Rokhlin's lemma for $\ZD$-actions \cite{MR0316680,MR910005}]\label{prop:rokhlin_lemma}
For every   free Borel $\mathbb{Z}^d$ dynamical system $\Y$, every $\mu \in \Prob(\Y)$ , $n \in \NN$ and $\epsilon>0$ there exists an $(F_n,\epsilon,\mu)$-tower.
\end{prop}	

Rokhlin's lemma for  $\ZZ$-actions is classical \cite{rohlin1948general,MR0014222}. In \cite{MR0316680} a much more general version of Rokhlin's lemma was obtained for actions of countable amenable groups (see also \cite{MR2052281}).

We will need  a version for free Borel $\ZD$-actions  of Rokhlin's lemma that works simultaneously for all invariant measures. This is a counterpart of  \cite[Proposition 7.9]{MR2186250} for actions of $\ZD$. The result follows  immediately from \cite[Theorem 3.1]{MR3359054}, which is a much stronger result (with a  more involved proof). An alternative proof can be obtained using the techniques developed in \cite[Corollary 2]{MR624699}.

\begin{prop}\label{prop:borel_ZD_rokhlin_lemma}
For every free Borel $\mathbb{Z}^d$ dynamical system $\Y=(Y,T)$, $n \in \NN$ and $\epsilon>0$ there exists a Borel subset of $Y$ which is the base of an $(F_n,\epsilon,\mu)$-tower for every $\mu \in \Prob_e(\Y)$.
\end{prop}

For completeness, in Section \ref{sec:almost_borel_universality} we provide a proof of Proposition \ref{prop:borel_ZD_rokhlin_lemma}, that assumes the measurable version of Rokhlin's lemma (Proposition \ref{prop:rokhlin_lemma}), but is otherwise self-contained.

We will repeatedly use the following simple results:
\begin{lem}\label{lem:mu_tower_subset}
	If $F \subset  \ZD$ is a finite set, $F' \subseteq F$,  $Z \subset Y$ is the base of an $(F,\epsilon,\mu)$-tower and $Z' \subset Z$ is measurable then 
	\begin{equation}\label{eq:mu_tower_subset}
	\mu\left(Y \setminus T^{F'}Z'\right) < \epsilon + \frac{|F \setminus F'|}{|F|}+\mu\left(Z \setminus Z' \mid Z\right).
	\end{equation}
\end{lem}
\begin{proof}
	Note that 
	\begin{equation}\label{eq:mu_tower_subset1}
	\mu\left(Y \setminus T^{F'}Z'\right) = \mu\left(Y \setminus T^{F}Z\right)+ \mu\left(T^{F\setminus F'}Z\right)+\mu\left(T^{F'}(Z\setminus Z')\right),
	\end{equation}
	because the $F$-translates of $Z$ are pairwise disjoint, $\mu(Z) \le \frac{1}{|F|}$, so $\mu(T^{F \setminus F'}Z) \le \frac{|F \setminus F'|}{|F|}$.
	Again, because the $F$-translates of $Z$ are pairwise disjoint,
	$$\mu\left(T^{F'}(Z\setminus Z')\right)=\sum_{\mi \in F'}\mu\left(T^{\mi}(Z \setminus Z')\right)=\sum_{\mi \in F'}\mu\left(Z \setminus Z' \mid Z\right)\mu(Z) \le \frac{|F'|}{|F|}\mu(Z \setminus Z' \mid Z) \le \mu(Z \setminus Z' \mid Z).$$
	Plugging these estimates in \eqref{eq:mu_tower_subset1} we get \eqref{eq:mu_tower_subset}.
\end{proof}

\begin{lem}\label{lem:eq:mu_tower_subset_cond}
	Suppose that $\delta,\theta >0$ that $\delta+\theta <1$ and  $n <\delta m $. If $Z_n$ is the base of an $(F_n,\epsilon_n,\mu)$-tower and $Z_m$ is the base of an $(F_m,\epsilon_m,\mu)$-tower then
	\begin{equation}
	\mu(T^{ \theta F_m }Z_m\mid Z_n) \le  (1-\epsilon_n)^{-1} (\theta + \delta).
	\end{equation}
\end{lem}
\begin{proof}
		By the law of total probability
	$$
	\mu\left(T^{(\theta+\delta) F_m }Z_m \right) \ge
	\sum_{\mi \in F_n}\mu\left(T^{(\theta+\delta)F_m}Z_m \mid T^{\mi}Z_n \right) \mu\left(T^{\mi}Z_n\right).
	$$
	
	Because $n < \delta m$,  for every
	$\mi \in F_n$ we have  $$T^{\mi+ \theta F_m }Z_m \subseteq T^{\theta F_m + F_n}Z_m \subseteq T^{(\theta+\delta)F_m}Z_m.$$
	So 
	$$
\mu\left(T^{(\theta+\delta) F_m }Z_m \right) 
	\ge \mu(Z_n)\sum_{\mi \in F_n}\mu\left(T^{\mi+ \theta F_m }Z_m\mid T^{\mi}Z_n \right) =
	|F_n|\mu(Z_n) \mu(T^{ \theta F_m }Z_m\mid Z_n).
	$$
	Now  because  $Z_n$ is the base of an $(F_n,\epsilon_n,\mu)$-tower we have that  $|F_n| \mu(Z_n) > 1-\epsilon_n$ and so
	$$
	\mu(T^{ \theta F_m }Z_m\mid Z_n) \le (1-\epsilon_n)^{-1} \mu\left(T^{(\theta+\delta) F_m }Z_m  \right) =  (1-\epsilon_n)^{-1} |(\theta +\delta)F_m| \mu(Z_m) ,
	$$
	where in the last equality we used that $(T^{\mi}Z_m)_{\mi \in (\theta+\delta) F_m}$ are pairwise disjoint because $Z_m$ is the base of an  $(F_m,\epsilon_m,\mu)$-tower. Also, $\mu(Z_m) < \frac{1}{|F_m|}$. So 
	$$
	\mu(T^{ \theta F_m }Z_m\mid Z_n) \le (1-\epsilon_n)^{-1} \frac{|(\theta +\delta)F_m|}{|F_m|}\le (1-\epsilon_n)^{-1} (\theta + \delta).$$
\end{proof}

\section{A Shannon-McMillan theorem and  ergodic for averages along Rokhlin towers}\label{sec:SMB_towers}

In this section we introduce some  ergodic theoretic tools that we need to prove the main result. These are variants of the Shannon-McMillan theorem and of the mean ergodic theorem, adapted to a sequence of Rokhlin towers:
\begin{prop}\label{prop:measure_SM_for_towers}
	Fix a sequence of $(F_n,\epsilon_n,\mu)$-towers with base $Z_n$ with $\epsilon_n\downarrow 0$.
	Let $\epsilon, \theta >0$. Then:
	\begin{enumerate}
		\item[(i)]
	For every measurable partition $\mathcal{P}$ with $H_\mu(\mathcal{P}) < \infty$ the following holds:
	\begin{equation}\label{eq:Rohklin_tower_SM_in_measure}
	\lim_{n \to \infty} \mu\left(  \left|\frac{1}{|\theta F_n|}\mathcal{I}_\mu(\mathcal{P}^{\theta F_n}\mid \mathcal{F}) - h_\mu(Y,T;\mathcal{P}\mid \mathcal{F})\right| > \epsilon  \mid Z_n\right) = 0.
	\end{equation}
	\item[(ii)]
	For every $f \in L^1(\mu)$ the following holds:
	\begin{equation}\label{eq:Rohklin_tower_mean_ET}
	\lim_{n \to \infty}\mu\left( \left| \frac{1}{|\theta F_n|}\sum_{\mi \in \theta F_n}f\circ T^{\mi}- \int f d\mu \right| > \epsilon  \mid Z_n\right) =0.
	\end{equation}
\end{enumerate}
\end{prop}

Variants of Proposition \ref{prop:measure_SM_for_towers}  can be found  in the literature and have been used in particular for ergodic embedding results of the type we are aiming to prove. See for instance \cite[Theorem $4.4$]{MR1844076} and the reference within to Rudolph's proof for the one dimensional case  \cite[Theorem 7.15]{MR1086631}. The proof here  is provided mainly for completeness and for the reader's convenience.

We will present a proof of  Proposition \ref{prop:measure_SM_for_towers} along the lines of  \cite[Theorem 7.15]{MR1086631}. We will rely on the following well known relative version of the Shannon-McMillan theorem:
\begin{prop}\label{prop:measure_SM}\emph{(Relative Shannon-McMillan theorem)}
		For every $\epsilon >0$ 
		\begin{equation}\label{eq:SM_in_measure}
		\lim_{n \to \infty} \mu\left(  \left|\frac{1}{| F_n|}\mathcal{I}_\mu(\mathcal{P}^{ F_n}\mid \mathcal{F}) - h_\mu(Y,T;\mathcal{P}\mid \mathcal{F})\right| > \epsilon \right) = 0.
		\end{equation}
	\end{prop}
 As mentioned in \cite[Proposition 2.2]{MR1355058} it can be obtained by following the proof of Shannon-McMillan theorem (say in \cite[Theorem 9.2.5]{MR797411}).
\begin{proof}[Proof of Proposition \ref{prop:measure_SM_for_towers}] 
	Let us prove $(i)$:
	We will use the shorthand 
	$$h=h_\mu(Y,T;\mathcal{P}\mid \mathcal{F})$$
	and
	\begin{equation*}
	I_{\theta,n} =\mathcal{I}_\mu(\mathcal{P}^{\theta F_n}\mid \mathcal{F}),~ n \in \NN,~ \theta \in (0,1].
	\end{equation*}

	Fix $\theta\in (0,1]$ and $\epsilon >0$. We will show that for every $\delta>0$  there exists $N \in \NN$ such that for all	$n >N$ 
	\begin{equation}\label{eq:conditioned_on_Z_n_SM_in_measure}
	\mu\left(  \left|\frac{1}{|\theta F_n|}I_{\theta,n} - h\right| > \epsilon  \mid Z_n \right) < \delta.
	\end{equation}
	Choose $\zeta, \gamma>0$ such that 
	\begin{equation}\label{eq:zeta_gamma_small}
	\zeta <  \frac{1}{4^{d}}\frac{\epsilon}{2 h + \epsilon} \theta,~ \gamma < \frac{\epsilon}{4}.
	\end{equation}
	Note that $\zeta < \theta$. From the relative Shannon-McMillan theorem (Proposition \ref{prop:measure_SM}) it follows there exists $N_1 \in \NN$ such that for every $n >N_1$ 
	\begin{equation}\label{eq:conditioned_on_set_SM_in_measure}
	\sup_{\mu(A) > \frac{1}{2}\zeta^d}  \mu\left(  \left|\frac{1}{|\theta F_n|}I_{\theta,n} - h\right| > \gamma  \mid A\right)
	 < \frac{\delta}{2}.
	\end{equation}
	
	Also, by possibly increasing $N_1$ we can assume that $\mu( T^{\zeta F_n} Z_n) > \frac{1}{2} \zeta^d$ for every $n> N_1$. Choose $N \in \NN$ so that $N > N_1 (\theta - \zeta)^{-1}$. From \eqref{eq:conditioned_on_set_SM_in_measure}  for any $n > N$,
	\begin{eqnarray*}
		\mu\left(  \left|\frac{1}{|(\theta+\zeta) F_n|}I_{\theta+\zeta,n} -h \right| > \gamma 
		\text{ and }  
		\left|\frac{1}{|(\theta-\zeta) F_n|}I_{\theta-\zeta,n} - h\right| > \gamma \mid T^{\zeta F_n}Z_n\right) < \delta.
	\end{eqnarray*}
	Thus there exists $\mi \in \zeta F_n$ such that
	\begin{eqnarray*}
		\mu\left(  \left|\frac{1}{|(\theta+\zeta) F_n|}I_{\theta+\zeta,n}-h\right| > \gamma 
		\text{ and }  
		\left|\frac{1}{|(\theta-\zeta) F_n|}I_{\theta-\zeta,n} - h\right| > \gamma \mid T^{\mi}Z_n\right) < \delta.
	\end{eqnarray*}		
	If $y\in Y$, $\mi \in \zeta F_n$ and $\left|\frac{1}{|(\theta+\zeta) F_n|}I_{\theta+\zeta,n}(T^{\mi}(y))-h\right| \leq \gamma$ then
	\begin{eqnarray*}
		\frac{1}{|\theta F_n|}I_{\theta,n}(y) -h \leq 
		\frac{(\theta+\zeta)^d}{\theta^d}\frac{1}{|(\theta+\zeta) F_n|}I_{\theta+\zeta,n}(T^{\mi}(y))- h
		\leq \frac{(\theta+\zeta)^d - \theta^d}{\theta^d}h+\frac{(\theta+\zeta)^d\gamma}{\theta^d}\le 
	\end{eqnarray*}
$$ \le  \zeta\frac{d (\theta +\zeta)^{d-1}}{\theta^d}h+\frac{(\theta+\zeta)^d\gamma}{\theta^d} < \epsilon,$$
where in the last inequality we used  \eqref{eq:zeta_gamma_small}. Similarly if $y\in Y$ such that  $\left|\frac{1}{|(\theta-\zeta) F_n|}I_{\theta-\zeta,n}(T^{\mi}(y))- h\right| \leq \gamma$ then 
	\begin{eqnarray*}
	\frac{1}{|\theta F_n|}I_{\theta,n}(y) -h \geq 
	\frac{(\theta-\zeta)^d}{\theta^d}\frac{1}{|(\theta-\zeta) F_n|}I_{\theta-\zeta,n}(T^{\mi}(y)) -h
	\geq \frac{(\theta-\zeta)^d - \theta^d}{\theta^d}h-\frac{(\theta-\zeta)^d\gamma}{\theta^d}>-\epsilon.
\end{eqnarray*}
    It follows that for all $n > N$,
    $$
    	\mu\left(  \left|\frac{1}{|\theta F_n|}I_{\theta,n} -h \right| > \epsilon  \mid Z_n \right) < \delta.
    	$$
	Hence \eqref{eq:conditioned_on_Z_n_SM_in_measure} holds for all $n>N$.
	
	The proof of $(ii)$ is almost identical, except for the following changes: By linearity of the integral and the triangle inequality, it is enough to prove \eqref{eq:Rohklin_tower_mean_ET} assuming  $f$ is non-negative $f$ and $\int f d\mu < \infty$. This time we denote
	$$ h= \int f d\mu$$
	and
	$$ I_{\theta,n} =\sum_{\mi \in \theta F_n}f \circ T^{\mi} \mbox{ for } \theta \in (0,1], n \in \NN.$$
	By the mean ergodic theorem, with these notations it follows  that for every  $\delta >0$  there exists $N$ such that for every $n >N$ \eqref{eq:conditioned_on_set_SM_in_measure} holds. From here the proof is identical as part $(i)$ (except that $h$ and $I_{\theta,n}$ have been redefined).  
\end{proof}

\section{Approximate covering numbers and relative approximate covering numbers for partitions}\label{sec:approx_cov}

Informally, the Shannon-McMillan theorem   interprets  the entropy of a process in terms of the number of ``symbols per iteration'' required to ``encode an orbit segment''. We now introduce some notation to formalize and exploit this interpretation. 

Let $\mathcal{P}$ be a finite Borel partition of $Y$, $\mu \in \Prob_e(Y,T)$, $A \subset Y$ a Borel set and $\epsilon >0$.

	We define the $\epsilon$-covering number of $A$ with respect to the partition $\mathcal{P}$ by:
	\begin{equation}
	\COV_{\mu,\epsilon,\mathcal{P}}(A) = \min\left\{|\mathcal{G}|~:~ \mathcal{G} \subset \mathcal{P} \mbox{ and } \mu(\bigcup \mathcal{G} \cap A) \ge (1- \epsilon)\mu(A) \right\}.
	\end{equation}
	\index{Definitions and notation introduced in Section 4 and Section 5!$\COV_{\mu,\epsilon,\mathcal{P}}(A)$}
	
	The Shannon-McMillan theorem is roughly equivalent to the statement that for an ergodic system $(Y,\mu,T)$ and a measurable partition $\mathcal{P}$ having finite Shannon entropy
	$$\forall \epsilon \in (0,1),\ h_\mu(Y,T; \mathcal{P}) = \lim_{n \to \infty}\frac{1}{|F_n|}\log \left( \COV_{\mu,\epsilon,\mathcal{P}^{F_n}}(A)\right),$$
	for any $A \in \Borel(Y)$ such that $\mu(A) >0$.

	Now let $\mathcal{Q}$ be another finite Borel partition. The $\epsilon$-covering number of $A$ with respect to the partition  $\mathcal{P}$ \emph{relative to $\mathcal{Q}$} is defined by:
	\begin{equation}
	\COV_{\mu,\epsilon,\mathcal{P}\mid \mathcal{Q}}(A) = \min \left\{ \max_{Q \in \mathcal{Q}'}\COV_{\mu,\epsilon,\mathcal{P}}(A \cap Q) : \mathcal{Q}' \subset \mathcal{Q} \mbox{ and } \mu(\bigcup \mathcal{Q}' \cap A) \ge (1-\epsilon)\mu(A)\right\}.
	\end{equation}\index{Definitions and notation introduced in Section 4 and Section 5!$\COV_{\mu,\epsilon,\mathcal{P}\mid \mathcal{Q}}(A)$}
This is closely related to relative entropy by the following formula:
$$\forall \epsilon \in (0,1),\ h_\mu(Y,T; \mathcal{P} \mid \mathcal{Q}^{\ZD}) = \lim_{n \to \infty}\frac{1}{|F_n|}\log \left( \COV_{\mu,\epsilon,\mathcal{P}^{F_n}\mid \mathcal{Q}^{F_n}}(A)\right).$$

In the following sections we will often use relative $\epsilon$-covering numbers, as a ``proxy'' to entropy. We are about to prove a few basic lemmas about  relative $\epsilon$-covering numbers. These statements have closely related well known counterparts in terms of entropy. 

The following is an elementary auxiliary result in basic probability theory that we include for completeness:

\begin{lem}\label{lem:partition_and_big_set}
	Suppose $0 < \delta < \epsilon < 1$ and that $\nu$ is a probability measure on $Y$, $A \subset Y$ measurable set such that $\nu(A) > 1-\delta$, and $\mathcal{P}$ is a measurable partition. Let:
	$$\mathcal{P}_\epsilon= \{P \in \mathcal{P} ~:~ \nu(P \cap A) \ge (1- \epsilon) \nu(P) \}.$$
	Then
	\begin{equation}
	\nu\left(\bigcup \mathcal{P}_\epsilon\right) \ge 1 -\frac{\delta}{\epsilon}.
	\end{equation}
\end{lem}
\begin{proof}
	Denote
	$$G = \bigcup \mathcal{P}_\epsilon \mbox{ and } B= Y \setminus G = \bigcup (\mathcal{P} \setminus \mathcal{P}_\epsilon).$$
	Then
	$$\nu(B \cap A) = \sum_{P\in \mathcal{P} \setminus \mathcal{P}_\epsilon}\nu(P \cap A) < (1- \epsilon) \sum_{P\in \mathcal{P} \setminus \mathcal{P}_\epsilon}\nu(P) =(1-\epsilon)\nu(B).$$
	So
	$$ 1- \delta < \nu(A) = \nu(G \cap A) + \nu(B \cap A) \le \nu(G) + (1-\epsilon)(1-\nu(G)).$$
	It follows that
	$$\nu(G) \ge \frac{\epsilon-\delta}{\epsilon}= 1 -\frac{\delta}{\epsilon}.$$
\end{proof}

	\begin{lem}\label{lem:COV_submultiplicative}
		Suppose  $A \in \Borel(Y)$,  $\mu(A) >0$ and that $\mathcal{P}_1,\mathcal{P}_2,\mathcal{P}_3$ are Borel partitions.
		For every $0<\epsilon<1$  the following  inequalities hold:
		\begin{equation}\label{eq:COV_almost_monotone}
		\COV_{\mu,\epsilon+\epsilon^2,\mathcal{P}_1 \mid \mathcal{P}_2 \vee \mathcal{P}_3}(A)
		\le 
			\COV_{\mu,\epsilon^2,\mathcal{P}_1 \mid \mathcal{P}_2}(A) 
		\end{equation}
		\begin{equation}\label{eq:COV_join_submultiplicative}
		\COV_{\mu,2\epsilon,\mathcal{P}_1 \vee \mathcal{P}_2 \mid \mathcal{P}_3} (A)
		\le
		\COV_{\mu,\epsilon,\mathcal{P}_1 \mid \mathcal{P}_3}(A)\cdot
		\COV_{\mu,\epsilon,\mathcal{P}_2 \mid \mathcal{P}_3}(A)
		\end{equation}
		\begin{equation}\label{eq:cov_triangle}
		\COV_{\mu,\epsilon, \mathcal{P}_1 \mid \mathcal{P}_3}(A)
		\le
		\COV_{\mu,\epsilon^2/6, \mathcal{P}_1 \mid \mathcal{P}_2}(A)\cdot
		\COV_{\mu,\epsilon^2/6, \mathcal{P}_2 \mid \mathcal{P}_3}(A).
		\end{equation}
		
	\end{lem}

	Inequalities \eqref{eq:COV_almost_monotone}, \eqref{eq:COV_join_submultiplicative} and \eqref{eq:cov_triangle} correspond to the following well-known entropy inequalities respectively:
	$$
	h_{\mu}(Y,T; \mathcal{P}_1 \mid \mathcal{P}_2^{\ZD} \vee \mathcal{P}_3^{\ZD} ) \le
	h_{\mu}(Y,T; \mathcal{P}_1 \mid \mathcal{P}_2^{\ZD}) 
	$$
	
	$$
	h_{\mu}(Y,T; \mathcal{P}_1 \vee \mathcal{P}_2 \mid \mathcal{P}_3^{\ZD}) \le  
	h_{\mu}(Y,T; \mathcal{P}_1  \mid \mathcal{P}_3^{\ZD})+
	h_{\mu}(Y,T; \mathcal{P}_2 \mid \mathcal{P}_3^{\ZD})
	$$
	and
	$$
	h_{\mu}(Y,T; \mathcal{P}_1  \mid \mathcal{P}_3^{\ZD}) \le
	h_{\mu}(Y,T; \mathcal{P}_1  \mid \mathcal{P}_2^{\ZD})+
	h_{\mu}(Y,T; \mathcal{P}_2  \mid \mathcal{P}_3^{\ZD}).
	$$
	The expressions appearing in the ``$\epsilon$'' parameter have not been fully optimized, and are not very significant for our applications. We only use the fact that we can make the ``$\epsilon$-expressions'' on the left hand side arbitrary small by making the `$\epsilon$-expressions'' on the right hand side of the inequalities sufficiently small.  	

\begin{proof}[Proof of Lemma \ref{lem:COV_submultiplicative}:]
	By replacing $\mu$ with $\mu(\cdot \mid A)$ we assume without loss of generality that $\mu( A)=1$.	Let us prove \eqref{eq:COV_almost_monotone}.
	By definition of $\COV_{\mu,\epsilon^2,\mathcal{P}_1 \mid \mathcal{P}_2}(A)$,  there exists $\mathcal{P}_2^{(\epsilon^2)} \subseteq \mathcal{P}_2$ such that 
	$$ \mu \left( \bigcup_{P \in \mathcal{P}_2 \setminus  \mathcal{P}_2^{(\epsilon^2)}}P \right) < \epsilon^2,$$
	and so that for every $P \in \mathcal{P}_2^{(\epsilon^2)}$	there exists a subset $\mathcal{G}_P^{(\epsilon^2)} \subset \mathcal{P}_1$ with $|\mathcal{G}_P^{(\epsilon^2)}| \le \COV_{\mu,\epsilon^2,\mathcal{P}_1 \mid \mathcal{P}_2}(A)$ and 
	$$ \mu\left( \bigcup_{Q \in  \mathcal{G}_P^{(\epsilon^2)}} Q \cap P\right)  > (1-\epsilon^2)\mu(P ).$$
	By Lemma \ref{lem:partition_and_big_set}, for every $P\in \P_2^{(\epsilon^2)}$ there exists $\P_3^{(P, \epsilon)} \subset \mathcal{P}_3$ such that for all $P'\in \P_3^{(P, \epsilon)}$ 
	$$\mu \left(\bigcup_{Q \in \G^{(\epsilon^2)}_P} Q \cap P\cap P'\right)> (1-\epsilon)\mu(P'\cap P)$$	
	and 
	$$\mu\left(\bigcup_{P'\in \P_3^{(P, \epsilon)}}P'
	 \cap P\right)> (1-\epsilon)\mu(P).$$
	
 Consider the set 
$\mathcal{P}_{2,3} \subset \mathcal{P}_2 \vee \mathcal{P}_3$ given by 
$$ 
\mathcal{P}_{2,3} = \left\{
	P \cap  P' :~ P \in \mathcal{P}_2^{(\epsilon^2)} ,~ P' \in \mathcal{P}_3^{(P,\epsilon)}
	\right\}.
	$$
We see that
$$\mu \left( \bigcup_{(P\cap P') \in  (\mathcal{P}_2 \vee \mathcal{P}_3) \setminus \mathcal{P}_{2,3}} P \cap P'  \right) \le
\mu \left(\bigcup_{P \in \mathcal{P}_2 \setminus  \mathcal{P}_2^{(\epsilon^2)}}P \right)
+ \sum_{P \in  \mathcal{P}_2^{(\epsilon^2)}}\mu \left(\bigcup_{P' \in \mathcal{P}_3 \setminus \mathcal{P}_3^{(P,\epsilon)}}(P \cap P') \right)  \le \epsilon^2 + \epsilon.
$$
Now if $P \cap P' \in \mathcal{P}_{2,3}$ then 
$$\COV_{\mu,\epsilon^2+\epsilon,\mathcal{P}_1}(A \cap P \cap P') \le\COV_{\mu,\epsilon^2,\mathcal{P}_1}(A \cap P) \le \COV_{\mu,\epsilon^2,\mathcal{P}_1\mid \mathcal{P}_2 }(A).$$
This proves \eqref{eq:COV_almost_monotone}.

	Let us prove \eqref{eq:COV_join_submultiplicative}: By definition, there exists subsets $\mathcal{P}_3^{(1)} \subset \mathcal{P}_3$ and $\mathcal{P}_3^{(2)} \subset \mathcal{P}_3$	such that for $i = 1,2$
	$$\mu \left(\bigcup_{P \in \mathcal{P}_3 \setminus \mathcal{P}_3^{(i)}}P  \right) < \epsilon.$$
	Also for every $P \in \mathcal{P}_3^{(i)}$ there exists a subset $\mathcal{G}_P^{(i)} \subset \mathcal{P}_i$ of size at most $\COV_{\mu,\epsilon,\mathcal{P}_i \mid \mathcal{P}_3}(A)$ such that the union of the elements of $\mathcal{G}_P^{(i)}$ cover all but an $\epsilon$-fraction of the measure of $P$. Let $\mathcal{P}_3^{(1,2)}= \mathcal{P}_3^{(1)} \cap\mathcal{P}_3^{(2)}$. For $P  \in \mathcal{P}_3^{(1,2)}$ define $\mathcal{G}_P^{(1,2)} \subset \mathcal{P}_1 \vee \mathcal{P}_2$ by
	$$\mathcal{G}_P^{(1,2)} = \left\{Q_1 \cap Q_2 :~ Q_1 \in \mathcal{G}_P^{(1)}\mbox{ and } Q_2 \in \mathcal{G}_P^{(2)} \right\}.$$
	Then 
	$$\mu \left(\bigcup_{P \in \mathcal{P}_3 \setminus  \mathcal{P}_3^{(1,2)}} P \right) <\mu \left(\bigcup_{P \in \mathcal{P}_3 \setminus  \mathcal{P}_3^{(1)}} P \right) + \mu \left(\bigcup_{P \in \mathcal{P}_3 \setminus  \mathcal{P}_3^{(2)}} P \right) <  2\epsilon,$$
	and for every $P \in \mathcal \mathcal{P}_3^{(1,2)}$, we have that
    $$ |\mathcal{G}_P^{(1,2)}| \le |\mathcal{G}^{(1)}_P|\cdot|\mathcal{G}^{(2)}_P| \le 
    \COV_{\mu,\epsilon,\mathcal{P}_1 \mid \mathcal{P}_3}(A) \cdot   \COV_{\mu,\epsilon,\mathcal{P}_2 \mid \mathcal{P}_3}(A)
    $$
    and
    $$ \mu\left(P \setminus \bigcup_{Q_1 \cap Q_2 \in \mathcal{G}_P^{(1,2)}}Q_1 \cap Q_2\right) \le
    \mu\left(P \setminus  \bigcup_{Q_1  \in \mathcal{G}_P^{(1)}}Q_1 \right)
+
    \mu\left(P \setminus  \bigcup_{Q_2  \in \mathcal{G}_P^{(2)}}Q_2 \right) < 2\epsilon \mu(P).$$
    This proves \eqref{eq:COV_join_submultiplicative}.
	
	Let us prove \eqref{eq:cov_triangle}:
	Denote 
	$$N_1=  \COV_{\mu,\epsilon^2/6, \mathcal{P}_2 \mid \mathcal{P}_3}(A),~
     N_2= 	\COV_{\mu,\epsilon^2/6, \mathcal{P}_1 \mid \mathcal{P}_2}(A).$$
	By definition, there exists 
	$\mathcal{P}'_3 \subset \mathcal{P}_3$ such that
	$$\mu\left( \bigcup_{P \in \mathcal{P}_3 \setminus  \mathcal{P}'_3} P\right)  < \epsilon^2/6.$$
	and so that for every $P \in \mathcal{P}_3'$ there exists a subset 
	$\mathcal{G}_{P} \subset \mathcal{P}_2$ with $|\mathcal{G}_P| \le N_1$ and
	$$ \mu\left( \bigcup_{Q \in  \mathcal{P}_2 \setminus \mathcal{G}_P} Q \cap P\right)  < \frac{\epsilon^2}{6}\mu(P ).$$
	Similarly, there exists $\mathcal{P}_2' \subset \mathcal{P}_2$  such that $$\mu\left( \bigcup_{Q \in \mathcal{P}_2 \setminus \mathcal{P}'_2}Q\right) < \frac{\epsilon^2}{6}$$ and so that for every $Q \in \mathcal{P}_2'$ there exists a subset $\mathcal{G}_{Q}' \subset \mathcal{P}_1$ with $|\mathcal{G}_Q'| \le N_2$ and
	$$ \mu\left( \bigcup_{Q' \in \mathcal{P}_1 \setminus \mathcal{G}_Q'}Q' \cap Q \right) < \frac{\epsilon^2}{6}\mu(Q ).$$
	It follows that
	$$ \mu \left(\bigcup_{Q \in \mathcal{P}_2'} \bigcup_{Q' \in \mathcal{P}_1 \setminus \mathcal{G}_Q'}Q' \cap Q \right) < \frac{\epsilon^2}{6}.$$
	Let
	$$\mathcal{P}_3^* = \left\{P \in \mathcal{P}_3:~  \mu\left( \bigcup_{Q \in \mathcal{P}_2'}\bigcup_{Q' \in \mathcal{P}_1 \setminus \mathcal{G}_Q'}Q' \cap Q \cap P \right) < \frac{\epsilon}{2}\mu(P)\right\}$$
	and
	$$\mathcal{P}_3'' = \left\{P \in \mathcal{P}_3:~ \mu\left( \bigcup_{Q \in \mathcal{P}_2 \setminus \mathcal{P}_2'}Q \cap P\right) < \frac{\epsilon}{2}\mu(P)\right\}.$$
	Then by Lemma \ref{lem:partition_and_big_set}
	$$\mu\left(\bigcup_{P \in \mathcal{P}_3\setminus  \mathcal{P}_3^*}P \right) < \epsilon/3
	\mbox{ and }
	\mu\left(\bigcup_{P \in \mathcal{P}_3  \setminus \mathcal{P}_3''}P \right) < \epsilon/3. $$
	For every $P \in \mathcal{P}_3'$ let 
		$$\mathcal{G}''_P = \left\{Q' \in \mathcal{P}_1 :~ \exists Q \in \mathcal{G}_P \cap \mathcal{P}_2' \mbox{ s.t. } Q' \in \mathcal{G}'_Q \right\}.$$
		
		Then 
		$$|\mathcal{G}''_P|  \le N_1 \cdot N_2.$$

	Let $\mathcal{P}_3''' = \mathcal{P}_3^* \cap \mathcal{P}_3' \cap \mathcal{P}_3''$. Then
		$$\mu\left(\bigcup_{P \in \mathcal{P}_3 \setminus  \mathcal{P}_3'''}P \right) < \epsilon.$$
	 
	It follows that for every $P \in \mathcal{P}_3'''$
	$$
	\mu\left( \bigcup_{Q' \in \mathcal{P}_1 \setminus  \mathcal{G}''_P}Q' \cap P \right) < \mu( \bigcup_{Q \in \mathcal{P}_2 \setminus \mathcal{P}_2'}Q \cap P) +
	\mu\left( \bigcup_{Q \in \mathcal{P}_2'}\bigcup_{Q' \in \mathcal{P}_1 \setminus \mathcal{G}_Q'}Q \cap Q' \cap P \right).$$
	Choose $P \in \mathcal{P}_3'''$. Because $P \in \mathcal{P}_3''' \subset \mathcal{P}_3''$, 
	$$ \mu( \bigcup_{Q \in \mathcal{P}_2 \setminus \mathcal{P}_2'}Q \cap P) < \frac{\epsilon}{2}\mu(P).$$
	Also, because $P \in \mathcal{P}_3''' \subset \mathcal{P}_3^*$, 
	$$\mu\left( \bigcup_{Q \in \mathcal{P}_2'}\bigcup_{Q' \in \mathcal{P}_1 \setminus \mathcal{G}_Q'}Q \cap Q' \cap P \right) < \frac{\epsilon}{2}\mu(P).$$
	
	Thus 
	$$ 
	\mu\left( \bigcup_{Q' \in \mathcal{P}_1 \setminus  \mathcal{G}''_P}Q' \cap P \right) <
	 \epsilon \mu(P).
	$$
	This shows that  $\COV_{\mu,\epsilon,\mathcal{P}_1\mid \mathcal{P}_3}(A) < N_1 \cdot N_2$. 
\end{proof}	

From now on, let $(\epsilon_{n})_{n=1}^\infty$  be a sequence of positive numbers  tending to $0$  and let $(Z_n)_{n=1}^\infty$ be a sequence of Borel subsets of $Y$ so that for every $n \in \NN$, $Z_n \subset Y$ is the base of an $(F_n,\epsilon_n)$-tower for $\Y$.

The following is a manifestation of the Shannon-McMillan theorem  for towers (Proposition \ref{prop:measure_SM_for_towers}), in terms of $\epsilon$-covering numbers:
\begin{lem}\label{lem:tower_cov}
	
	Let $\mathcal{P}$ be a finite Borel partition  of $Y$, $\theta \in (0,1)$ and $\mu \in \Prob_e(\Y)$. For every $\epsilon >0$ there exists $N \in \N$ so that for every $n >N$
	\begin{equation}\label{eq:COV_small1}
	\COV_{\mu,\epsilon,\mathcal{P}^{\theta F_n}}( Z_n)< \exp\left( |\theta F_n|\left(h_\mu(Y,T ; \mathcal{P})+\epsilon \right)\right).
	\end{equation}
\end{lem}
\begin{proof}
	For $n \in \NN$ and $\epsilon >0$ and
	\begin{equation}
	\tilde Z_{n,\epsilon} = \left\{y \in Z_n:~  
		\frac{1}{|\theta F_n|}\log\left(\mu(\mathcal{P}^{\theta F_n}(y))^{-1}\right) \le h_\mu(Y,T;\mathcal{P} )+ \epsilon \right\}.
	\end{equation}
	Choose $y \in  \tilde Z_{n,\epsilon}$. We have that
	$$\mu\left(\mathcal{P}^{\theta F_n}(y)  \right) \ge e^{-(h_\mu(Y,T ; \mathcal{P}) + \epsilon)|\theta F_n|}.$$
	It follows that  every element of $\mathcal{P}^{\theta F_n}$ that intersects $ \tilde Z_{n,\epsilon}$ covers at least an
	$e^{-(h_\mu(Y,T ; \mathcal{P}) + \epsilon)|\theta F_n|}$ fraction of $Y$. Thus  $\tilde Z_{n,\epsilon}$ can be covered by at most
	$$e^{(h_\mu(Y,T ; \mathcal{P}) + \epsilon)|\theta F_n|}$$ elements of $\mathcal{P}^{\theta F_n}$. By Proposition \ref{prop:measure_SM_for_towers},  if $n$ is sufficiently big then
	$$\mu\left(\tilde Z_{n,\epsilon}\right) \ge (1-\epsilon)\mu( Z_n).$$
	This shows that \eqref{eq:COV_small1} holds.
\end{proof}

We remark that Proposition \ref{prop:measure_SM_for_towers} also implies a corresponding lower bound on $\COV_{\mu,\epsilon,\mathcal{P}^{\theta F_n}}( Z_n)$. We will not use it in this paper.

Here is a ``relative'' version of Lemma \ref{lem:tower_cov}:
\begin{lem}\label{lem:rel_tower_cov}
	Let $\mathcal{P},\mathcal{Q}$ be finite Borel partitions  of $Y$, $\theta \in (0,1)$ and $\mu \in \Prob_e(\Y)$. For every $ \epsilon >0$ there exists $N \in \N$ so that for every $n >N$
	\begin{equation}\label{eq:rel_COV_small1}
	\COV_{\mu,\epsilon,\mathcal{P}^{\theta F_n} \mid \mathcal{Q}^{\theta F_n}}( Z_n)< \exp\left( |\theta F_n|\left(h_\mu(Y,T ; \mathcal{P}\mid \mathcal{Q}^{\ZD})+\epsilon \right)\right).
	\end{equation}
\end{lem}
\begin{proof}
	Choose $ \epsilon >0$. For $n \in \NN$  let $\mu_{Z_n}(\cdot) = \mu(\cdot \mid Z_n)$ and
	\begin{equation}
	\tilde Z_{n,\epsilon} = \left\{y \in Z_n:~  
	\frac{1}{|\theta F_n|}\log\left[\mu(\mathcal{P}^{\theta F_n}(y)\mid \mathcal{Q}^{\theta F_n}(y))^{-1}\right] \le h_\mu(Y,T;\mathcal{P} \mid \mathcal{Q}^{\ZD})+ \epsilon/2 \right\}.
	\end{equation}

	We define
	\begin{equation}
	\mathcal{Q}'_n = \{Q  \in \mathcal{Q}^{\theta F_n}:~ \mu_{Z_n}(\tilde Z_{n,\epsilon} \cap Q) > (1- \epsilon)\mu_{Z_n}( Q)\}.
	\end{equation}
	Choose $Q \in \mathcal{Q}'_n$. Then for every $y \in \tilde Z_{n,\epsilon} \cap Q$,
	$$\mu\left(\mathcal{P}^{\theta F_n}(y)  \cap Q \right) \ge e^{-(h_\mu(Y,T ; \mathcal{P}\mid \mathcal{Q}^{\ZD}) + \epsilon/2)|\theta F_n|}\mu(Q).$$
	Thus for every $Q \in \mathcal{Q}'_n$ we see that $Q \cap \tilde Z_{n,\epsilon}$ can be covered by at most
	$$e^{(h_\mu(Y,T ; \mathcal{P}\mid \mathcal{Q}^{\ZD}) + \epsilon/2)|\theta F_n|}$$ elements of $\mathcal{P}^{\theta F_n}$. 
	This means that 
	\begin{equation}
	\COV_{\mu,\epsilon,\mathcal{P}^{\theta F_n}}(Q \cap Z_{n})< \exp\left( |\theta F_n|\left(h_\mu(Y,T;\mathcal{P} \mid\mathcal{Q}^{\ZD})+\epsilon/2 \right)\right).
	\end{equation}
	Recall that by definition for $Q \in \mathcal{Q}'_n$ we have
	$$\mu\left(Q \cap \tilde Z_{n,\epsilon}\right) > (1-\epsilon)\mu(Q \cap Z_n).$$
	
	By Proposition \ref{prop:measure_SM_for_towers} there exists $N \in \NN$ so that for every $n >N$, $\mu_{Z_n}(\tilde Z_{n,\epsilon})  > 1- \epsilon^2$.
	By Lemma \ref{lem:partition_and_big_set} applied with $\nu = \mu_{Z_n}$, $A=\tilde Z_{n,\epsilon}$, $\mathcal{P}= \mathcal{Q}^{\theta F_n}$ and $\delta = \epsilon^2$, we see that for every $n >N$  we have $\mu(\bigcup \mathcal{Q}'_n \cap Z_n) \ge (1-\epsilon) \mu(Z_n)$. By definition of the relative $\epsilon$-covering, we see that for $n >N$ \eqref{eq:rel_COV_small1} holds.
\end{proof}

\section{Specification and flexible sequences}\label{sec:flexible_sequences}
We now turn our attention from ergodic theory and measurable dynamics to topological dynamics.
Recall that $(X,S)$ is a topological $\ZD$ dynamical system. This means that $X$ is a compact metric space with a metric $d_X:X \times X \to \mathbb{R}_+$ and 
 for every $\mi \in \ZD$, $S^{\mi}:X \to X$ is a homeomorphism so that
$ S^{\mi + \mj}= S^{\mi} \circ S^{\mj}$ for every $\mi,\mj \in \ZD$.

\subsection{Flexible sequences}\label{subsection:Flexible_sequences}\index{Definitions and notation introduced in Section 4 and Section 5!flexible sequences}
Let $(\epsilon_n)_{n=1}^\infty$ and  $(\delta_n)_{n=1}^\infty$ be sequences of numbers between $0$ and $1$, both decreasing  to $0$, and so that
\begin{equation}\label{eq:n_delta_n_to_infty}
\lim_{n\to \infty}n\delta_n=\infty.
\end{equation}

Throughout most of the paper the sequences  $(\epsilon_n)_{n=1}^\infty$ and  $(\delta_n)_{n=1}^\infty$ will be fixed in the background.
%We think of  $(\epsilon_n)_{n=1}^\infty$ as the ``specification precision sequence'' and think of $(\delta_n)_{n=1}^\infty$ as the ``specification gap sequence''.

For $n,m \in \NN$, we will use the notation \index{Definitions and notation introduced in Section 4 and Section 5!$n \ll m$} $n \ll m$  to mean that ``$m$ is much bigger than $n$'', in some manner that takes into account the sequences $(\epsilon_n)_{n=1}^\infty$ and  $(\delta_n)_{n=1}^\infty$.

Formally it can be defined as follows:
\begin{equation}\label{eq:n_ll_m}
n \ll m \mbox{ iff }
\epsilon_m < \frac{1}{10^6 d}\epsilon_n\cdot \frac{1 -\epsilon_n}{|F_n|},~ \delta_m < \frac{1}{10^6 d}\delta_n \cdot \frac{1-\epsilon_n}{|F_n|} \mbox{ and } 
n^2 < \frac{\delta_m}{1000 d}m.
\end{equation}
The notation $n\ll m$ obscures the dependence on the sequences  $(\epsilon_n)_{n=1}^\infty$ and  $(\delta_n)_{n=1}^\infty$, but it will simplify the expressions appearing in later parts of the proof.

Recall that  $(X,S)$ is a  topological $\ZD$ dynamical system.
Let $\C =(C_n)_{n=1}^\infty \in (2^X)^{\NN}$  be a sequence of finite subsets of $X$, namely $C_n \subset X$ for every $n$.

	Given $x \in X$ and $K \subset \ZD$, we say that it \index{Definitions and notation introduced in Section 4 and Section 5!$x \in X$, shadows $W\in C_k^{K}$}shadows $W  \in C_k^{K}$ if
\begin{equation}\label{eq:x_interpulate}
d_X^{F_{k}}(S^{\mi}(x),W_\mi ) <  \frac{1}{4}\epsilon_{k} \mbox{ for every  } \mi \in K.
\end{equation}
	
Given $k,n \in \NN$ with $k \ll n$, let $\Inter_{n,k}(\C)$\index{Definitions and notation introduced in Section 4 and Section 5!$\Inter_{n,k}(\C)$} denote the collection of  subsets 
$ W  \in C_k^{K}$, where $K \subset (1-\delta_n)F_n$ is $(1+\delta_k)F_k$-spaced.
Let 
$\Inter_n(\C)= \bigcup_{k \ll n}\Inter_{n,k}(\C)$\index{Definitions and notation introduced in Section 4 and Section 5!$\Inter_n(\C)$}.

	Let $\C =(C_n)_{n=1}^\infty \in (2^X)^{\NN}$  be a sequence of finite subsets of $X$, namely $C_n \subset X$ for every $n$.
	We call the sequence $\C=(C_n)_{n=1}^\infty \in (2^X)^{\NN}$ \emph{flexible}  with respect to
	$((\epsilon_n)_{n=1}^\infty,(\delta_n)_{n=1}^\infty)$ if for every $k \in \NN$, the set $C_k \subset X$ is $(\epsilon_k,F_k)$-separated, and in addition
	for any $W \in \Inter_n(\C)$ there exists $x \in C_n$ that shadows $W$. 
		Thus, if $\mathcal C$ is a flexible sequence then for all $n\in \N$, there exists a function
		\begin{equation}
	\index{Definitions and notation introduced in Section 4 and Section 5!$\Ext_n$}	\Ext_n: \Inter_n(\C) \to C_n,
		\end{equation}
		so that
		\begin{equation}\label{eq:Ext_def}
		\Ext_n(W) ~\mbox{shadows } W \mbox { whenever } W \in \Inter_n(\C).
		\end{equation}
	
	The phrase ``$(X,S)$ admits a flexible sequence $\C=(C_n)_{n=1}^\infty \in (2^X)^\NN$'' formally means that there exists monotone decreasing sequences $(\epsilon_n)_{n=1}^\infty$ and $(\delta_n)_{n=1}^\infty$ both tending to $0$ so that \eqref{eq:n_delta_n_to_infty} holds  and so that $\C$ is a flexible sequence with respect to those.

	Given a  flexible sequence $\C \in (2^X)^\NN$, we say that it is a \index{Definitions and notation introduced in Section 4 and Section 5!flexible marker sequence}\emph{flexible marker sequence} (with respect to $(\epsilon_n)_{n=1}^\infty$ and $(\delta_n)_{n=1}^\infty$) if there exists $\epsilon_\star>0$ such that 
	\begin{equation}\label{eq:marker_property}
	\text{for all } x\in X^{\ZD},  \mbox{ and } n \in \NN, \mbox{ the set } \left\{\mi \in \ZD~:~ d_X^{F_n}(S^{\mi}(x),C_n) < {\epsilon_\star} \right\}
	\mbox{ is } (1- \delta_n)F_n \mbox{-spaced}. 
	\end{equation}
	
	For  a flexible sequence $\C= (C_n)_{n=1}^\infty \in (2^X)^{\NN}$ we define:
	\begin{equation}\label{eq:h_C_def}
\index{Definitions and notation introduced in Section 4 and Section 5!Entropy!$h(\C)$ for a flexible sequence $\C$}	h(\C)= \lim_{n\to \infty}\frac{1}{|F_n|}\log |C_n|
	\end{equation}
	and
	\begin{equation}\label{eq:h_epsilon_C_lim}
\index{Definitions and notation introduced in Section 4 and Section 5!Entropy!$h_{\epsilon}(\C)$ for a flexible sequence $\C$}	h_\epsilon(\C)= \liminf_{n \to \infty}\frac{1}{|F_n|}\log \sep_{\epsilon}(C_n,F_n).
	\end{equation}
The limit in \eqref{eq:h_C_def} always exists in the extended sense, as we show in Lemma \ref{lem:limit_h_C_exists} below.

From now, assume that   $(\epsilon_n)_{n=1}^\infty$ and $(\delta_n)_{n=1}^\infty$ are fixed, and that $\C=(C_n)_{n=1}^\infty$ is a flexible sequence for the topological $\ZD$ dynamical system $(X,S)$.

We also assume that  $\Ext_n:\Inter_n(\C) \to C_n$ satisfies \eqref{eq:Ext_def}.
\begin{lem}\label{lem:Ext_sep}
	Suppose that $n \ll m$, that $K \subset (1-\delta_m)F_m$ is  $(1+\delta_n)F_n$-spaced and that $W, W' \in C_n^K$.
	Then:
	\begin{equation}\label{eq:d_W_W_prime}
	d_X^{F_m}\left(\Ext_{m}(W),\Ext_{m}(W')\right) \ge \frac{1}{2}\max_{\mi \in K}d_X^{F_n}\left(W_\mi,W'_\mi\right),
	\end{equation}
	with equality iff $W=W'$, in which case both sides of the inequality are zero. In particular, if $W \ne W'$,
	$$d_X^{F_m}\left(\Ext_{m}(W),\Ext_{m}(W')\right) > \frac{1}{2}\epsilon_n.$$
\end{lem}
\begin{proof}
	Suppose  $W,W' \in C_n^K$. If $W=W'$ then it is trivial that both sides of the inequality in \eqref{eq:d_W_W_prime} are zero.
	So suppose
	$W \ne W'$. Choose $\mi \in K$ so that $W_\mi \ne W'_\mi$.
	Since the elements of $C_n$ are $(\epsilon_n, F_n)$-separated it follows that $d_X^{F_n}(W_\mi,W'_\mi) > \epsilon_n$.
	We have
	$d_X^{\mi +F_n}(\Ext_{m}(W),S^{-\mi}(W_\mi)) < \frac{1}{4}\epsilon_n $ and 
	$d_X^{\mi +F_n}(\Ext_{m}(W'),S^{-\mi}(W'_\mi)) < \frac{1}{4}\epsilon_n $. 
	By the triangle inequality 
	$$d_X^{F_m}(\Ext_{m}(W),\Ext_{m}(W')) \ge d_X^{F_n}(W_\mi,W'_\mi)-2 \frac{\epsilon_n}{4} >\frac{1}{2}d_X^{F_n}\left(W_\mi,W'_\mi\right),$$
	where in the last inequality we used that $d_X^{F_n}(W_\mi,W'_\mi) > \epsilon_n$.
	
\end{proof}
\begin{lem}\label{lem:C_m_lower_bound}
	Suppose $n \ll m$. Then for any $\eta > \epsilon_n$
	\begin{equation}\label{eq:sep_C_m_big1}
	\sep_{\frac{1}{2}\eta}(C_m, F_m)\geq \left(\sep_{\eta}(C_n,F_n)\right)^{\frac{|(1- 3\delta_m) F_m|}{|(1+2\delta_n) F_n|}}.
	\end{equation}
	In particular,
	\begin{equation}\label{eq:sep_C_m_big}
	\sep_{\frac{1}{2}\epsilon_n}(C_m, F_m)\geq |C_n|^{\frac{|(1- 3\delta_m) F_m|}{|(1+2\delta_n) F_n|}}.
	\end{equation}
\end{lem}
\begin{proof}
Fix $\eta >0$ and $n \ll m$ so that $\epsilon_n < \eta$. Define $$K=  (1- \delta_m)F_m \cap (2+ 3\delta_n)n \ZD.$$
	So $K\subset (1-\delta_m)F_m$ 
	is $(1+\delta_n)F_n$-spaced. For $0 < \eta$, let $X_\eta \subset C_n$ be 
	an $(\eta,F_n)$-separated set of maximal cardinality. By Lemma \ref{lem:Ext_sep}, the image of $X_\eta^K$ under the function $\Ext_m:\Inter_m(\C) \to C_m$ is $(\frac{1}{2}\eta,F_m)$-separated.
	This shows that $C_m$ contains an $(\frac{1}{2}\eta,F_m)$-separated set of size at least $|X_\eta|^{|K|}$.	
	The inequality \eqref{eq:sep_C_m_big1} follows because
	$$|K| \ge \frac{|(1- 3\delta_m) F_m|}{|(1+2\delta_n) F_n|}.$$
\end{proof}

\begin{lem}\label{lem:limit_h_C_exists}
	Let $\C= (C_n)_{n=1}^\infty \in (2^X)^{\NN}$ be a flexible sequence, then the limit in \eqref{eq:h_C_def} exists. Also, for every $\epsilon >0$, $h_\epsilon(\C) \in [0,+\infty)$ and $h(\C) \in [0,\infty]$.
	Furthermore, we have:
	\begin{equation}\label{eq:h_C_lim_h_C_epsilon}
	h(\C)= \lim_{\epsilon \to 0}h_\epsilon(\C).
	\end{equation}
\end{lem}
\begin{proof} The function  $\epsilon \mapsto h_{\epsilon}(\C)$ is monotone non-increasing on $(0,\infty)$ so the limit $\lim_{\epsilon\to 0}h_{\epsilon}(\C)$ exists.
	
	By Lemma \ref{lem:C_m_lower_bound}, if $n \ll m$ then
	$$ \frac{1}{|(1-3\delta_m)F_m|}\log \sep_{\frac{1}{2}\epsilon_n}(C_m,F_m) \ge \frac{1}{|(1+2\delta_n)F_n|}\log |C_n|.$$
	Thus, we have
	$$ h_{\frac{1}{2}\epsilon_n}(\C)=\liminf_{m \to \infty}\frac{1}{|{ F_m}|}\log \sep_{\frac{1}{2}\epsilon_n}(C_m,F_m) \ge \frac{1}{|(1+2\delta_n)F_n|}\log |C_n|.$$
	
	Taking $n \to \infty$ on both sides of the inequality, it follows that
	$$\lim_{\epsilon \to 0}h_\epsilon(\C)\ge \limsup_{n \to \infty}\frac{1}{|F_n|}\log |C_n|.
	$$
	On the other hand, $\sep_{\epsilon}(C_m,F_m)\le |C_m|$ for every $m$ so 
	$h_\epsilon(\C)\leq \liminf_{n\to \infty}\frac{1}{|F_n|}\log |C_n|$ for every $\epsilon >0$.
	This proves that the limit  in \eqref{eq:h_C_def} exists and \eqref{eq:h_C_lim_h_C_epsilon} holds.
\end{proof}

The formula \eqref{eq:h_C_lim_h_C_epsilon}  can be viewed as an alternative definition of $h(\C)$. In particular, it explains the fact that $h(\C) \le h(X,S)$ for any flexible sequence.

\subsection{Non-uniform specification implies a flexible marker sequence}

Our next goal is to show that non-uniform specification  implies the existence of a suitable flexible marker sequence. 

\begin{prop}\label{prop:specification_implies_flexible}
	For $d=1$ if $(X, S)$ has non-uniform specification then it has a flexible marker sequence $\C=(C_n)_{n=1}^\infty$ such that $h(\C)=h(X, S)$ and $\bigcup_{n=1}^\infty C_n$ is dense in $X$.
\end{prop}

Burguet's theorem about non-uniform specification implying almost Borel  universality for $\mathbb{Z}$-actions (Theorem \ref{thm:spec_implies_universal}) follows directly from our Theorem \ref{thm:spec_sequence_implies_univesality} combined with Proposition \ref{prop:specification_implies_flexible} above.

\begin{remark}
	The only part of our proof of  Proposition \ref{prop:specification_implies_flexible} that is specific for $d=1$ is  our proof of Lemma \ref{lem:marker_triple} which uses a result of Downarowicz-Weiss \cite{MR2048214}. As an alternative, we  could have used a direct proof of the conclusion of Lemma \ref{lem:marker_triple} by strengthening the positive entropy assumption assuming and replacing it with non-uniform specification. This would work for any $d\ge 1$.
\end{remark}	

The main issue in our proof of Proposition \ref{prop:specification_implies_flexible} is to extract suitable ``markers''. A similar component appears in most proofs of ergodic universality. The reader is invited to compare our proof below with  the ``Marker Lemma'' of Quas and Soo  \cite[Lemma $24$]{MR3453367}.

The following basic  ``Marker Lemma'' says that for a system with positive entropy it is possible to find $F_n$-orbit segments that do not overlap with each other and have only trivial self overlaps.
\begin{lem}\label{lem:marker_triple}
	If $(X,S)$ is a topological $\Z$ dynamical system with positive topological entropy, then for all sufficiently small $\epsilon_1 >0$, if $n_1 \in \NN$ is sufficiently big there exist $\t x^{(0)},\t x^{(1)}, \t x \in X$ so that
	\begin{equation}\label{eq:x_0_x_1_no_overlap}
	d_X^{F_{n_1} \cap (i +F_{n_1})}(S^{-i}(\t x^{(0)}),\t x^{(1)}) > 2\epsilon_1 \mbox{ for all } i \in {\frac32}F_{n_1},
	\end{equation}
	\begin{equation}\label{eq:x_0_1_no_self_overlap}
	d_X^{F_{n_1} \cap (i +F_{n_1})}(S^{-i}(\t x^{(t)}),\t x^{(t)}) > 2\epsilon_1 \mbox{ for all } i \in {\frac32}F_{n_1}\setminus \{0\} \mbox{ and } t \in \{0,1\},
	\end{equation}
	and
	\begin{equation}\label{eq:x_0_1_no_overlap_t_x}
	d_X^{F_{n_1} \cap (i +F_{n_1})}(S^{-i}(\t x^{(t)}),\t x) > 2\epsilon_1 \mbox{ for all } i \in {\frac32}F_{n_1}\mbox{ and } t \in \{0,1\}.
	\end{equation}
	
\end{lem}

\begin{proof} 
	From a result of  Downarowicz and Weiss \cite[Theorem $3$]{MR2048214} it follows that for any $\mu \in \Prob_{e}(X,S)$ for which  $h_\mu(X,S) >0$, if $\epsilon_1 >0$ is sufficiently small there exists $\tilde h>0$ such that
	$$\lim_{n \to \infty}\mu \left(\left\{x \in X:~ \exists i \in F_{e^{\tilde h n}} \setminus \{\vec{0}\} ~ \mbox{ s.t. } d_X^{\frac{1}{2}F_n}(x,S^{i}(x)) \le 2\epsilon_1 \right\}\right) =0.$$
	Because $F_{n} \cap (i + F_n)$ contains a translate of $\frac{1}{2}F_n$ for all $i \in \frac{3}{2}F_n$,
	this easily implies that  for $\epsilon_1>0$ small enough to satisfy the above, if $n_1 \in \NN$ is sufficiently big, 
	we can choose $x$ randomly according to $\mu$ and with high  probability the choice of $\t x= S^{n_1}(x)$,  $\t x^{(0)}= S^{3n_1}(x)$ and $\t x^{(1)}=S^{5n_1}(x)$ will satisfy the conclusion of the lemma.
\end{proof}

The result below provides us with a sequence of markers whose orbits avoid a given open set. 
This unavoidably uses an extra property in addition to  positive entropy, in this case non-uniform specification. Indeed, a system satisfying the conclusion of this lemma cannot be minimal.
In this lemma we still assume $d=1$, but only so we can apply Lemma \ref{lem:marker_triple}. 
\begin{lem}\label{lem:marker_patterns}
	Suppose $(X,S)$ has non-uniform specification. Then  there exists $\epsilon_1 >0$, $n_1\in \N$ and $\t x \in X$ so that for every $\delta >0$ and every sufficiently big  $n \in \NN$ there exists $x^{(n)} \in X$ with the following properties:
	\begin{equation}\label{eq:t_x_far_from_x_n}
	d_X^{F_{n_1}}\left(S^{\mi}(x^{(n)}),\t x\right) > \epsilon_1 \mbox{ for all } \mi \in   \Z,
	\end{equation}
	\begin{equation}\label{eq:x_n_no_short_periods}
	d_X^{(\mj +F_n) \cap (\mi + F_n)}\left(S^{-\mi}(x^{(n)}),S^{-\mj}(x^{(n)})\right) > \epsilon_1 \mbox{ whenever }  \mi-\mj  \in  (2-\delta)F_n \mbox{ and } \mi \ne \mj.
	\end{equation}
	
\end{lem}

\begin{proof}
It is enough to prove for the case $\delta<1$. Let $\epsilon_1>0$, $n_1 \in \NN$ and $\t x^{(0)},\t x^{(1)}, \t x \in X$ satisfy \eqref{eq:x_0_x_1_no_overlap} , \eqref{eq:x_0_1_no_self_overlap} and \eqref{eq:x_0_1_no_overlap_t_x},  as in Lemma \ref{lem:marker_triple}, and in addition, assume that $n_1 \in \NN$ is sufficiently big so that $g_{n_1}(\frac{\epsilon_1}{2}) < \frac{1}{2}\delta$,  where $g_{n_1}:(0,1) \to (0,+\infty)$ is the function that appears in the definition of non-uniform specification in Subsection \ref{subsec:non_uniform_spec}.
                Choose $n$ sufficiently big so that  ${ \delta}n > 3n_1$. Now let $ w \in \{0,1\}^{\Z/3n\Z}$ be chosen uniformly at random.
	Let $L= \lceil  ({2}+2g_{n_1}(\frac{\epsilon_1}{2})) n_1 \rceil$.  
	By almost weak specification, there exists $x^{(n)} \in X$ so that
	$$ d_X^{L\mi+ F_{n_1}}(x^{(n)},{S^{-L\mi}}(\t x^{(0)})) < \epsilon_1/2 \mbox{ if } w(\mi \mod 3n\Z) =0$$
	and
	$$ d_X^{L\mi+ F_{n_1}}(x^{(n)},{S^{-L\mi}}(\t x^{(1)})) < \epsilon_1/2 \mbox{ if } w(\mi \mod 3n\Z) =1.$$
	Then because $L < 2.5 n_1$, the properties \eqref{eq:x_0_x_1_no_overlap}  and \eqref{eq:x_0_1_no_self_overlap}  of   $\t x^{(0)},\t x^{(1)}$ ensure that 
	$$d_X^{(\mj +F_n) \cap (\mi + F_n)}\left(S^{-\mi}(x^{(n)}),S^{-\mj}(x^{(n)})\right) > \epsilon_1\text{ whenever }\mi - \mj \in {(2-\delta)F_n}\setminus L\Z,$$ and the property \eqref{eq:x_0_1_no_overlap_t_x} ensures
	that \eqref{eq:t_x_far_from_x_n} holds.
	If $\mi - \mj \in {(2-\delta)F_n}\cap L\Z$ then still, 
	$$d_X^{(\mj +F_n) \cap (\mi + F_n)}\left(S^{-\mi}(x^{(n)}),S^{-\mj}(x^{(n)})\right) > \epsilon_1$$
	 unless 
	$$w\left(\frac{\vec{k}-\mi}{L} \mod 3n\Z\right)=w\left(\frac{\vec{k}-\mj}{L} \mod 3n\Z\right) $$
for all $\vec{k}$ such that  $\vec{k} +F_{n_1} \subset (\mi +F_n) \cap (\mj + F_n)$, and $\vec{k} -\mi \in L\Z$.

	For each $\mi \ne\mj $ the probability of this event is less than 
	$$2^{-\frac{1}{2L}|(\mi + F_n) \cap (\mj + F_n)|}$$
	for large enough $n$. Since $\mi - \mj \in (2-\delta)F_n$ this probability is bounded above by 
	$$2^{-\frac{\delta}{2L}n}$$ whenever $n$ is large enough. Thus the probability of this event occurring for at least one such choice of $\mi, \mj$ is bounded above by
	$${|F_n|^{2}} \cdot 2^{-\frac{\delta}{2L}n},$$
	which goes to zero as $n$ tends to infinity. Thus there exists $ w \in \{0,1\}^{\Z/3n\Z}$  so that the corresponding $x^{(n)}$ satisfies \eqref{eq:x_n_no_short_periods}.
\end{proof}

\begin{proof}[Proof of Proposition \ref{prop:specification_implies_flexible}]
	
	For each $n \in \NN$, let  $g_n:(0,1) \to (0,+\infty)$ be a function that witnesses the  non-uniform specification of $(X,S)$ as in the definition above. { By modifying the functions $g_n:(0,1) \to (0,+\infty)$,  we can assume without loss of generality that  $g_n:(0,1) \to (0,+\infty)$ is a decreasing function for every $n \in \NN$, and that for every fixed $\epsilon >0$ the sequence $(g_n(\epsilon) )_{n=1}^\infty$ {decreases} $0$, but slowly enough so that
		$$
		\lim_{ n \to \infty} n g_n(\epsilon) = \infty.
		$$}
	We can thus find a sequence $(\epsilon_n)_{n=1}^\infty$ of positive numbers decreasing to zero at a slow enough rate such that $g_n(\epsilon_n/8)$ tends to zero, $\lim_{n \to \infty} n g_n(\epsilon_n/8)= +\infty$ and
	\begin{equation}\label{eq:sep_epsilon_n_to_h_X_S}
	\lim_{n\to \infty }\frac{1}{|F_n|}\log \sep_{\epsilon_n}(X, F_n)= h(X, S)
	\end{equation}
	and $\epsilon_1$ is small enough such that Lemma \ref{lem:marker_patterns} applies to it.
	
	Define the sequence $(\delta_n)_{N=1}^\infty$ by $\delta_n= 10 \cdot g_n(\epsilon_n/8)$.
	For $n \in \NN$, let $E_n = F_n\setminus (1-\frac{1}{8}\delta_n )F_n$. 
	Apply Lemma \ref{lem:marker_patterns} to find  $n_1 \in \NN$,  $\t x \in X$ such that for $n$ large enough by almost weak specification there exists $x^{(n)} \in X$ that satisfies  \eqref{eq:t_x_far_from_x_n} and 	
	\begin{equation}\label{eq:x_n_no_short_periods1}
	d_X^{E_n \cap (i + E_n)}\left(S^{-i}(x^{(n)}),x^{(n)}\right) > \epsilon_1 \mbox{ whenever }  i  \in \frac{1}{10}\delta_n F_n\setminus\{0\}.
	\end{equation}
	For $k\gg n_1$ let
	$$\tilde \epsilon_k:=\max\{\epsilon_{k_1}+ \epsilon_{k_2}+ \ldots + \epsilon_{k_r}~:~ n_1\ll k_1\ll k_2\ll \ldots k_r\ll k\}$$
	and 
	$$\tilde \epsilon_\infty:=\max\{\epsilon_{k_1}+ \epsilon_{k_2}+ \ldots + \epsilon_{k_r}~:~ n_1\ll k_1\ll k_2\ll \ldots k_r\}.$$
	By \eqref{eq:n_ll_m} it follows that $\tilde \epsilon_k<  \tilde \epsilon_\infty<\epsilon_{n_1}$.
	
	Further for $n\gg n_1$ for which $\delta_n<1/8$, let $X_n \subset X$ consist of all those $x \in X$ that satisfy the following properties:
	\begin{equation}\label{eq:X_n_decorated_on_corners}
	d_X^{E_n}(x,x^{(n)})< \frac{1}{4}\epsilon_{1},
	\end{equation}
	and 
	\begin{equation}\label{eq:X_n_t_x_appears_densely}
	\forall i \in (1- \frac17\delta_n)F_n~ \exists j+F_{n_1}\subset i + \frac{1}{64} \delta_n F_n \mbox { s.t } d_X^{ F_{n_1} }(S^j(x), \t x)<\frac{1}{8}\epsilon_1+  \tilde \epsilon_n
	\end{equation}
	and empty otherwise. The non-uniform specification property easily implies that $\bigcup_{n=1}^\infty X_n$ is dense in $X$,
	and as in the proof of Lemma \ref{lem:C_m_lower_bound} because $\delta_n n \to \infty$ as $n \to \infty$ it follows that for any $\epsilon>0$,
	$$ \liminf_{n \to \infty} \frac{\log \sep_{\epsilon_n}(X_n,F_n)}{\log \sep_\epsilon(X,F_n)} \ge 1.$$
	Together with \eqref{eq:sep_epsilon_n_to_h_X_S} this implies
	\begin{equation}\label{eq:sep_X_n_big}
	\lim_{n \to \infty}\frac{1}{|F_n|} \log \sep_{\epsilon_n} (X_n,F_n) = h(X,S).
	\end{equation}
To check the marker property we will verify that whenever $x,x' \in X_n$ and $i \in (2-2\delta_n)F_n \setminus \{0\}$ then 
	\begin{equation}\label{eq:marker_X_n}
	d_X^{F_n \cap (i + F_n)}\left(S^{-i}(x),x'\right) > \frac{1}{8}\epsilon_1.
	\end{equation}

	There are two cases to check: For  $i \in \frac{1}{10}\delta_n F_n \setminus \{0\}$,  use the fact that the orbits of both $x$ and $x'$ are $\frac{1}{4}\epsilon_1$-close to $x^{(n)}$ on $E_n$  in the sense of \eqref{eq:X_n_decorated_on_corners}. Along with \eqref{eq:x_n_no_short_periods1}, the marker property in this case follows by the triangle inequality.
	For  $i \in (2-2\delta_n)F_n \setminus \frac{1}{10} \delta_n F_n$, first note that there exists $i'\in (1- \frac{1}{7} \delta_n)F_n $ such that 
	$$i'+ \frac1{64} F_{n_1}\subset (i+E_n)\cap F_n.$$ 
	In this case \eqref{eq:marker_X_n}  follows by \eqref{eq:X_n_t_x_appears_densely}, \eqref{eq:X_n_decorated_on_corners} and \eqref{eq:t_x_far_from_x_n} and the triangle inequality.
	
	Fix $\epsilon_\star< \frac18\epsilon_1$ and let $C_n \subset X_n$ be an $(\epsilon_n,F_n)$-separated set of maximal cardinality. As proved above, we know that it is a marker sequence. We claim that $\C = (C_n)_{n=1}^\infty$ it is a flexible sequence for $(X,S)$ with respect to the sequences $(\epsilon_n)_{n=1}^\infty$ and $(\delta_n)_{n=1}^\infty$. 
	Indeed, suppose  $k \ll n$, that $K \subset (1-\delta_n)F_n$ is $(1+\delta_k)F_k$-spaced and that $W \in C_k^K$.
	
	Let $V \subset \ZZ$ be a maximal $(1+g_{k}(\epsilon_k/8))F_{k}$-spaced subset so that
	$$V + (1+g_{k}(\epsilon_k/8))F_{k}\subseteq (1-\frac{1}{7}\delta_n)F_n \setminus  \left(K+ (1+\delta_k)F_k\right).$$
	
	Then the non-uniform specification property implies that there exists $x'\in X$ such that
	$$d_X^{i +F_{k}}(x', S^{-i}(W_i) )<\epsilon_{k}/8\mbox{ for every } i \in K,$$
	$$d_X^{E_n}(x',x^{(n)}) < \epsilon_k/8 \text{ and }$$
	$$d_X^{i+F_{k}}(x',S^{-i}(\tilde x) )< \epsilon_{k}/8  \mbox{ for every } i \in V.$$
	
	Let us check that $x'\in X_n$. The second equation directly implies that \eqref{eq:X_n_decorated_on_corners} holds. So we have that check that \eqref{eq:X_n_t_x_appears_densely} holds as well.  Let $i \in (1- \frac17\delta_n)F_n$.
	
	There are two possibilities to consider. In the first case it might so happen that for some $k \in K$ we have that $k + (1+\delta_k) F_k\subset i + \frac{1}{64}\delta_nF_n$.
	Then there exists $j$ so that $j+F_{n_1}\subset i + (1+\delta_k) F_k$ such that
	$$d_X^{j+F_{n_1}}(x', S^{-j}\tilde x)<1/8 \epsilon_1+  \tilde \epsilon_{k}+ \epsilon_k/8<1/8 \epsilon_1+  \tilde \epsilon_n.$$    

	We are left with the case when for all $i' \in K$	, $i'+ (1+\delta_k) F_k$ is not contained in $ i + \frac{1}{64}\delta_nF_n$. In this case, because $k \ll n$ it follows that  $i + \frac{1}{64}\delta_nF_n$ contains a translate of 
	$2(1+g_k(\epsilon_k/8))F_k$ which is disjoint from $K+(1+\delta_k)F_k$.
	
	By the maximality property of $V$, there exists $j\in V$ such that $j+(1+g_{k}(\epsilon_k/8))F_{k}\subset i + \frac{1}{64}\delta_nF_n$ which proves \eqref{eq:X_n_t_x_appears_densely}.
	
	Now since $C_n$ is a maximal  $(\epsilon_n, F_n)$-separated subset of $X_n$, it is $(\epsilon_n,F_n)$-dense in $X_n$. This means that  there exists $x\in C_n$ such that $d_X^{F_n}(x, x')<\epsilon_n$. Then $x$  shadows $W$ (because $\epsilon_n < \frac{1}{10^6}\epsilon_k$).
	By \eqref{eq:sep_X_n_big} it follows that $h(\C)=h(X,S)$. 
	Also, because every $C_n$  is $(\epsilon_n,F_n)$-dense in $X_n$ and $\bigcup_{n=1}^\infty X_n$ is dense in $X$, it follows that $\bigcup_{n=1}^\infty C_n$ is dense in $X$.
\end{proof}

\section{Ergodic universality via approximate embeddings}\label{section:ergodic_universal}
In this section we apply the tools introduced in the previous sections to prove  full ergodic universality for systems admitting a flexible sequence, Proposition \ref{prop:full_universality}, which is a partial result towards Theorem \ref{thm:spec_sequence_implies_univesality}. 
Let us list our notational conventions and  standing assumptions:
\begin{itemize}
	\item  $\Y = (Y,T)$ is a free Borel $\ZD$ dynamical system.
	\item $(\epsilon_n)_{n=1}^\infty$ and $(\delta_n)_{n=1}^\infty$ are decreasing sequences of positive numbers tending to $0$, $\epsilon_\star>0$ so that $\lim_{ n \to \infty} \delta_n  n  = \infty$, and $\epsilon_n < \frac{1}{2}\epsilon_\star$ for every $n$. 
	\item For every $n \in \NN$, $Z_n \subset Y$ is the base of an $((1+\delta_n)F_n,\epsilon_n)$-tower. 
	\item There is a sequence of finite measurable partitions $(\mathcal{P}_k)_{k=1}^\infty$ that together generate the $\sigma$-algebra on $Y$, so that $\mathcal{P}_k \prec \mathcal{P}_{k+1}$.
	
	\item $(X,S)$ is a topological $\ZD$ dynamical system, and $d_X:X \times X \to \mathbb{R}_+$ is a compatible metric on $X$.
	\item $\X=(X^{\ZD},S)$ is the space of approximate orbits for $(X,S)$. 
	\item $\C= (C_n)_{n=1}^\infty \in (2^X)^{\NN}$ is a flexible marker sequence for $(X,S)$ with respect to $(\epsilon_n)_{n=1}^\infty$ and $(\delta_n)_{n=1}^\infty$, in the sense that it satisfies \eqref{eq:marker_property}.
	\item 
	We fix an element of $X$ and denote it by  $x_\star \in X$.
	\item Recall that the notation $n \ll m$ intuitively means that $n$ is ``much smaller than m'' and is formally defined by \eqref{eq:n_ll_m}.

\end{itemize}

As in most of the proofs for ergodic universality, the idea is to construct an embedding of $(Y,T,\mu)$  into $(X,S)$ as a limit of ``approximate embeddings'' of some sort. However, the target for our  ``approximate  embeddings'' is not $(X,S)$ itself, but rather  $\X=(X^{\ZD},S)$, which as we explained can be viewed as  ``the space of approximate orbits of (X,S)''. Also, in this section we fix $\mu \in \Prob_e(Y,T)$ with $h_\mu(Y,T) < h(\C)$.  Later on, when we prove universality in the ``almost Borel'' category we will consider $\mu \in \Prob_e(Y,T)$ as a ``variable'' and pay closer attention to the manner in which other parameters depend on $\mu$.

Let us introduce a bit more notation and definitions:
\begin{defn}
	For $F \subseteq \ZD$ and $w \in X^F$ let
	\begin{equation}
	\index{Definitions and notation introduced in Section 6!$[w]$}[w] := \left\{ w' \in X^{\ZD}:~w'|_F =w \right\}.
	\end{equation}
	Also, for $x \in X$ let
	\begin{equation}
	\index{Definitions and notation introduced in Section 6!$[x]_F$}[x]_F :=  \left\{ w' \in X^{\ZD}:~ w'_{\mi} =S^{\mi}(x) ~\forall \mi \in F \right\}.
	\end{equation}
\end{defn}

\begin{defn}
	Given $\rho \in \Mor(\Y,\X)$ we define the following Borel partition of $Y$:
	\begin{equation}
	\index{Definitions and notation introduced in Section 6!$\mathcal{P}_\rho$}\mathcal{P}_\rho := \left\{\rho^{-1} [\rho(y)]_{\{\vec{0}\}} :~ y \in Y\right\}.
	\end{equation}
\end{defn}

\begin{defn}
We say that  $\rho \in \Mor(\Y,\X)$  is a \index{Definitions and notation introduced in Section 6!symbolic morphism}symbolic morphism if $\rho(y)_0$ takes finitely many values as $y$ ranges over $Y$.
\end{defn}
Whenever $\rho \in \Mor(\Y,\X)$  is a symbolic morphism then $\mathcal{P}_\rho$ is a finite measurable partition. If  $\rho \in \Mor(\Y,\X)$  is a symbolic morphism it follows that the closure of $\rho(Y)$ in $X^{\ZD}$ is a zero dimensional compact metrizable space, this is our reason for the term ``symbolic morphism''.

Recall that $x_\star \in X$ is a fixed element of $X$. Here is what we mean by an ``approximate embedding'':

\begin{defn}
We say that $\rho \in \Mor(\Y,\X)$ is \emph{$n$-towerable} \index{Definitions and notation introduced in Section 6!$n$-towerable}if
   \begin{equation}\label{eq:rho_y_traces_x}
\forall y \in Z_n, ~\exists x \in C_n \mbox{ s.t. } \rho(y)_{\mi}= S^{\mi}(x) ~ \mbox{ for all } \mi \in F_{ n}
\end{equation}
	and
\begin{equation}
\rho(y)_{0} = x_\star~\forall y \in Y \setminus T^{F_n}Z_n.
\end{equation}
Fix $\mu \in \Prob_e(Y,T)$ and integers $k,n \in \NN$ and $\epsilon >0$.
A $(k,n,\epsilon,\mu)$-approximate embedding\index{Definitions and notation introduced in Section 6!$(k,n,\epsilon,\mu)$-approximate embedding} is an $n$-towerable map  $\rho \in \Mor(\Y,\X)$ such that there exists a Borel set $Z[\rho]\subset Z_n$\index{Definitions and notation introduced in Section 6!$Z[\rho]$} satisfying
	\begin{equation}\label{eq:Z_rho_1}
	\mu(Z_n \setminus Z[\rho]) \le \epsilon\mu(Z_n)
	\end{equation}
	and
	\begin{equation}\label{eq:rho_almost_injective}
	\forall y,y'\in Z[\rho] \mbox{ if }  \rho(y)_{\m 0} = \rho(y')_{\m 0} \mbox{ then } \mathcal{P}_k^{F_n}(y) = \mathcal{P}_k^{F_n}(y').
	\end{equation}
	Thus there exists a map \index{Definitions and notation introduced in Section 6!$\Psi_{k,n}$}$\Psi_{k,n}: X\to \P_k^{F_n}$ such that
\begin{equation}\label{eq:Psi_n_k_def}
\Psi_{k,n}(\rho(y)_{\m0})=\P_k^{F_n}(y)\text{ for all $y\in Z[\rho]$.}
\end{equation}
\end{defn}

\begin{remark}
	An  $n$-towerable map  $\rho \in \Mor(\Y,\X)$ is in particular a symbolic morphism because  $\rho(y)_{\m 0}$ takes values only in $\bigcup_{\mi \in F_n}T^{\mi}C_n \cup \{x_\star\}$, which is a finite set. 
\end{remark}
\begin{remark}\label{rem:approximate_emb_monotone}
	Because $\mathcal{P}_{j} \prec \mathcal{P}_{j+1}$ for all $j \in\NN$, whenever $\rho \in \Mor(\Y,\X)$ is a $(k,n,\epsilon,\mu)$-approximate embedding then it is also a $(k_0,n,\epsilon,\mu)$-approximate embedding for every $k_0 < k$.
\end{remark}
	
Given  a Borel function $\tilde \Phi: Y \to X$ and $n \in \NN$ we define \index{Definitions and notation introduced in Section 6!$\rho_{\tilde \Phi,n} $}$\rho_{\tilde \Phi,n} \in\Mor(\Y,\X)$  by
\begin{equation}\label{eq:rho_Phi_n_def}
\rho_{\tilde \Phi,n}(y)_\mi :=
\begin{cases}
S^{\mi-\mj}\left(\tilde \Phi(T^{\mj}(y))\right) & \mbox{ if } \mi - \mj \in F_n \mbox{ and } y \in T^{-\mj}Z_n\\
x_\star & \mbox{ if } T^{\mi}(y) \notin T^{F_n}Z_n.
\end{cases}
\end{equation}
Because $\{T^{-\mj}Z_n\}_{\mj \in F_n}$ are pairwise disjoint, $\rho_{\tilde \Phi,n} \in \Mor(\Y,\X)$ is well defined. Furthermore,  if $\tilde \Phi(Z_n) \subset C_n$ then
$\rho= \rho_{\tilde \Phi,n}$ satisfies  \eqref{eq:rho_y_traces_x}.

The following basic lemma asserts that $(k,n,\epsilon,\mu)$-approximate embeddings exist, as long as $n$ is sufficiently big in terms of the other parameters. 
\begin{lem}\label{lem:approximate_embedding_exists}
	For every $k \in\NN$ and $\epsilon >0$ there exists $N$ such that for every $n >N$ there exists a  $(k,n,\epsilon,\mu)$-approximate embedding.
\end{lem}

\begin{proof}

	Apply  Lemma \ref{lem:tower_cov} with $\mathcal{P}=\mathcal{P}_k$. It follows that for sufficiently large $n$ there exists a subset $\mathcal{G} \subset \mathcal{P}_k^{F_n}$ such that
	\begin{equation}\label{eq:mu_G_big}
	\mu(Z_n \cap \bigcup\mathcal{G}) \ge (1-\epsilon)\mu(Z_n)
	\end{equation}
	and
	\begin{equation}
	|\mathcal{G}| \le \exp\left(|F_n|\frac{1}{2}(h_\mu(Y,T)+h(\C))\right).
	\end{equation}
	Recall that $h_\mu(Y,T) < h(\C)$,
	so 
	if $n$ is sufficiently big then $|C_n| > \exp\left(|F_n| \frac{1}{2}(h_\mu(Y,T)+h(\C))\right)$. This implies that
	\begin{equation}\label{eq:size_G_small}
	|\mathcal{G}| < |C_n|.
	\end{equation}
	Let $\Phi:\mathcal{P}_k^{F_n} \to C_n$ be a function such that the restriction to $\mathcal{G}$ is  injective and
	$\Phi(\mathcal{G}) \cap \Phi(\mathcal{P}_k^{F_n} \setminus \mathcal{G}) =\emptyset$, and let
	$\tilde \Phi : Y \to C_n$ be given by
	\begin{equation}
	\tilde \Phi(y) = \Phi(\mathcal{P}_k^{F_n}(y)).
	\end{equation}
	Let $\rho = \rho_{\tilde \Phi,n} \in \Mor(\Y,\X)$  be given by \eqref{eq:rho_Phi_n_def}. Then $\t\Phi(Z_n) \subset C_n$ so  \eqref{eq:rho_y_traces_x} is satisfied. Because $\Phi\mid_\mathcal{G}$ is injective, if we set $Z[\rho] = \bigcup \mathcal{G}$ then  \eqref{eq:rho_almost_injective} will also be satisfied. This shows that the map $\rho \in \Mor(\Y,\X)$ is indeed a $(k,n,\epsilon,\mu)$-approximate embedding.
	
\end{proof}	
	
\begin{defn}
For $\rho,\tilde \rho \in \Mor(\Y,\X)$, and $\epsilon >0$, $n \in \NN$ let 
			\begin{equation}\label{eq:D_n_epsilon_rho_def}
		\index{Definitions and notation introduced in Section 6!$D_{n,\epsilon}[\rho,\tilde \rho]$}	D_{n,\epsilon}[\rho,\tilde \rho] := \left\{y \in Z_n:~ d_X^{F_{n}}(\rho(y),\tilde \rho(y)) \ge \epsilon \right\}.
			\end{equation}	
\end{defn}

For future reference, we write the following formula, which is a direct consequence of the definition: 
\begin{equation}\label{eq:D_n_n_0}
\forall n>n_0~ \mbox{ and } \rho,\tilde \rho \in \Mor(\Y,\X),~D_{n_0,\epsilon}[\rho,\tilde \rho] \subseteq \left(Y \setminus T^{F_{n-n_0}}Z_n\right) \cup T^{F_{n-n_0}}(D_{n,\epsilon}[\rho,\tilde \rho] ).
\end{equation}
Also, for future reference we write the following simple consequence of the triangle inequality:
\begin{equation}\label{eq:D_triangle}
D_{n,\epsilon_1 + \epsilon_2}[\rho_a,\rho_c] \subset D_{n,\epsilon_1}[\rho_a,\rho_b] \cup D_{n,\epsilon_2}[ \rho_b,\rho_c].
\end{equation}

The following lemma asserts that every $(k,n,\epsilon,\mu)$-approximate embedding $\rho$ admits a ``stable approximate inverse'', in the sense that ``for most'' $y \in Y$  it is possible to recover $\mathcal{P}_k(y)$ from  a ``sufficiently good approximation'' of $\rho(y)$. The map  $\tilde \rho \in \Mor(\Y,\X)$ in the statement of the lemma below plays the role of this  ``sufficiently good approximation''.  
\begin{lem}\label{lem:approximate_inverese}
	Let   $\rho \in \Mor(\Y,\X)$ be a $(k,n_0,\epsilon,\mu)$-approximate embedding, and suppose that $\tilde \rho \in \Mor(\Y,\X)$.
	Then for every $y \in  T^{(1-2\delta_{n_0})F_{n_0}}\left(Z[\rho] \setminus D_{n_0, \frac{1}{2}\epsilon_{n_0}}[\rho,\tilde \rho]\right)$,	
	the value of	$\mathcal{P}_{k}(y)$ is determined by $\t \rho (y)\mid_{F_{2n_0}}$.
	\end{lem}

\begin{proof}
Let $\rho$  be a $(k,n_0, \epsilon,\mu)$-approximate embedding and $\tilde \rho \in \Mor(\Y,\X)$. We will show that there exists a Borel function $\Phi_{k,n_0}: X^{F_{2n_0}} \to \mathcal{P}_k$ so that  for every $\tilde \rho \in \Mor(\Y,\X)$ the following holds:
	\begin{equation}\label{eq:Phi_k_n_recover_stable}
	\Phi_{k,n_0}(\tilde \rho(y)\mid_{F_{2n_0}})= \mathcal{P}_k(y) \mbox{ for every } y\in T^{(1-2\delta_{n_0})F_{n_0}}\left(Z[\rho] \setminus D_{n_0, \frac{1}{2}\epsilon_{n_0}}[\rho,\tilde \rho]\right).
	\end{equation}
	Define  a function  ${I}_{n_0}:X^{F_{2n_0}} \to  F_{n_0}$ as follows:
	\begin{equation}
	\index{Definitions and notation introduced in Section 6!${I}_{n_0}(w)$}{I}_{n_0}(w) :=
	\min\left\{ \mi \in F_{n_0}:~ \exists x \in C_{n_0} \mbox{ s.t. } d_X^{F_{n_0}}(S^{\mi}(w),x) <  \epsilon_{n_0}/2 \right\}.
	\end{equation}
	The minimum in the definition of $I_{n_0}$ above is with respect to some fixed total order on $F_{n_0}$. If the set is empty we arbitrarily define $I_{n_0}(w)=\vec{0}$.
	Suppose $\mi \in (1-2\delta_{n_0})F_{n_0}$ and $T^{\mi}(y) \in Z_{n_0}$ then  $d_X^{F_{n_0}}(S^{\mi}(\rho(y)),C_{n_0})=0$. By the marker property \eqref{eq:marker_property} it follows that in this case $I_{n_0}(\rho(y))=\mi$. 
	We conclude that 
	\begin{equation}\label{eq:I_n_to_Z_n}
	T^{I_{n_0}(\rho(y))}(y) \in Z_{n_0}~ \forall ~ y \in T^{(1-2\delta_{n_0})F_{n_0}} Z_{n_0}.
	\end{equation}
	Furthermore, if $T^{\mi}(y) \in Z_{n_0}\setminus D_{n_0,\frac{1}{2}\epsilon_{n_0}}[\rho,\tilde \rho]$ then the same consideration using the marker property \eqref{eq:marker_property} also implies that $I_{n_0}(\tilde \rho(y))=\mi$.
	We conclude that
	\begin{equation}\label{eq:I_n_rho_tilde_rho}
	I_{n_0}(\tilde \rho(y))=I_{n_0}(\rho(y)) ~\forall ~y \in T^{(1-2\delta_{n_0})F_{n_0}} (Z_{n_0}\setminus D_{n_0, \epsilon_{n_0}/2}[\rho,\tilde \rho]).
	\end{equation}
	
	Let $\Psi_{k,n_0}:X \to \mathcal{P}_k^{F_{n_0}}$ satisfy \eqref{eq:Psi_n_k_def}. Define $\Phi_{k,n_0}:X^{F_{2n_0}} \to \mathcal{P}_k$ as follows:
	\begin{equation}
	\Phi_{k,n_0}(w)= \left(\Psi_{k,n_0}(\pi_{n_0}(S^{I_{n_0}(w)}(w)|_{F_{n_0}} ) )\right)_{-I_{n_0}(w)},
	\end{equation}
	where \index{Definitions and notation introduced in Section 6!$\pi_n$}$\pi_n:X^{F_n} \to C_n$ is a Borel function satisfying
	$$
	d_X^{F_n}(\pi_n(w),w) = \min_{x \in C_n}d_X^{F_n}(x,w) ~\forall ~ w \in X^{F_n}.
	$$
	By \eqref{eq:Psi_n_k_def} and \eqref{eq:I_n_to_Z_n} we have that $\Phi_{k,{n_0}}(\rho(y)\mid_{F_{2n_0}}) = \mathcal{P}_k(y)$ whenever $y \in T^{(1-2\delta_{n_0})F_{n_0}}Z[\rho]$.
	Also, if $ T^{\mi}(y) \in Z[\rho]\setminus D_{n_0,\frac{1}{2}\epsilon_{n_0}}[\rho,\tilde \rho]$ for some $\mi \in (1-2\delta_{n_0})F_{n_0}$ then $\Phi_{k,{n_0}}(\rho(y))= \Phi_{k,{n_0}}(\tilde \rho(y))$.
	We conclude that \eqref{eq:Phi_k_n_recover_stable} holds.

\end{proof}

\begin{lem}\label{lem:D_n_epsilon_D_m_epsilon}
	Suppose $1 \ll n \ll m$ and $\rho,\tilde \rho \in \Mor(\Y,\X)$ then:
	\begin{equation}\label{eq:D_n_epsilon_D_m_epsilon}
\mu\left(
	D_{n,\epsilon}[\rho,\tilde \rho] \mid Z_n
	\right) \le 2\epsilon_m + 8d \delta_m +
2 \mu\left(D_{m,\epsilon}[\rho,\t \rho] \mid Z_m \right).
	\end{equation}
\end{lem}
\begin{proof}
				Let 
		\begin{equation*}\label{eq:D_n_epsilon_rho_def1}
		\t D_{2n,\epsilon}[\rho,\tilde \rho] := \left\{y \in Y:~ d_X^{F_{2n}}(\rho(y),\tilde \rho(y)) \ge \epsilon \right\}.
		\end{equation*}

	By the law of total probability
	$$
	\mu\left(\t D_{2n,\epsilon}[\rho,\t \rho]\right) \ge
	\sum_{\mi \in F_n}\mu\left(\t D_{2n,\epsilon}[\rho,\t \rho] \mid T^{\mi}Z_n \right) \mu\left(T^{\mi}Z_n\right).
	$$
	Note that $T^{\mi}D_{n,\epsilon}[\rho,\t \rho] \subseteq \t D_{2n,\epsilon}[\rho,\t \rho]$ for every
	$\mi \in F_n$ .
	So 
	$$
	\mu\left(\t D_{2n,\epsilon}[\rho,\t \rho]\right) 
	\ge \mu(Z_n)\sum_{\mi \in F_n}\mu\left(T^{\mi}D_{n,\epsilon}[\rho,\tilde \rho] \mid T^{\mi}Z_n \right) =
	|F_n|\mu(Z_n) \mu(D_{n,\epsilon}[\rho,\t \rho] \mid Z_n).
	$$
	Now because $n \gg 1$ it follows that $|F_n| \mu(Z_n) > \frac{1}{2}$,
	so
	$$
	 \mu(D_{n,\epsilon}[\rho,\t \rho] \mid Z_n) <  2\mu\left(\t D_{2n,\epsilon}[\rho,\t \rho]\right).
	$$

Again by the law of total probability,
$$\mu\left(\t D_{2n,\epsilon}[\rho,\t \rho] \right) \le \epsilon_m + 4d \delta_m + \sum_{\mi \in (1-\delta_m)F_m} \mu(\t D_{2n,\epsilon}[\rho,\t \rho] \mid T^{\mi}Z_m)\mu(T^{\mi}Z_m).$$
Now if $n \ll m$, for every $\mi \in (1-\delta_m)F_m$ we have
$\t D_{2n,\epsilon}[\rho,\t \rho] \cap T^{\mi}Z_m \subseteq T^{\mi}D_{m,\epsilon}[\rho,\t \rho]$.
It follows that
$$\mu\left(\t D_{2n,\epsilon}[\rho,\t \rho] \right) \le \epsilon_m + 4d \delta_m+  \sum_{\mi \in (1-\delta_m)F_m} \mu( T^{\mi}D_{m,\epsilon}[\rho,\t \rho] \mid T^{\mi}Z_m)\mu(T^{\mi}Z_m)\le $$
$$\le \epsilon_m + 4d \delta_m +\mu\left(D_{m,\epsilon}[\rho,\t \rho] \mid Z_m \right).$$
The inequality \eqref{eq:D_n_epsilon_D_m_epsilon} follows.
\end{proof}

The following lemma roughly says that if $\rho$ is a $(k_0,n_0,\delta,\mu)$-approximate embedding with $\delta>0$ sufficiently small and $n_0 \in \NN$ sufficiently big, then the log of the approximate covering number of  $\mathcal{P}_{k_0}^{F_n}$ relative to $\mathcal{P}_{\rho}^{F_n}$ on $Z_n$ is a very small fraction of  $|F_n|$, provided $n$ is big enough.
\begin{lem}\label{lem:COV_rel_to_approximate_embedding}
	For any $\eta >0$ and $k_0 \in\NN$ there exist $N_0 \in \NN$ such that for any $(k_0,n_0, \frac{1}{N_0},\mu)$-approximate embedding $\rho \in \Mor(\Y,\X)$ with $n_0 \ge N_0$,  $\gamma \in [0,1)$ and  $\delta >0$ there	
	 exists $N \in \NN$ so that for every $n >N$
	 \begin{equation}\label{eq:COV_P_k_cond_rho_small}
	 \COV_{\mu,\delta,\mathcal{P}_{k_0}^{F_n \setminus \gamma F_n }\mid \mathcal{P}_{\rho}^{(1-2\delta_n)F_n \setminus \gamma F_n}}(Z_n) < e^{ \eta |F_n|}.
	 \end{equation}
\end{lem}
\begin{proof}
	For  $N_0 \in \NN$, denote
	\begin{equation}\label{eq:t_epsilon_def}
	\t\epsilon_{N_0} = \epsilon_{N_0} +  6d\delta_{N_0}+ \frac{2}{N_0}.
	\end{equation}
	For $\eta \in (0,1)$ and $k_0 \in \NN$, let
	\begin{equation}\label{eq:N_r_def}
	\index{Definitions and notation introduced in Section 6!$\overline{N}_R(\eta,k_0)$}\overline{N}_R(\eta,k_0) = \inf\left\{ N_0 \in \NN:~	{\H( \t \epsilon_{N_0})+(\tilde \epsilon_{N_0}+12d\delta_{N_0})\log|\P_{k_0}| } < \eta\right\}.
	\end{equation}
	Then $\overline{N}_R(\eta,k_0) \in \NN$ is  well defined because $\H(\epsilon) \to 0$ as $\epsilon \to 0+$ and $\tilde \epsilon_{N},\delta_N \to 0$ as $N \to \infty$.
	
	Fix $\eta \in (0,1)$ and $k_0 \in \NN$. Denote $N_0 = \overline{N}_R(\eta,k_0)$. 
	Choose any $n_0 > N_0$, $\gamma \in [0,1)$, any  
	$(k_0,n_0,\frac{1}{N_0},\mu)$-approximate embedding  $\rho \in \Mor(\Y,\X)$ and $\delta  > 0$. 

	Denote
	\begin{equation}\label{eq:G_n_rho_def}
	G_{n_0} = T^{(1-2\delta_{n_0})F_{n_0}}\left(Z[\rho] \right).
	\end{equation}
	
	Recall that $Z_{n_0}$ is the base of an $((1+\delta_{n_0})F_{n_0},\epsilon_{n_0})$-tower.
	Also because $\rho$ is a $(k_0,n_0,\frac{1}{N_0},\mu)$-approximate embedding
	$$  
	\mu(Z_{n_0}\setminus Z[\rho]|Z_{n_0})<\frac{1}{N_0}.$$
	
	Thus by Lemma \ref{lem:mu_tower_subset}
	$$\mu(Y \setminus G_{n_0})   <  \tilde \epsilon_{N_0}.$$ 

	Let 

	\begin{equation}\label{eq:A_N_0_n_n_0_def}
	A_{N_0,n,n_0} = \left\{y \in Z_n:~ \sum_{\mi \in F_n \setminus \gamma F_n} 1_{G_{n_0}}(T^{\mi}(y))>(1- \t \epsilon_{N_0}) |F_n \setminus \gamma F_n |\right\}.
	\end{equation}
	We apply  the mean ergodic theorem  for Rokhlin towers (Proposition \ref{prop:measure_SM_for_towers}) to deduce that 
	\begin{equation*}
	\lim_{n \to \infty}\mu\left(A_{N_0,n,n_0}\mid Z_n\right)=1.
	\end{equation*}
	More specifically, we apply \eqref{eq:Rohklin_tower_mean_ET} with $f= 1_{G_{n_0}}$ twice (with $n$ and $\gamma n$) , taking into account that 
	$\int f d\mu = \mu(G_{n_0})  > 1- \t \epsilon_{N_0}$.
	So there exists $N \gg n_0$ so that for every $n >N$ 
	\begin{equation}\label{eq:mu_A_epsilon_n_big}
	\mu\left(A_{N_0,n,n_0}\mid Z_n\right) > 1- \delta/2.
	\end{equation}
	It is well known and easy to show that for any natural numbers $k <n $ we have
	\begin{equation}
	\frac{1}{n}\log {n \choose k} \le \H(\frac{k}{n}),
	\end{equation}
	where $\H(p)$ is given by \eqref{eq:H_shannon}. Thus,

	\begin{equation}\label{eq:F_n_choose_epsilon_small}
	{|F_n\setminus \gamma F_n| \choose \t \epsilon_{N_0} |F_n\setminus \gamma F_n|} \le e^{\H( \t \epsilon_{N_0} )|F_n\setminus \gamma F_n|}.
	\end{equation}

	Now choose any $n > N$, so that \eqref{eq:mu_A_epsilon_n_big} holds. By Lemma \ref{lem:approximate_inverese}, for every $y \in G_{n_0}$, the value of	$\mathcal{P}_{k_0}(y)$ is determined by $\rho (y)\mid_{F_{2n_0}}$.
	This means that 
	$$\COV_{\mu,0,\mathcal{P}_{k_0} \mid \mathcal{P}_{\rho}^{F_{2n_0}}}(G_{n_0}) =1.$$
	Choose $S \subset F_n \setminus \gamma F_n$. Then (taking into account that $n\gg n_0$) for every $\mi \in S \cap (1-3 \delta_n)F_n\setminus (\gamma+\delta_n)F_n$, 
	$$\COV_{\mu,0,T^{\mi}\mathcal{P}_{k_0}\mid \mathcal{P}_{ \rho}^{(1-2\delta_n)F_n \setminus \gamma F_n}}(\bigcap_{\mj \in S}T^{\mj}G_{n_0})=1.$$
	
	It follows (for instance by applying a degenerate easy case of Lemma \ref{lem:COV_submultiplicative} with $\epsilon=0$) that for any $S \subset F_n \setminus \gamma F_n$
	\begin{equation}\label{eq:cov_S_n}
	\COV_{\mu,0,\mathcal{P}_{k_0}^{(1-2\delta_n)F_n \setminus \gamma F_n}\mid \mathcal{P}_{ \rho}^{(1-2\delta_n)F_n \setminus \gamma F_n}}(\bigcap_{\mi \in S}T^{\mi}G_{n_0}\cap Z_n) \le |\mathcal{P}_{k_0}|^{|(F_n \setminus \gamma F_n) \setminus S|+ |F'_{n,\gamma}|}
	\end{equation}
	where:
	$$
	F'_{n,\gamma} = F_{n} \setminus (1-3\delta_n)F_{n} \cup  ((\gamma+\delta_n) F_n ) \setminus \gamma F_n). 
	$$
	Because 
	$n >N \gg n_0$ it follows that
	$$
	| F'_{n,\gamma}| \le 8 d \delta_n |F_n|.
	$$

	Note that 
	\begin{equation}\label{eq:A_epsilon_bigcup}
	A_{N_0,n,n_0} = \bigcup_{ |S| \ge (1-\tilde \epsilon_{N_0})|F_n \setminus \gamma F_n|}\left( \bigcap_{\mi \in S}T^{\mi}G_{n_0}\cap Z_n\right),
	\end{equation}
	where the union is over all $S \subset F_n \setminus \gamma F_n$ such that $|S| \ge (1-\t \epsilon_{N_0})|F_n \setminus \gamma F_n|$.
	From \eqref{eq:cov_S_n} and \eqref{eq:A_epsilon_bigcup} it follows that
	\begin{equation}\label{eq:A_epsilon_sum}
	\COV_{\mu,0,\mathcal{P}_{k_0}^{(1-2\delta_n)F_n \setminus \gamma F_n}\mid \mathcal{P}_{ \rho}^{(1-2\delta_n)F_n \setminus \gamma F_n}}(A_{N_0,n,n_0}) \le
	\sum_{ |S| \ge (1-\t \epsilon_{N_0})|F_n \setminus \gamma F_n|}|\mathcal{P}_{k_0}|^{|(F_n \setminus \gamma F_n) \setminus S|+ 8d\delta_n|F_n|}.
	\end{equation}
	If $|S| \ge (1-\t \epsilon_{N_0})|F_n \setminus \gamma F_n|$ then $$|\mathcal{P}_{k_0}|^{|(F_n \setminus \gamma F_n) \setminus S|} \le |\mathcal{P}_{k_0}|^{\t \epsilon_{N_0}|F_n|}= e^{\t \epsilon_{N_0}  (\log |\mathcal{P}_{k_0}|) \cdot  |F_n|}.$$
	There are ${|F_n \setminus \gamma F_n| \choose \t \epsilon_{N_0} |F_n \setminus \gamma F_n|}$ summands in the sum in the right hand side of  \eqref{eq:A_epsilon_sum}. 
	Thus 
	$$\COV_{\mu,0,\mathcal{P}_{k_0}^{(1-2\delta_n)F_n \setminus \gamma F_n}\mid \mathcal{P}_{ \rho}^{(1-2\delta_n)F_n \setminus \gamma F_n}}(A_{N_0,n,n_0})  \le {|F_n \setminus \gamma F_n| \choose \t \epsilon_{N_0} |F_n \setminus \gamma F_n|} e^{(\tilde \epsilon_{N_0}+8d\delta_n)|F_n|\log|\P_{k_0}|  }.$$
	By \eqref{eq:F_n_choose_epsilon_small} it follows that
	$$\COV_{\mu,0,\mathcal{P}_{k_0}^{(1-2\delta_n)F_n \setminus \gamma F_n}\mid \mathcal{P}_{\rho}^{(1-2\delta_n)F_n \setminus \gamma F_n}}(A_{N_0,n,n_0})  \le e^{\H( \t \epsilon_{N_0} )|F_n\setminus \gamma F_n|}\cdot e^{  (\tilde \epsilon_{N_0}+8d\delta_n)|F_n|\log|\P_{k_0}| }.$$
	By  \eqref{eq:mu_A_epsilon_n_big},
	$$ \COV_{\mu,\delta/2,\mathcal{P}_{k_0}^{(1-2\delta_n)F_n \setminus \gamma F_n}\mid \mathcal{P}_{ \rho}^{(1-2\delta_n)F_n \setminus \gamma F_n}}(Z_n) \le 
	\COV_{\mu,0,\mathcal{P}_{k_0}^{(1-2\delta_n)F_n \setminus \gamma F_n}\mid \mathcal{P}_{\rho}^{(1-2\delta_n)F_n \setminus \gamma F_n}}(A_{N_0,n,n_0}).$$
	So
	$$
	\COV_{\mu,\delta/2,\mathcal{P}_{k_0}^{(1-2\delta_n)F_n \setminus \gamma F_n}\mid \mathcal{P}_{\rho}^{(1-2\delta_n)F_n \setminus \gamma F_n}}(Z_n) \le
	e^{\H( \t \epsilon_{N_0} )|F_n\setminus \gamma F_n| +(\tilde \epsilon_{N_0}+8d\delta_n)|F_n|\log|\P_{k_0}| } \le
	$$
	$$\le
	e^{\left(\H( \t \epsilon_{N_0})+(\tilde \epsilon_{N_0}+8d\delta_n)\log|\P_{k_0}| \right)|F_n|}.
	$$
	
	Apply  \eqref{eq:COV_join_submultiplicative} of Lemma \ref{lem:COV_submultiplicative}  to deduce that
	$$
	\COV_{\mu,\delta,\mathcal{P}_{k_0}^{F_n \setminus \gamma F_n}\mid \mathcal{P}_{ \rho}^{(1-2\delta_n)F_n \setminus \gamma F_n}}(Z_n) \le$$
	$$\le
	\COV_{\mu,\delta/2,\mathcal{P}_{k_0}^{(1-2\delta_n)F_n \setminus \gamma F_n}\mid \mathcal{P}_{ \rho}^{(1-2\delta_n)F_n \setminus \gamma F_n}}(Z_n)
	\cdot 
	\COV_{\mu,\delta/2,\mathcal{P}_{k_0}^{F_n \setminus (1-2\delta_n)F_n }\mid \mathcal{P}_{\rho}^{(1-2\delta_n)F_n \setminus \gamma F_n}}(Z_n).
	$$
	Clearly,
	$$
	\COV_{\mu,\delta/2,\mathcal{P}_{k_0}^{F_n \setminus (1-2\delta_n)F_n }\mid \mathcal{P}_{ \rho}^{(1-2\delta_n)F_n \setminus \gamma F_n}}(Z_n) \le |\mathcal{P}_{k_0}|^{|F_n \setminus (1-2\delta_n)F_n|},
	$$
	so
	$$ \COV_{\mu,\delta,\mathcal{P}_{k_0}^{F_n \setminus \gamma F_n}\mid \mathcal{P}_{ \rho}^{(1-2\delta_n)F_n \setminus \gamma F_n}}(Z_n) \le \COV_{\mu,\delta/2,\mathcal{P}_{k_0}^{(1-2\delta_n)F_n \setminus \gamma F_n}\mid \mathcal{P}_{ \rho}^{(1-2\delta_n)F_n \setminus \gamma F_n}}(Z_n) \cdot |\mathcal{P}_{k_0}|^{|F_n \setminus (1-2\delta_n)F_n|}\le$$
	$$\le
	e^{{\left(\H( \t \epsilon_{N_0})+(\tilde \epsilon_{N_0}+8d\delta_n)\log|\P_{k_0}| \right)|F_n|}}e^{\log|\mathcal{P}_{k_0}| \cdot |F_n \setminus (1-2\delta_n)F_n|} \le
	$$
	$$ \le 
	e^{\left(\H( \t \epsilon_{N_0})+(\tilde \epsilon_{N_0}+12d\delta_n)\log|\P_{k_0}| \right)|F_n|}.
	$$
	Since $n > n_0 > \overline{N}_R(\eta,k_0)$ it follows that
	$$
	{\left(\H( \t \epsilon_{N_0})+(\tilde \epsilon_{N_0}+12d\delta_n)\log|\P_{k_0}| \right)} < \eta.
	$$
	This proves that \eqref{eq:COV_P_k_cond_rho_small} holds.

\end{proof}
Our next goal is to state certain sufficient  conditions for a sequence of approximate embeddings converge to a proper embedding.
\begin{lem}\label{lem:convergence_set_for_rho_j_s}
	Suppose  $(k_j)_{j=1}^\infty$ and $(n_j)_{j=1}^\infty$ are strictly increasing sequences of natural numbers so that $n_{j} \ll n_{j+1}$ for every $j \ge 1$. Let $(\rho_j)_{j= 1}^\infty \in \Mor(\Y,\X)^\NN$ be a sequence of morphisms such that $\rho_j$ is a $(k_j,n_j,\frac{1}{2^j},\mu)$-approximate embedding for every $j$.
	Let $Y_\infty \subset Y$ be given by:
	\begin{equation}\label{eq:Y_infty_def}
Y_\infty = \bigcap_{\mi \in \ZD}T^{-\mi}\bigcup_{j=1}^\infty \bigcap_{t \ge j}T^{(1-2\delta_{n_t})F_{n_t}}\left(Z[\rho_t] \setminus D_{n_t,\frac{3}{8}\epsilon_{n_t}}[\rho_t, \rho_{t+1}]\right).
	\end{equation}
	Then:
	\begin{enumerate}
		\item[(i)]
		The limit $\rho(y) = \lim_{j \to \infty} \rho_j(y)$ exists for every $y \in Y_\infty$.
		\item[(ii)]
		For every $y \in Y_{\infty}$ there exist $x \in X$ so that $\rho(y)_\mi = S^{\mi}(x)$ for every $\mi \in \ZD$.
		\item[(iii)]
		The function $\rho:Y_\infty \to X^{\ZD}$ is injective. 
	\end{enumerate}
     In other words, on the Borel set $Y_\infty \subset Y$ , the sequence $(\rho_j)_{j=1}^\infty$ converges pointwise to  an equivariant Borel embedding $\rho:Y_\infty \to X$. 
\end{lem}
The constant $\frac{3}{8}$ in front of $\epsilon_{n_t}$ is somewhat arbitrary; the important point is that it is strictly smaller than $\frac{1}{2}$. Have a look at Remark \ref{remark:constant_half}.
\begin{proof}
For  $j \in \NN$ let
$$Y_j = \bigcap_{t \ge j}T^{(1-2\delta_{n_t})F_{n_t}}\left(Z[\rho_t] \setminus D_{n_t,\frac{3}{8}\epsilon_{n_t}}[\rho_t, \rho_{t+1}]\right).$$
If $y \in Y_j$ then for all $t  \ge j$, there exists $\mi_t \in (1-2\delta_{n_t})F_{n_t}$ such that
$$T^{\mi_t}(y) \in Z[\rho_t] \mbox{ and } d_X^{F_{n_t}}\left(\rho_t(T^{\mi_t}(y)),\rho_{t+1}(T^{\mi_t}(y))\right) < \frac{3}{8}\epsilon_{n_t}.$$
Because $ n_t \gg  n_s$ whenever $t >s$, it follows  that for any $t>s >j$, if $\mi_s \in F_{n_s}$ and $\mi_t \in (1-2\delta_{n_t})F_{n_t}$ then
$F_{n_s}+ \mi_s \subset F_{n_t} + \mi_t$.
This implies that if $y \in Y_j$, then for every $t > s > j$
$$\exists \mi_s \in (1-2\delta_{n_s})F_{n_s} \mbox{ s.t. } T^{\mi_s}(y) \in Z[\rho_s] \mbox{ and } d_X^{F_{n_s}}\left(\rho_t(T^{\mi_s}(y)),\rho_{t+1}(T^{\mi_s}(y))\right) < \frac{3}{8}\epsilon_{n_t}.$$
Because $\epsilon_{n_{t+1}} < \frac{1}{16} \epsilon_{n_t}$, we have $$\sum_{t \ge j} \frac{3}{8}\epsilon_t \le \left(\sum_{\ell=0}^\infty \frac{3}{8}16^{-\ell} \right) \epsilon_j = \frac{2}{5}\epsilon_j.$$
So by the triangle inequality and for every $y \in Y_j$ and every $t > s> j$ 
\begin{equation}\label{eq:d_rho_t_rho_s_at_mi_s}
\exists 
 \mi_s \in (1-2\delta_{n_s})F_{n_s} \mbox{ s.t. } 
 T^{\mi_s}(y) \in Z[\rho_s] \mbox{ and }
d_X^{F_{n_s}}\left(\rho_t(T^{\mi_s}(y)),\rho_{s}(T^{\mi_s}(y))\right) < \frac{2}{5}\epsilon_{n_s}.
\end{equation}
Because $F_{n_{s-1}}\subset \mi_s+ F_{n_s}$ this shows  
\begin{equation}
d_X^{F_{n_{s-1}}}\left(\rho_t(y),\rho_{s}(y)\right) < \frac{2}{5}\epsilon_{n_s} \mbox{ for every } y \in Y_j \mbox{ and } t> s> j.
\end{equation}

Because $\lim_{s \to \infty}\epsilon_{n_s} =  0$ and $F_{n_{s-1}}$ increase to $\ZD$ as $s \to \infty$, it follows that for every $y \in Y_\infty$ and every $\mi \in \ZD$ the sequence $(\rho_j(y)_\mi)_{j=1}^\infty$ is a Cauchy sequence, and thus converges. So the limit $\rho(y)= \lim_{j \to \infty} \rho_j(y)$ exists for every $y \in Y_\infty$. This proves $(i)$.
By \eqref{eq:d_rho_t_rho_s_at_mi_s} for every $y \in Y_\infty$ there exists $j \in \NN$ such that  for every $ s> j$ there exists $\mi_s \in (1-2\delta_{n_s})F_{n_s}$ so that $T^{\mi_s}(y) \in Z[\rho_s]$ and 
\begin{equation}\label{eq:d_rho_rho_s_at_mi_s}
d_X^{F_{n_s}}\left(\rho(T^{\mi_s}(y)),\rho_{s}(T^{\mi_s}(y))\right) \le \frac{2}{5}\epsilon_{n_s}.
\end{equation}
Since $\rho_s$ is a $(k_s,n_s,\frac{1}{2^s},\mu)$ approximate-embedding and $T^{\mi_s}(y) \in Z[\rho_s] \subset Z_{n_s}$, and $F_{n_{s-1}} \subset F_{n_s}+ \mi_s$ there exists $x_s \in C_s \subset X$ such that $d_X^{F_{n_{s-1}}}(\rho(y),x_s) \le \frac{2}{5}\epsilon_{n_s}$. Using the previous argument, the sequence $(x_s)_{s=1}^\infty$ is a Cauchy sequence so it converges to a point $x \in X$, and so $(ii)$ holds, meaning that $\rho(y)$ is in the image of $X$ for every $y \in Y_\infty$.

We now  prove $(iii)$, namely that the function $\rho$ is injective.

 Because  the sequence of partitions $(\mathcal{P}_k)_{k=1}^\infty$ separates points, in order to show that $\rho:Y_\infty \to X^{\ZD}$ is injective, it suffices to show that for every $k \in \NN$ there exists a function $\Phi_k:X^{\ZD} \to \mathcal{P}_k$ so that 
\begin{equation}\label{eq:Phi_k_recover_P_k}
\forall y \in Y_\infty,~ \Phi_k(\rho(y))= \mathcal{P}_k(y).
\end{equation}
By Lemma \ref{lem:approximate_inverese},  because $\rho_j$ is a $(k_j,n_j,\frac{1}{2^j},\mu)$ approximate embedding, for every $s \in \NN$ and  every $k <k _s$ (using the fact that $\mathcal{P}_k \prec \mathcal{P}_{k_s}$ for every $k < k_s$) there exists a Borel  function $\Phi_{k,n_s}:X^{\ZD}\to \mathcal{P}_{k}$ so that 
\begin{equation}
\Phi_{k,n_s}(\rho(y)) = \mathcal{P}_{k}(y) \mbox{ for every } y \in T^{(1-2\delta_{n_s})F_{n_s}}(Z[\rho_s] \setminus D_{n_s,\frac{1}{2}\epsilon_{n_s}}[\rho_s,\rho]).
\end{equation}
We have already concluded that for every $y \in Y_\infty$ there exists $j \in \NN$ such that for every $s >j$ there exists $\mi_s \in (1-2\delta_{n_s})F_{n_s}$ so that $T^{\mi_s}(y) \in Z[\rho_s]$ and   \eqref{eq:d_rho_rho_s_at_mi_s} holds. But this precisely means that $$y \in T^{(1-2\delta_{n_s})F_{n_s}}(Z[\rho_s] \setminus D_{n_s,\frac{1}{2}\epsilon_{n_s}}[\rho_s,\rho]).$$
We conclude that for every $y \in Y_\infty$ there exists $j \in \NN$ such that for every $s >j$, and every $k \le k_s$ 
 $\Phi_{k,n_s}(\rho(y)) = \mathcal{P}_k(y)$.

 It follows that for every $k \in \NN$ and every $y \in Y_\infty$ the sequence $(\Phi_{k,n_s}(\rho(y)))_{s=1}^\infty$  stabilizes and we have $$\lim_{s \to \infty} \Phi_{k,n_s}(\rho(y))= \mathcal{P}_k(y).$$
 So we can define a Borel function $\Phi_k:X^{\ZD} \to \mathcal{P}_k$ so that 
 $\Phi_k(x)=\lim_{s \to \infty}\Phi_{k,n_s}(x)$ whenever the limit exists and \eqref{eq:Phi_k_recover_P_k} holds.
This completes the proof.
\end{proof}

\begin{lem}\label{lem:mu_Y_0_1}
	Let $(k_j)_{j=1}^\infty$ $(n_j)_{j=1}^\infty$ and $(\rho_j)_{j=1}^\infty$ be as in Lemma \ref{lem:convergence_set_for_rho_j_s} above. Further suppose that for every $j >1$
	\begin{equation}\label{eq:rho_j_shadows_rho_j_prev}
	\mu \left( D_{n_{j-1},\frac{3}{8}\epsilon_{n_{j-1}}}[\rho_j,\rho_{j-1}] \mid Z_{n_{j-1}}\right)<\frac{1}{2^j}.
	\end{equation}
	Then the set $Y_\infty \subset Y$ given by \eqref{eq:Y_infty_def} satisfies $\mu(Y_\infty)=1$. 
\end{lem}
\begin{proof}
	For $t \in \NN$ let
	\begin{equation}\label{eq:G_t_def}
	G_t = T^{(1-2\delta_{n_{t-1}})F_{n_{t-1}}}\left(Z[\rho_{t-1}] \setminus D_{n_{t-1},\frac{3}{8}\epsilon_{n_{t-1}}}[\rho_t, \rho_{t-1}]\right).
	\end{equation}
	By Lemma \ref{lem:mu_tower_subset}
	$$
	\mu(Y \setminus G_t) \le \epsilon_{n_{t-1}}+  6d\delta_{n_{t-1}}+
		\mu \left( D_{n_{t-1},\frac{3}{8}\epsilon_{n_{t-1}}}[\rho_t,\rho_{t-1}] \mid Z_{n_{t-1}}\right)
		+
		\mu\left( Z_{n_{t-1}} \setminus Z[\rho_{t-1}] \mid Z_{n_{t-1}} \right).
		$$
		Because $\rho_t$ is a $(k_t,n_t,\frac{1}{2^t},\mu)$-approximate embedding, 
		$$\mu \left(Z_{n_t} \setminus Z[\rho_t]  \mid Z_{n_t}\right) < \frac{1}{2^t}.$$
		So using \eqref{eq:rho_j_shadows_rho_j_prev} it follows that 
		$$\mu(Y \setminus G_{t+1}) \le
		\epsilon_{n_t} + 6d \cdot \delta_{n_t} + \frac{2}{2^t}.
		$$
	It follows that 
	$$ \sum_{t=2}^\infty \mu \left(Y \setminus G_t\right) \le \sum_{t=1}^\infty \left(\epsilon_{n_t}+ 6d \cdot \delta_{n_t} + \frac{1}{2^{t-1}}\right).$$
	By our assumption that $n_t \ll n_{t+1}$ it follows that  the series on the right hand side converges.
		Let $Y_\infty'= \bigcup_{j=1}^\infty \bigcap_{t > j} G_t$.
			By the Borel-Cantelli Lemma $\mu (Y_\infty') =1$. Now $Y_\infty = \bigcap_{\mi \in \ZD} T^{-\mi}Y_\infty'$, so it follows that $\mu(Y_\infty)=1$.
\end{proof}

\subsection{The case of infinite entropy}

Our next goal is to  prove Theorem \ref{thm:infinite_entropy} which states  flexibility implies universality for systems with infinite entropy. 
The proof follows similar pattern to that of the finite entropy case,  but it is  considerably less involved,  so  the reader can consider it as a preparation.

The following lemma regarding the existence of  approximate embeddings shows that in the presence of of a flexible sequence  $\C = (C_n)_{n=1}^\infty  \in X^\NN$ with $h(\C)=\infty$ one can actually get a $(k,n,\epsilon,\mu)$-approximate embedding for any $\mu \in \Prob_e(\Y)$ , $k \in \NN$ with $\epsilon =0$, provided that $n$ is sufficiently big.  The infinite entropy assumption  $h(\C)=\infty$ is essential for this.
\begin{lem}\label{lem:approximate_embedding_exists_inf_entropy}
	Suppose $\C = (C_n)_{n=1}^\infty  \in X^\NN$ is a flexible sequence with $h(\C)=\infty$.
	For every $k \in\NN$ 
	there exists $N_k$ such that for every $n >N_k$ there exists $\rho \in \Mor(\Y,\X)$ which is a
	  $(k,n,0,\mu)$-approximate embedding for any $\mu \in \Prob_e(\Y)$.
\end{lem}
The proof is a simplified version of the proof of Lemma \ref{lem:improve_approx_emb_inf_entropy} below, so we omit it.

\begin{lem}\label{lem:improve_approx_emb_inf_entropy}

	Suppose $\C = (C_n)_{n=1}^\infty  \in X^\NN$ is a flexible sequence with $h(\C)=\infty$. For every $k_0,k \in \NN$ and $\gamma \in (0,1)$ there exists $N_{k,\gamma} \in \NN$ so that for every 
		$n_0$-towerable $\rho \in \Mor(\Y,\X)$ with $n_0 \ge N_{k,\gamma}$, $\t x \in C_{n_0}$
		and any $n \gg n_0$,
		there exists $\tilde \rho \in \Mor(\Y,\X)$ which is a  $(k,n,0,\mu)$-approximate embedding for all $\mu \in \Prob_e(\Y)$, so that
	\begin{equation}\label{eq:tilde_rho_traces_rho2}
	D_{n_0,\frac{3}{8}\epsilon_{n_0}}[\rho,\tilde \rho] \subset Y \setminus T^{(1-2\delta_n)F_n \setminus \gamma F_{n}} Z_n
	\end{equation}
	and so that for any $y \in  Z_{n}$, $d_X(\t x, \t \rho(y)) < \epsilon_{n_0}$.
	
\end{lem}
\begin{proof}

Because $h(\C) = \infty$, for any  $k \in \NN$  and $\gamma \in (0,1)$ there exists 
$N_{k,\gamma}$ so that for any $n_0 > N_{k,\gamma}$
$$ \frac{1}{|(1+\delta_{n_0})F_{n_0}|}\log |C_{n_0}| >  {\frac{4^d}{\gamma^d}} \log |\mathcal{P}_k|.$$ 
Now fix $n\in \N$ such that $n > N_{k,\gamma}$ and $n \gg \lceil \gamma^{-1} n_0 \rceil $.
Let 
$$K_{n, n_0, \gamma}=  \frac{1}{2} \gamma F_{n}\cap\left(\lceil(2+ 2\delta_{n_0})n_0\rceil\right)\Z^d \setminus \{\vec{0}\}.$$
We have that $|K_{n, n_0, \gamma}|\ge\frac{|\frac{1}{2}\gamma F_{n}|}{|(1+\delta_{n_0})F_{n_0}|-1}$.
Thus
$$|C_{n_0}^{K_{n, n_0, \gamma}}| \ge | \P_{k}^{F_{n}}|.$$
Then there is an injective map 
$$\phi:  \P_{k}^{F_{n}}\to C_{n_0}^{K_{n, n_0, \gamma}}.$$
Let $\phi^{-1}$ be a left inverse of $\phi$. 
For $y \in Y$ set
$$K_y=\{\mi\in (1-2 \delta_n)F_n\setminus \gamma F_n ~:~ T^\mi(y)\in Z_{n_0}\}\cup K_{n, n_0, \gamma}.$$
and $W_y\in {C_{N_k}}^{K_y}$ by
$$(W_y)_\mi=\begin{cases}
\t x & \text{ if } \mi = \vec{0}\\
\phi\left(\P_{k+1}^{F_{N_{k+1}}}(y)\right)_\mi &\text{ if } \mi \in  K_{n, n_0, \gamma}\\
\rho(T^{i}(y))&\text{ otherwise.}
\end{cases}$$
Since $K_y \subset (1-2\delta_{N_k+1})F_n$ is $(1+\delta_{N_k})F_{N_k}$-spaced we can 
define $\t \Phi: Y\to C_{N_{k+1}}$ by 
$$\t \Phi(y)= \Ext(W_y).$$

Let $\tilde \rho = \rho_{\tilde \Phi,n} \in \Mor(\Y,\X)$ be given by \eqref{eq:rho_Phi_n_def}. 
Then $\tilde \rho$ is a $(k,n,0,\mu)$-approximate embedding because we can take
$Z[\tilde \rho]=Z_n$ and 
\eqref{eq:tilde_rho_traces_rho2} holds. Also, whenever $y \in Z_n$,  $d_X(\t x, \t \rho(y)) < \epsilon_{n_0}$.
\end{proof}

\begin{proof}[Proof of Theorem \ref{thm:infinite_entropy}]
	We assume that $\C$ is a flexible sequence on $X$ and that $h(\C)=\infty$. Suppose we have a sequence $(x_k)_{k=1}^\infty$ so that $x_k \in C_k$ and $\{x_k\}_{k=1}^\infty$ is dense in $\cup_{k=1}^\infty C_k$.
	By induction, construct an increasing sequence $(n_k)_{k=1}^\infty$ of natural numbers so that $n_1 \gg 1$ is also big enough to satisfy the conclusion of Lemma \ref{lem:approximate_embedding_exists_inf_entropy} with $k=1$ and also so that for any $k \ge 1$ 
	$$ n_{k} \ge N_{k+1,2^{-(k+11)}} \mbox { and } n_{k+1} \gg \lceil 2^{(k+11)}n_k \rceil,$$
	where $N_{k,\gamma}$ is a number that satisfies the conclusion of Lemma \ref{lem:improve_approx_emb_inf_entropy}.

Note that this in particular  implies that for any $k \ge 1$
$$\epsilon_{n_k}+ \delta_{n_k} < 2^{-(k+1)}.$$
By induction, we will construct sequences $\rho_k \in \Mor(\Y,\X)$, $\tilde x_k \in C_{n_k}$ so that 
$$ d_X(\tilde x_k, x_k) < \epsilon_k,$$
and 
 $\rho_k$ is a $(k,n_k,0,\mu)$-approximate embedding for all $\mu \in \Prob_e(\mathcal Y)$ and such that
$$
D_{n_k,\frac{1}{4}\epsilon_{n_k}}[\rho_k,\rho_{k+1}]\subseteq  E_k,
$$
where
$$E_k = Y \setminus T^{(1-2\delta_{n_{k+1}})F_{n_{k+1}} \setminus 2^{-(k+10)} F_{n_{k+1}}} Z_{n_{k+1}},$$
and so that
$$ \forall\  y\in Z_{n_{k+1}},~ d_X(\rho_{k+1}(y),\tilde x_k) < \epsilon_{n_k}.$$ 
Lemma \ref{lem:convergence_set_for_rho_j_s} implies that the sequence converges to
 a Borel embedding of $(Y_\infty,T)$ into $(X,S)$, where 
$$
	Y_\infty = \bigcap_{\mi \in \ZD}T^{-\mi}\bigcup_{j=1}^\infty \bigcap_{t \ge j}T^{(1-2\delta_{n_t})F_{n_t}}\left(Z_t \setminus E_t\right).
	$$
	Such sequences can be constructed inductively: To start the induction, apply Lemma \ref{lem:approximate_embedding_exists_inf_entropy}. 
	For the induction step, apply  Lemma \ref{lem:improve_approx_emb_inf_entropy} with $\rho= \rho_k$, and $n_0=n_k$, $n = N = n_{k+1}$, $\gamma = \frac{1}{2^{k+10}}$. Let us check that  \eqref{eq:rho_j_shadows_rho_j_prev} holds:
	Note that
	$$  E_k 
	= \left(Y \setminus T^{(1-2\delta_{n_{k+1}})F_{n_{k+1}}}Z_{n_{k+1}}\right) \cup
	\left(T^{2^{-(k+10)} F_{n_{k+1}}} Z_{n_{k+1}}\right).$$
	Fix $\mu \in \Prob_e(\mathcal Y)$. Now 
	$$ \mu\left( Y \setminus T^{(1-2\delta_{n_{k+1}})F_{n_{k+1}}}Z_{n_{k+1}} \mid Z_{n_k}\right) \le
	\frac{1}{\mu(Z_{n_k})}\mu\left( Y \setminus T^{(1-2\delta_{n_{k+1}})F_{n_{k+1}}}Z_{n_{k+1}}\right) .
	$$
	So Lemma \ref{lem:mu_tower_subset}  implies that  
	$$ \mu\left( Y \setminus T^{(1-2\delta_{n_{k+1}})F_{n_{k+1}}}Z_{n_{k+1}} \mid Z_{n_k}\right)
	 \le 2|F_{n_k}| \left(\epsilon_{n_{k+1}} + 6d\delta_{n_{k+1}}\right)
	 \le \epsilon_{n_k} +\delta_{n_k} < 2^{-(k+1)}.
	$$
	By
	 Lemma \ref{lem:eq:mu_tower_subset_cond},
	$$
	 \mu\left(T^{2^{-(k+10)} F_{n_{k+1}}} Z_{n_{k+1}} \mid Z_{n_k}\right) 
	 \le 2^{-(k+9)} + \delta_{n_k}< 2^{-(k+1)}.
	$$
Together this implies that  
	$$
	\mu\left(E_k \mid Z_{n_k} \right) < 2^{-k}.
	$$
 Lemma \ref{lem:mu_Y_0_1} now implies that $\mu(Y_\infty)=1$ for any $\mu \in \Prob_e(\Y)$.
To see that the support of $\mu \circ \rho^{-1}$ is full for any $\mu \in \Prob_e(\Y)$, note that the assumption that $\bigcup_{k=1}^\infty C_k$ is dense in $X$ implies that the sequence $(\t x_k)_{k=1}^\infty$ is also dense in $X$.
The construction implies that for any $k \in \NN$
$$ \mu(\left\{y \in Z_{k}:~ d_X(\t x_k,\rho(y)) < 2\epsilon_k \right\} \mid Z_{k}) > \frac{1}{2}.$$
This implies that $\mu \circ \rho^{-1} (U) >0$ for  every non-empty open set  $U \subseteq X$.
\end{proof}

\subsection{The case of finite entropy}
In the following  subsection we prove that flexibility  implies ergodic universality  for the more difficult case that $h(\C) < \infty$, that in particular  confirms \cite[Conjecture $1$]{MR3453367}.
For the finite entropy case we use  Lemma \ref{lem:improve_approx_emb} below as a replacement for   Lemma \ref{lem:improve_approx_emb_inf_entropy}. The statement and the proof of Lemma \ref{lem:improve_approx_emb} are similar to those of Lemma \ref{lem:improve_approx_emb_inf_entropy}, but slightly more involved.

\begin{lem}\label{lem:C_gamma_n}

	Fix $0 < \hat{h}< h(\C)$. 
	Consider the functions
	$$\alpha,\beta:(0,1) \to \mathbb{R}$$
		defined by:
	\begin{equation}\label{eq:beta_def}
\index{Definitions and notation introduced in Section 6!$\beta(\gamma)$}	\beta(\gamma) = \sqrt[d]{\frac{\hat{h}+h(\C)}{2h(\C)}}\gamma
	\end{equation}
	and
	\begin{equation}\label{eq:alpha_def}
\index{Definitions and notation introduced in Section 6!$\alpha(\gamma)$}	\alpha(\gamma) = \frac{1}{2}\min\left\{\left(\beta(\gamma)^d h(\C)- \hat{h}\gamma^d\right),\frac{1}{10}\left(\gamma - \beta(\gamma)\right)\right\},
	\end{equation}
	then for every $\gamma \in (0,1)$ 
	$$ \alpha(\gamma),\beta(\gamma) \in (0,1),$$
	and
	\begin{equation}\label{eq:gamma_big}
	\gamma > \beta(\gamma) + 10 \alpha(\gamma).
	\end{equation}
	In addition, there exists $K_0 \in \NN$ and  a function	$$N_T:(0,1) \to \NN$$ \index{Definitions and notation introduced in Section 6!$N_T(\gamma)$}
	so that for every $\gamma \in (0,1)$,  $k > K_0$   and  $n >N_T(\gamma)$ the following hold:
    \begin{equation}\label{eq:gamma_n_big}
    (1-2\delta_n) \gamma> (1+\delta_n)\beta(\gamma)+10 \alpha(\gamma),
    \end{equation}
	\begin{equation}\label{eq:C_n_gamma_big}
    \frac{|\beta(\gamma)F_n|}{|(1+\delta_k)F_k|}\log |C_k| > \alpha(\gamma)|F_n|+\hat{h} |\gamma F_n|, 
	\end{equation}
	\begin{equation}\label{eq:C_n_alpha_big}
 \frac{|\alpha(\gamma)F_n|}{|(1+\delta_k)F_k|}\log |C_k| > \hat {h} \alpha(\gamma)^d|F_n|. 
	\end{equation}

\end{lem}
\begin{proof}
	It is clear from \eqref{eq:beta_def} that $\beta(\gamma) >0$ for every $\gamma >0$. On the other hand, because $\hat{h} < h(\C)$ it follows that $\beta(\gamma)< \gamma$ whenever $\gamma >0$. In particular if $0< \gamma <1$ then $0 < \beta(\gamma) <1$.
	So for $\gamma \in (0,1)$ it follows that
	$$ 0 < \beta(\gamma) < \gamma < 1,$$
	and in this case by \eqref{eq:alpha_def} 
	$$\alpha(\gamma) \le \frac{1}{10}(\gamma-\beta(\gamma)) < \frac{1}{10} <1.$$
	In addition, by \eqref{eq:beta_def},
 	\begin{equation}\label{eq:gamma_beta_bla}
 	\beta(\gamma)^d h(\C)- \hat{h}\gamma^d= \left(\frac{h(\C)-\hat{h}}{2} \right)\gamma^d.
 	\end{equation}
	So $\alpha(\gamma) >0$ whenever $\gamma >0$. So we see that $\alpha,\beta:(0,1) \to (0,1)$ are well defined.
	Also, from the definition of $\alpha(\gamma)$ it is clear that \eqref{eq:gamma_big} holds. Since $\lim_{n \to \infty}\delta_n =0$,
	it follows that \eqref{eq:gamma_n_big} holds for all sufficiently big $n$.
	By definition of $\alpha(\gamma)$,
	$$\beta(\gamma)^dh(\C) > \alpha(\gamma)+ \hat{h}\gamma^d.$$
	Note that as $n \to \infty$.
	$$e^{\alpha(\gamma)|F_n|+\hat{h}|\gamma F_n|} = e^{\alpha(\gamma)|F_n|+\hat{h}\gamma^d|F_n|+o(|F_n|)}.$$
	By definition,
	$$\lim_{k \to \infty}\frac{1}{|F_k|}\log |C_k| = h(\C).$$
	So there exists $K_0 \in \NN$ such that for any $k > K_0$ 
	$$ \frac{1}{|(1+\delta_k)F_k|}\log |C_k| > \hat h.$$
It follows that \eqref{eq:C_n_gamma_big} and \eqref{eq:C_n_alpha_big}    hold for all sufficiently large $n$
and we can choose $N_T(\gamma)$ to be the smallest $N \in \NN$ so that \eqref{eq:gamma_n_big}, \eqref{eq:C_n_gamma_big}  and \eqref{eq:C_n_alpha_big} hold for all $n >N$.

\end{proof}

The following lemma is a crucial step in the proof of our main result. It says that we can slightly modify a given approximate embedding and get a much better one that is ``close'' to the original on a ``big part of the space''. The ``extent of modification required'' depends on the ``quality'' of the original approximate embedding, and goes to zero as the original approximate embedding gets better and better.

\begin{lem}\label{lem:improve_approx_emb}
	For every $\gamma \in (0,1)$ there exists $k_0,N_0 \in \NN$ such that for every $n_0 \ge N_0$ and every   $(k_0,n_0,\frac{1}{N_0},\mu)$-approximate embedding $\rho$, every $k \in \NN$ and $\delta >0$ there exists $N\in\N$ such that for all $n>N$ there is a $(k,n,\delta,\mu)$-approximate embedding $\tilde \rho \in \Mor(\Y,\X)$ so that
	\begin{equation}\label{eq:tilde_rho_traces_rho}
D_{n_0,\frac{3}{8}\epsilon_{n_0}}[\rho,\tilde \rho] \subset Y \setminus T^{(1-2\delta_n)F_n \setminus \gamma F_n} Z_n.
	\end{equation}
	
\end{lem}

\begin{remark}\label{remark:constant_half}
	Informally, equation \eqref{eq:tilde_rho_traces_rho} means that for ``most'' $y\in Y$, $\rho(y)$ and  $\tilde \rho(y)$ are very close in $d_X^{F_{n_0}}$ metric. To see this note that for large $n$ the set on the right hand side of \eqref{eq:tilde_rho_traces_rho} has very small measure because it consists of a very small fraction of the levels of the tower with base $Z_n$ together with the complement of the tower. The constant $\frac{3}{8}$ in front of the $\epsilon_{n_0}$ is somewhat arbitrary; the important point is that it is strictly smaller than $\frac{1}{2}$.  This is important due to the fact that $C_{n_0}$ is $(\epsilon_{n_0},F_{n_0})$-separated and allows us to recover the $\mathcal{P}_{k_0}$ names of $y$. This remark applies to similar numerical constants appearing later in the paper.
\end{remark}

Before proceeding with the proof of Lemma \ref{lem:improve_approx_emb} let us give a brief informal outline: Given a $(k_0,n_0,\frac{1}{N_0},\mu)$-approximate embedding $\rho$, the new $(k,n,\delta,\mu)$-approximate embedding $\tilde \rho \in \Mor(\Y,\X)$ will be defined roughly as follows:
Whenever $y \in Z_n$, $\tilde \rho(y)\mid_{F_n}$ will be the $F_n$-orbit segment an element of $C_n$, 
so that for ``most'' $\mi \in F_n$, $S^{\mi}(\tilde \rho(y))$ is very close to $S^{\mi}(\rho(y))$.
More specifically, outside some ``small fixed regions'' in $F_n$ (namely those that are $2\delta_n$-close to the boundary and those within the small box $\gamma F_n$) , if $T^{\mi}(y) \in Z_{n_0}$ then 
$S^{\mi}(\tilde \rho(y))\mid_{F_{n_0}}$ will be very close to $S^{\mi}( \rho(y))\mid_{F_{n_0}}$. The orbit segment $\tilde\rho (y)\mid_{\gamma F_n}$ is used to ``encode additional information regarding $\mathcal{P}_k^{F_n}(y)$'', beyond the information in $\rho(y)\mid_{F_n}$.  
 \begin{proof}
   Fix $\gamma \in (0,1)$. Choose $\hat{h} \in (h_\mu(\Y), h(\C))$. Because $(\mathcal{P}_{k})_{k=1}^\infty$ are an increasing sequence of partitions and 
$\bigvee_{k=1}^\infty \mathcal{P}_k^{\ZD} = \Borel(Y)$, it follows that $\lim_{k \to \infty} h_\mu(\Y \mid \mathcal{P}_{k}^{\ZD}) =0$. So we can define
\begin{equation}\label{eq:k_0_def}
\overline{k_0}(\mu,\gamma)= \min \{k \in \NN:~ h_\mu\left(\Y \mid \mathcal{P}_k^{\ZD} \right) < \alpha(\gamma)/8 \},
\end{equation}
where $\alpha:(0,1) \to \mathbb{R}$ is given by \eqref{eq:alpha_def}.
Let $k_0 = \overline{k_0}(\mu,\gamma)$. We thus have, 
\begin{equation}\label{eq:h_mu_rel_P_k_0_small}
h_\mu(\Y \mid \mathcal{P}_{k_0}^{\ZD}) < \alpha(\gamma)/8.
\end{equation}

For any $n_0,n \in \NN$ with $n > n_0$ recall $S_{n,n_0}$ as defined in \eqref{eq:S_n_n0_def}.
Because every element of $S_{n,n_0}$ is a subset of $F_n$ that has cardinality at most $|F_{n}|/|F_{n_0}|$, 
for every $n >n_0$ 
$$\frac{1}{|F_n|} \log|S_{n,n_0}| \le \frac{1}{|F_n|}\log\left(|F_n|   {{|F_n|}\choose{\frac{|F_n|}{|F_{n_0}|}}} \right)  \leq \frac{\log{|F_n|}}{|F_n|} +\mathcal H\left(\frac{1}{|F_{n_0}|}\right).$$
In particular, because $\lim_{p \to 0+}\H(p)=0$ it follows that
\begin{equation}\label{eq:S_n_n0_small0}
\lim_{n_0 \to \infty} \sup_{n > n_0}\frac{1}{|F_n|}\log |S_{n,n_0}|=0.
\end{equation}

For $\epsilon >0$ let
\begin{equation}\label{eq:N_S_def}
\overline{N}_{S} (\epsilon) = \inf\left\{ N \in \NN:~\ \forall\ n > n_0 > N, ~ |S_{n,n_0}| < \exp\left(\epsilon |F_n| \right)\right\}.
\end{equation}
The limit \eqref{eq:S_n_n0_small0} shows that $\overline{N}_{S}:(0,1) \to \NN$ is indeed well defined and finite.

Set
\begin{equation}\label{eq:N_0_def}
N_0 = \max \left\{\overline{N}_{S}(\alpha(\gamma)^d\hat{h}),\overline{N}_R(\alpha(\gamma)/4,k_0),K_0\right\},
\end{equation}
where $\alpha:(0,1) \to (0,1)$ is given by \eqref{eq:alpha_def}, and $\overline{N}_R(\alpha(\gamma)/4,k_0)$ 
is the integer obtained by applying Lemma \ref{lem:COV_rel_to_approximate_embedding} with $\eta=\alpha(\gamma)/4$ and $k_0$, and $K_0 \in \NN$ is the constant that appears in Lemma \ref{lem:C_gamma_n}.
 Explicitly, this  means that for every $n_0 > N_0$ , $(k_0,n_0,\frac{1}{N_0},\mu)$-approximate embedding $\rho \in \Mor(\Y,\X)$  and any $\delta'>0$
 there exists $N \in \NN$ such that for all $n >N$ 
 \begin{equation}\label{eq:COV_P_k_0_rel_rho}
 \COV_{\mu,\delta',\mathcal{P}_{k_0}^{F_n \setminus \gamma F_{n}}\mid \mathcal{P}_\rho^{(1-2\delta_n)F_n \setminus \gamma F_n}}(Z_n) < e^{\alpha(\gamma)/4 |F_n|}
 \end{equation} and also
 \begin{equation}\label{eq:S_n_n0_small}
 |S_{n,n_0}| < \exp\left(\alpha(\gamma)^d  \hat{h} |F_n| \right),
 \end{equation}
 where $S_{n,n_0}$ is defined by \eqref{eq:S_n_n0_def}.

 At this point we fix an arbitrary $n_0 >N_0$,  $x_0 \in C_{n_0}$,  a  $(k_0,n_0,\frac{1}{N_0},\mu)$-approximate embedding $\rho \in \Mor(\Y,\X)$ and $\delta \in (0,1)$.
 Let $\delta ' >0$ be a positive number much smaller than  $\delta$, so that 
 \begin{equation}\label{eq:delta_prime_small}
 \delta' < 10^{-10}\delta^4.
 \end{equation}

By Lemma \ref{lem:tower_cov}  there exists $N \in \NN$ so that for every $n >N$ 
\begin{equation}\label{eq:COV_P_k_0_gamma_F_n}
\COV_{\mu,\delta', \mathcal{P}_{k_0}^{\gamma F_n} }(Z_n) < e^{ \hat h |\gamma F_n|}.
\end{equation}
Using the inequality \eqref{eq:COV_almost_monotone} from Lemma \ref{lem:COV_submultiplicative} it follows from \eqref{eq:COV_P_k_0_gamma_F_n} that
\begin{equation}\label{eq:COV_P_k_0_gamma_F_n_cond}
\COV_{\mu,\sqrt{\delta'}+\delta', \mathcal{P}_{k_0}^{\gamma F_n} \mid \mathcal{P}_\rho^{(1-2\delta_n)F_n \setminus \gamma F_n}}(Z_n) < 
	e^{ \hat h |\gamma F_n|}.
\end{equation}

Choose $k \in \NN$. By Lemma \ref{lem:rel_tower_cov} and \eqref{eq:h_mu_rel_P_k_0_small} there exists $N \in \NN$ so that for every $n >N$ 
\begin{equation}\label{eq:rel_COV_P_k_F_n_P_k_0_F_n_Z_n}
\COV_{\mu,\delta', \mathcal{P}_k^{F_n} \mid \mathcal{P}_{k_0}^{F_n}}(Z_n) < e^{\alpha(\gamma)/4 |F_n|}.
\end{equation}

Let $N_T:(0,1)\to \NN$ be the function given by  Lemma \ref{lem:C_gamma_n}. In particular, for every $n > N_T(\gamma)$ \eqref{eq:C_n_alpha_big} holds. Also, for $n \gg n_0$ and $\theta \in (0,1)$ let
\begin{equation}\label{eq:K_n_n_0_theta_def}
\index{Definitions and notation introduced in Section 6!$K_{n,n_0,\theta}$}K_{n,n_0,\theta}= \theta F_n \cap  \lceil (2+2\delta_{n_0})n_0\rceil \ZD.
\end{equation} 
 Using \eqref{eq:S_n_n0_small}, it follows that if $n > \max\{n_0,N_T(\gamma)\}$ then
\begin{equation}\label{eq:C_n_alpha_bigger_S_n_n_0}
|C_{n_0}^{K_{n,n_0,\alpha(\gamma)}}| > |S_{n,n_0}|.
\end{equation}

Also, by Lemma \ref{lem:C_gamma_n}, for every $n > N_T(\gamma)$ we have that \eqref{eq:gamma_n_big} and \eqref{eq:C_n_gamma_big} hold. Let $N$ be the smallest integer so that $N > N_T(\gamma)$, $N \gg n_0$, \eqref{eq:COV_P_k_0_gamma_F_n}, \eqref{eq:rel_COV_P_k_F_n_P_k_0_F_n_Z_n} hold and also
\begin{equation}\label{eq:epsilon_n_smaller_eta}
\epsilon_N < \frac{1}{16}\min\{\alpha(\gamma/4),\epsilon_{n_0}\}.
\end{equation}
 It follows that \eqref{eq:COV_P_k_0_rel_rho}, \eqref{eq:COV_P_k_0_gamma_F_n_cond} and  \eqref{eq:C_n_alpha_bigger_S_n_n_0} also hold for all $n >N$.

Now choose any $n >N$. From \eqref{eq:COV_P_k_0_rel_rho} and \eqref{eq:COV_P_k_0_gamma_F_n_cond} using \eqref{eq:COV_join_submultiplicative} of Lemma \ref{lem:COV_submultiplicative} it follows that
\begin{equation}\label{eq:COV_P_k_0_F_n_cond_rho}
\COV_{\mu, \sqrt{\delta'} + 2\delta' ,\mathcal{P}_{k_0}^{F_n}  \mid \mathcal{P}_\rho^{(1-2\delta_n)F_n \setminus \gamma F_n}}(Z_n)< e^{\frac{\alpha(\gamma)}{4}|F_n| + \hat h |\gamma F_n| }.
\end{equation}
By \eqref{eq:cov_triangle} of Lemma \ref{lem:COV_submultiplicative} 
$$
\COV_{\mu,\delta/3,\mathcal{P}_{k}^{F_n} \mid \mathcal{P}_\rho^{(1-2\delta_n)F_n \setminus \gamma F_n}}(Z_n) \le 
\COV_{\mu,\frac{(\delta/3)^2}{6},\mathcal{P}_{k}^{F_n} \mid \mathcal{P}_{k_0}^{F_n}}(Z_n)
\cdot
\COV_{\mu, \frac{(\delta/3)^2}{6} ,\mathcal{P}_{k_0}^{F_n}  \mid \mathcal{P}_\rho^{(1-2\delta_n)F_n \setminus \gamma F_n}}(Z_n).
$$
The inequality \eqref{eq:delta_prime_small} ensures that $\delta'$ is sufficiently small so that
$$
\sqrt{\delta'}+2\delta'< \frac{(\delta/3)^2}{6}.
$$
So using 
  \eqref{eq:rel_COV_P_k_F_n_P_k_0_F_n_Z_n}  and \eqref{eq:COV_P_k_0_F_n_cond_rho} it follows that
\begin{equation}\label{equation: covering is good}
\COV_{\mu,\delta/3,\mathcal{P}_{k}^{F_n} \mid \mathcal{P}_\rho^{(1-2\delta_n)F_n \setminus \gamma F_n}}(Z_n) < e^{\alpha(\gamma)/2 |F_n| + \hat h |\gamma F_n|}.
\end{equation}

This means that there is a set $\hat X_{n} \subset X^{(1-2\delta_n)F_n\setminus \gamma F_n}$ so that 
\begin{equation}\label{eq:mu_hat_X_n1}
\mu \left(Z_n \setminus \bigcup_{w \in \hat{X}_n}\rho^{-1}([w])\right) \le  \frac{\delta}{3}\mu(Z_n),
\end{equation}
and so that for every $w \in \hat{X}_n$ there exists a set $\mathcal{G}_w \subset \mathcal{P}_k^{F_n}$ such that 
\begin{equation}\label{eq:G_w_small}
|\mathcal{G}_w| < e^{\alpha(\gamma)/2 |F_n| + \hat h |\gamma F_n|}
\end{equation}
and
\begin{equation}\label{eq:mu_G_w_big}
\mu\left( Z_n  \cap \rho^{-1}([w])\setminus \bigcup \mathcal{G}_w \right) \le \frac{\delta}{3}\mu(\rho^{-1}([w])  \cap Z_n).
\end{equation}

By \eqref{eq:C_n_gamma_big} and \eqref{eq:G_w_small} it follows that 
\begin{equation}\label{eq:G_w_small_C_eta}
\forall w \in \hat{X}_n, ~|\mathcal{G}_w| < |C_{n_0}^{K_{n,n_0,\beta(\gamma)}} |.
\end{equation}
For each $w \in \hat X_n$ let $\Phi_w:\mathcal{P}_k^{F_n} \to C_{n_0}^{K_{n,n_0,\beta(\gamma)}} $ be a function such that the restriction of $\Phi_w$ to 
$\mathcal{G}_w$ is injective and 
let $\Phi_w^{-1}:C_{n_0}^{K_{n,n_0,\beta(\gamma)}} \to \mathcal{P}_k^{F_n}$ be a left inverse on $\mathcal{G}_w$.
By this we mean that 
\begin{equation}\label{eq:check_inverse_w}
{\Phi_w}^{-1} \circ \Phi_w (P) = P \mbox{ for every } P \in {\mathcal{G}_w}.
\end{equation}

 	Define $\Upsilon_{n,n_0}:Y \to S_{n,n_0}$ by
 	\begin{equation}\label{eq:Psi_n_n0_def}
 	\Upsilon_{n,n_0}(y) = \left\{\mi \in (1-2\delta_{n})F_{n}:~ T^{\mi}(y) \in Z_{n_0}	\right\}.
 	\end{equation}
 	The fact that $\Upsilon_{n,n_0}(y) \in S_{n,n_0}$ follows because $\{T^{\mi}Z_{n_0}\}_{\mi \in (1+\delta_{n_0})F_{n_0}}$ are pairwise disjoint.
 
 	Recall  that $C_{n_0}^{K_{n,n_0,\alpha(\gamma)}}$ satisfies \eqref{eq:C_n_alpha_bigger_S_n_n_0}. Let $\Phi_b:S_{n,n_0} \to C_{n_0}^{K_{n,n_0,\alpha(\gamma)}}$ be an injective function with left inverse $\Phi_b^{-1}:C_{n_0}^{K_{n,n_0,\alpha(\gamma)}} \to  S_{n,n_0}$.
 	 By this we mean that
 	\begin{equation}\label{eq:check_inverse_b}
 	{\Phi_b}^{-1} \circ \Phi_b (s) = s \mbox{ for every } s \in {S_{n,n_0}}.
 	\end{equation}
Choose $\mi_b$ so that 
$$ (1+\delta_n)(\beta(\gamma) + \alpha(\gamma)) \cdot n < \|\mi_b\|_\infty < (1+\delta_n)(\beta(\gamma)+ 2\alpha(\gamma)) \cdot n.$$
It follows that
\begin{equation}\label{eq:mi_b_disjoint}
\left(\mi_b + (1+\delta_{n})F_{\lfloor \alpha(\gamma) n \rfloor}\right) \cap  (1+\delta_{n})F_{\lfloor \beta(\gamma )n \rfloor} = \emptyset.
\end{equation}
By \eqref{eq:gamma_n_big}
\begin{equation}\label{eq:mi_1_mi_2_mi_3_subset_gamma}
\left(\mi_b + (1+\delta_{n})F_{\lfloor \alpha(\gamma) n \rfloor}\right) \cup  (1+\delta_{n})F_{\lfloor \beta(\gamma )n \rfloor} \subset   (1-2\delta_n)\gamma F_{n}.
\end{equation}
For $y \in Y$, let
$$K_y =  \Upsilon_{n,n_0}(y) \cap (1-2\delta_n)F_n \setminus \gamma F_n$$
and
\begin{equation}\label{eq:overline_K_y}
\overline{K}_y = K_y \cup \left(\mi_b+ K_{n,n_0,\alpha(\gamma)}\right) \cup K_{n,n_0,\beta(\gamma)}.
\end{equation}
Then $\overline{K}_y \subset (1-\delta_n)F_{n}$ is $(1+\delta_{n_0})F_{n_0}$-spaced.  
 Let 
$\overline{W}_y \in C_{n_0}^{\overline{K}_y}$ be given by 

\begin{equation}\label{eq:overline_W_y_def}
(\overline{W}_y)_{\mi}=\begin{cases}
\rho(T^{\mi}(y)) & \mi \in K_y\\
(\Phi_b(\Upsilon_{n,n_0}(y)))_{\mi-\mi_b} & \mi \in \mi_b + K_{n,n_0,\alpha(\gamma)}\\
 (\Phi_{\rho(y)\mid_{(1-2\delta_n)F_n\setminus \gamma F_n}}(\mathcal{P}_k^{F_n}(y)))_{\mi} & \mi \in K_{n,n_0,\beta(\gamma)}.
\end{cases}
\end{equation}
Define $\tilde \phi:Y \to C_n$ by 
\begin{equation}\label{eq:tilde_phi_def}
\tilde \phi(y) = \Ext_n(\overline{W}_y),
\end{equation} 
and
let $\tilde \rho = \rho_{\tilde \phi,n} \in \Mor(\Y,\X)$ be given by \eqref{eq:rho_Phi_n_def}. 
If $y \in Z_{n_0}$ whenever $\mi \in (1-2\delta_n)F_n \setminus \gamma F_n$ such that $T^{\mi}y \in Z_n$ then 
$$\tilde \rho(y)\mid_{F_{n_0}}=S^{-\mi}(\tilde \phi(T^{\mi}(y)))\mid_{F_{n_0}} =$$
$$ = S^{-\mi}(\Ext(\overline{W}_{T^{\mi}(y)}))\mid_{F_{n_0}}.$$
Also, because $y \in Z_{n_0}$, $-\mi \in \Upsilon_{n,n_0}(T^{\mi}(y)) \cap (1-2\delta_n)F_n \setminus \gamma F_n$.
It follows that $-\mi \in \overline{K}_{T^{\mi}(y)}$ and $(\overline{W}_{T^{\mi}(y)})_{-\mi}= \rho(y)_{\vec{0}} $.
Thus,
$$d_X^{F_{n_0}}(\rho(y),\tilde \rho(y)) = d_X^{F_{n_0}}(\rho(y),S^{-\mi}(\Ext(\overline{W}_{T^{\mi}(y)})))  < \frac{1}{4}\epsilon_{n_0}.$$
This implies \eqref{eq:tilde_rho_traces_rho}. Here we could have used the better constant $\frac14 \epsilon_{n_0}$ but the bound $\frac38 \epsilon_{n_0}$ is sufficient for our purposes. It remains to show that $\tilde \rho$ is a $(k,n,\delta,\mu)$ approximate embedding.
Because $\tilde \phi$ takes values in $C_n$, it is clear that $\tilde \rho$ is $n$-towerable. So to complete the proof we will find a  Borel set $Y_0 \subset Z_n$ so that
\begin{equation}\label{eq:P_k_deteremined_by_tilde_rho_on_Y_0}
\forall y \in Y_0, ~ \mathcal{P}_k^{F_n}(y) \mbox{ is uniquely determined by } \tilde \rho(y)\mid_{F_n}
\end{equation}
   and
\begin{equation}\label{mu_Z_n_setminus_Y_0_small}
\mu(Z_n \setminus Y_0) < \delta \mu(Z_n).
\end{equation}
Let
\begin{equation}\label{eq:Y_0_def}
Y_0 = \left\{y \in Z_n:~ \exists w \in \hat X_n \mbox{ so that } \rho(y)|_{(1-2\delta_n)F_n \setminus \gamma F_n} =w \mbox{ and } \mathcal{P}_k^{F_n}(y) \in \mathcal{G}_{w}\right\}.
\end{equation}
Thus,
$$\mu(Z_n \setminus Y_0) \le \mu(Z_n \setminus \bigcup_{w \in \hat X_n} \rho^{-1}([w]) )+
\sum_{w \in \hat X_n}\mu\left( Z_n \cap \rho^{-1}([w]) \setminus \bigcup \mathcal{G}_w \right).$$
So \eqref{mu_Z_n_setminus_Y_0_small} follows from  \eqref{eq:mu_hat_X_n1} and \eqref{eq:mu_G_w_big}.

For $w \in X^{\ZD}$ let $\Pi_{\gamma,n}(w) \in  C_{n_0}^{K_{n,n_0,\alpha(\gamma)}}$ satisfy 
\begin{equation}\label{eq:Pi_C_star_def}
\max_{\mj \in K_{n,n_0,\alpha(\gamma)}}d_X^{F_{n_0}}(S^{\mj}(w),\Pi_{\gamma,n}(w)_\mj) = \min\left\{\max_{\mj \in K_{n,n_0,\alpha(\gamma)}}d_X^{F_{n_0}}(S^{\mj}(w),w'_\mj):~ w' \in  C_{n_0}^{K_{n,n_0,\alpha(\gamma)}}\right\}.
\end{equation}
By choosing some fixed rules for ``breaking ties'', this  clearly defines a  Borel function $\Pi_{\gamma,n}:X^{\ZD} \to C_{n_0}^{K_{n,n_0,\alpha(\gamma)}}$ so that \eqref{eq:Pi_C_star_def} holds for every $w \in X^{\ZD}$.

Because $C_{n_0}$ is $(\epsilon_{n_0},F_{n_0})$-separated it follows that whenever $w \in X^{\ZD}$, $w' \in C_{n_0}^{K_{n,n_0,\alpha(\gamma)}}$  and $$\max_{\mj \in K_{n,n_0,\alpha(\gamma)}}d_X^{F_{n_0}}(S^{\mj}(w),w'_\mj) < \frac{1}{2}\epsilon_{n_0}$$ then $\Pi_{\gamma,n}(w) = w'$.
Because 
$(\Phi_b(\Upsilon_{n,n_0}(y)))_\mi  =  (\overline{W}_y)_{{\mi + \mi_b}}$ for every $\mi \in K_{n,n_0,\alpha(\gamma)}$,
$$\max_{\mj \in K_{n,n_0,\alpha(\gamma)}}d_X^{F_{n_0}}(S^{\mi_b+\mj}(\Ext_n(\overline{W}_y)),\Phi_b(\Upsilon_{n,n_0}(y))_\mj) < \frac{1}{4}\epsilon_{n_0}.$$
  It follows that for any $y \in Y_0$
 	\begin{equation}\label{eq:tilde_Phi_y_determines_Psi_n_y}
 	\Upsilon_{n,n_0}(y) = \Phi_b^{-1}\left(\Pi_{\gamma,n}(S^{\mi_b}(\Ext_n(\overline{W}_y)))\right)= \Phi_b^{-1}\left(\Pi_{\gamma,n}(S^{\mi_b}(\tilde \rho(y)))\right).
 	\end{equation}
 	This argument shows that for $y \in Z_n$, we can use the value of $\tilde \rho(y)\mid_{\mi_b +F_{\lfloor \alpha(\gamma) n\rfloor}}$ to uniquely recover  $\Upsilon_{n,n_0}(y)$ (in a Borel manner).
 	Once we have ``recovered'' $\Upsilon_{n,n_0}(y)$ from $\tilde \rho(y)\mid_{\mi_b +F_{\lfloor \alpha(\gamma) n\rfloor}}$, using   \eqref{eq:tilde_rho_traces_rho} together with the fact that  $\rho$ is a $n_0$-towerable and that $C_{n_0}$ is $(\epsilon_{n_0},F_{n_0})$-separated, 
  we can similarly ``read-off''  	$\rho(y) \mid_{\mj + F_{n_0}}$ from  $\tilde \rho(y)\mid_{\mj +F_{n_0}}$
  for every $j \in \Upsilon_{n,n_0}(y) \cap (1-2\delta_n)F_n \setminus \gamma F_n$.

 	Because $\rho$ is $n_0$-towerable, for every $y \in Y$, if we know $\rho(y)\mid_{\mj + F_{n_0}}$ for every $\mj \in \Upsilon_{n,n_0}(y) \cap   (1-2\delta_n)F_n \setminus \gamma F_{n}$ then $\rho(y)\mid_{ (1-2\delta_n)F_n \setminus \gamma F_{n}}$ is trivially determined because $\rho(y)_\mi = x_\star$ for all remaining $\mi \in (1-2\delta_n)F_n \setminus F_{\gamma_n}$.
 	The above argument explicitly describes a function $\hat \Psi: C_n \to X^{(1-2\delta_n)F_n \setminus\gamma F_{n}}$ so that for any $y \in Z_n$, $\hat \Psi(\tilde \rho (y)_{\vec{0}}) = \rho(y) \mid_{(1-2\delta_n)F_n \setminus \gamma F_{n}}$.
 
 	Now let $\Pi'_{n,\gamma}:X^{\ZD} \to C_{n_0}^{K_{n,n_0,\beta(\gamma)}}$ be a Borel function that satisfies
 	\begin{equation}\label{eq:Pi_gamma_prime_def}
\max_{\mj \in K_{n,n_0,\beta(\gamma)}} d_X^{F_{n_0}}(S^{\mj}(w),\Pi'_{n,\gamma}(w)_\mj) = \min\left\{\max_{\mj \in K_{n,n_0,\beta(\gamma)}} d_X^{F_{n_0}}(S^{\mj}(w),w'_{\mj}):~ w' \in C_{n_0}^{K_{n,n_0,\beta(\gamma)}}\right\}.
 	\end{equation}
 	It follows that for $y \in Y_0$,
 	\begin{equation}\label{eq:recover_P_k_from_Phi}
 	\mathcal{P}_k^{F_n}(y) = \Phi_{\hat \Psi (\tilde \rho(y)_{\m 0})}^{-1}\left( \Pi'_{n,\gamma}(\tilde \rho(y))\right).
 	\end{equation}
 	This shows that \eqref{eq:P_k_deteremined_by_tilde_rho_on_Y_0} holds.
 \end{proof}

Combining everything we proved so far gives the following:
\begin{prop}\label{prop:flex_implies_ergodic_univesality}
Let $(X,S)$ be a topological $\ZD$ dynamical system and $\Y=(Y,T)$ a Borel $\ZD$ dynamical system. If $\C=(C_n)_{n=1}^\infty \in X^{\NN}$  is a flexible sequence, $\mu \in \Prob_e(\Y)$	and $h_\mu(\Y) < h(\C)$ then there exists a Borel $T$-invariant $Y_\infty \subset Y$ with $\mu(Y_\infty)=1$  and an injective equivariant Borel embedding $\rho:Y_\infty \to X$.
\end{prop}

\begin{proof}
		For every $j$, let $k_j$ and $N_j$ be numbers obtained as $k_0,N_0$ by applying Lemma \ref{lem:improve_approx_emb} with $\gamma=\frac{1}{2^{j+1}}$. Without loss of generality we can assume that the sequence $k_j$ is increasing and $N_j>2^j$ for all $j \in \N$. 
	We will inductively construct a sequence of natural numbers $(n_j)_{j=1}^\infty$ with $n_j \ge N_j$ and a sequence $(\rho_j)_{j= 1}^\infty \in \Mor(\Y,\X)^\NN$  so that $n_{j} \ll n_{j+1}$ 
	and \eqref{eq:rho_j_shadows_rho_j_prev} holds for every $j$.
	To start the induction apply Lemma \ref{lem:approximate_embedding_exists} with $\epsilon =  \frac{1}{N_1}$ and  $k = k_1$.
	Let $n_1=N_1$.  
	 Let $\rho_1 \in \Mor(\Y,\X)$ be the resulting $(k_1,n_1,\frac{1}{N_1},\mu)$-approximate embedding.
	
	For the induction step, suppose  $n_j$ and $\rho_j$ have been defined for a fixed $j \in \NN$.
	Apply Lemma \ref{lem:improve_approx_emb} with $\gamma= \frac{1}{2^{j+1}}$, $\rho=\rho_j$, $k=k_{j+1}$, $\delta= \frac{1}{2^{j+2}}$.
    Let $\tilde N_j$ be the resulting number $N$. Define $n_{j+1}$ to be the smallest integer greater or equal to $\max\{\tilde N_j,N_{j+1}\}$ that also satisfies $n_{j} \ll n_{j+1}$.
      Then apply Lemma \ref{lem:improve_approx_emb} and let $\rho_{j+1}$ be the resulting $(k_{j+1},n_{j+1},\frac{1}{N_{j+1}},\mu)$-approximate embedding.
      We have
      $$
      D_{n_j,\frac{1}{4}\epsilon_{n_j}}[\rho_j,\rho_{j+1}]\subseteq  Y \setminus T^{(1-2\delta_{n_{j+1}})F_{n_{j+1}} \setminus 2^{-j-1} F_{n_{j+1}}} Z_{n_{j+1}}.
      $$
      Using Lemma \ref{lem:eq:mu_tower_subset_cond} it follows that 
 \eqref{eq:rho_j_shadows_rho_j_prev} holds.
      
	By Lemma \ref{lem:convergence_set_for_rho_j_s} and Lemma \ref{lem:mu_Y_0_1} the limit $\rho=\lim_{n\to\infty}\rho_j$ exist on the set $Y_\infty$ given by \eqref{eq:Y_infty_def}, and satisfies the statement of the lemma.

\end{proof}
\subsection{Full universality - Realizing ergodic measures with full support}

In this section we will prove the following result which deals with full ergodic universality:
\begin{prop}\label{prop:full_universality}
	Under the assumptions of Theorem \ref{thm:spec_sequence_implies_univesality},
	if in addition every $x \in X$ is an accumulation point of  $\bigcup_{n=1}^\infty C_n$ then $(X,S)$ is fully ergodic $h(\C)$-universal.
\end{prop}
This follows by a slight modification of Lemma \ref{lem:improve_approx_emb}:
\begin{lem}\label{lem:improve_approx_emb_with_something_in_support}
	In the statement of Lemma \ref{lem:improve_approx_emb}, for every  $x_0 \in C_{n_0}$ it is possible to 
	arrange that the  resulting  $(k,n,\delta,\mu)$-approximate embedding $\tilde \rho \in \Mor(\Y,\X)$ will have the additional  property that for some $\mi_0 \in F_n$
	\begin{equation}\label{eq:tilde_rho_support_x}
	\forall y \in T^{-\mi_0}Z_n, ~ d_X(\tilde \rho(y),x_0) < \epsilon_{n_0}.
	\end{equation}
\end{lem}
\begin{proof}
	In the proof of Lemma \ref{lem:improve_approx_emb} (with the standing assumptions that $n$ is large enough) given 
	 $\mi_b \in F_n$ that satisfies \eqref{eq:mi_b_disjoint} and \eqref{eq:mi_1_mi_2_mi_3_subset_gamma}, 
	 we can find $\mi_0 \in F_n$ such that
	$$\mi + (1+\delta_{n_0})F_{n_0} \subset (1-2\delta_n)\gamma F_{ n}$$
	and also so that $\mi_0 + (1+\delta_{n_0})F_{n_0}$ is disjoint from 
	$$\left(\mi_b + K_{n,n_0,\alpha(\gamma)}+ (1+\delta_{n_0})F_{n_0}\right)  \cup\left( 
	K_{n,n_0,\beta(\gamma)} +  (1+\delta_{n_0})F_{n_0} \right).$$
	Then proceed exactly as in the proof of Lemma \ref{lem:improve_approx_emb}: Define for $y \in Y$, $\overline{W}_y \in \Inter_n(\C)$ by \eqref{eq:overline_W_y_def}.
	Then for each $y \in Y$ let $\overline{K}_y \subset  (1-\delta_n)F_n$ be given by \eqref{eq:overline_K_y} as in the proof of Lemma \ref{lem:improve_approx_emb}, and let $\overline{K}_y^*= \overline{K}_y \cup \{\mi_0\}$.
	Then $\overline{K}_y^* \subset (1-\delta_n)F_n$ is also $(1+\delta_{n_0})F_{n_0}$-spaced.
	Extend $\overline{W}_y$ to 
	$\overline{W}_y^* \in C_{n_0}^{\overline{K}_y^*}$ by defining 
	\begin{equation*}
	(\overline{W}_y^*)_{\mi}  = 
	\begin{cases}
	x_0 & \mi=\mi_{0}\\
	(\overline{W}_y)_{\mi} & \mi \in \overline{K}_y
	\end{cases},
	\end{equation*}
	\begin{equation}
	\tilde \phi(y)  = \Ext_n(\overline{W}_y^{*}),
	\end{equation}
	and $\tilde \rho =\rho_{\tilde \phi,n}$.
	It follows that the resulting $\tilde \rho \in \Mor(\X,\Y)$ satisfies \eqref{eq:tilde_rho_support_x}.
\end{proof}
\begin{proof}[Proof of Proposition \ref{prop:full_universality}]
	Given $\mu \in \Prob_e(\Y,T)$ we need to exhibit $\rho \in \Mor(\Y,\X)$ that induces an embedding of $(Y,T,\mu)$ into $(X,S)$ and  such that the closed support of $\mu \circ \rho^{-1}$ is $X$.
	
	For each $n$ write $C_n= \{x_{1,n},\ldots, x_{|C_n|,n}\}$, and let  
	$\left(x_j\right)_{j=1}^\infty$ be the concatenation of the lists $(x_{1,n},\ldots, x_{|C_n|,n})$:
	$$\left(x_j\right)_{j=1}^\infty = (x_{1,1}\ldots,x_{|C_1|,1},\ldots,x_{1,n},\ldots,x_{|C_n|,n},x_{1,n+1},\ldots).$$
    Thus, $\bigcup_{n=1}^\infty C_n= \{x_j\}_{j=1}^\infty$.
     For every $j \in \NN$ let $w_j \in C_{k_j}^{\{0\}}$ be given by $w_j(\vec{0})=x_j$. 
	 Using Lemma \ref{lem:improve_approx_emb_with_something_in_support} we can construct a sequence  $(\rho_j)_{j= 1}^\infty \in \Mor(\Y,\X)^\NN$ as we did in the proof of Proposition \ref{prop:flex_implies_ergodic_univesality}, with the additional feature that for every $j \in \NN$ there exists $\mi_j \in (1-\delta_{n_j})F_{n_j}$ so that for any  $y \in T^{-\mi_j}Z_{n_j}$, $d_X(\rho_j(y),\tilde x_j) < \epsilon_{j}$, where $\tilde x_j = \Ext_{n_j}(w_j)$.
	It follows that there is a Borel set $Y_\infty \subset Y$ with $\mu(Y_\infty)=1$ so that the limit $\rho= \lim_{j\to \infty}\rho_j$ exists on $Y_\infty$ and is an  injective equivariant Borel embedding $\rho:Y_\infty \to X$, and in addition for every $j \in \NN$,
	$$\mu\left(\left\{y \in Y :~ d(\rho(y), x_j) < \frac{1}{2}\epsilon_{j} \right\}\right) > 0.$$
	By the assumption that every $x \in X$ is an accumulation point of  $\bigcup_{n=1}^\infty C_n$, for every $x \in X$ and $\epsilon >0$ there exists $j \in \NN$ such that the $(\epsilon,d_X)$-ball  centered at $x$ contains the $(2\epsilon_j,d_X)$-ball centered at $x_j$. This proves that $\mu\circ \rho^{-1}(U)>0$ for every non-empty open subset of $U \subseteq X$.
\end{proof}

\section{Almost Borel Universality}\label{sec:almost_borel_universality}
In this section we prove Theorem \ref{thm:spec_sequence_implies_univesality}. This follows the basic strategy  of the previous section and relies on it. 

There are additional complications in certain steps: We use the same procedure to construct a converging sequence of approximate  embeddings for every measure $\mu \in \Prob_e(\Y)$, but make sure that the dependence on the measure $\mu$ is ``Borel measurable''. For ``almost all'' points $y \in Y$ that are generic with respect to $\mu$ we apply the embedding that corresponds to $\mu$. In doing so, we need is to  make sure that generic points corresponding to different measures do not get mapped to the same point in $X$. 
For this we apply Lemma \ref{lem:trace_and_recover_Q} below, which in some sense is responsible for  ``encoding the measure $\mu$''  into $X$. 

For the record, we mention the following corollary of  Theorem \ref{thm:spec_sequence_implies_univesality}:
\begin{prop}\label{prop:flex_factor}
	Let $(X,S)$ be topological $\ZD$ dynamical system with a flexible marker sequence $\C$, and let $\Y=(Y,T)$ be a  free  Borel $\ZD$ dynamical system. 
	Then there exists a $T$-invariant Borel subset $Y_0 \subset Y$ so that $Y \setminus Y_0$ is null with respect to any $T$-invariant probability measure and a Borel equivariant  map from  $Y_0$ into $X$.
\end{prop}
\begin{proof}
	This is  a direct consequence of Theorem \ref{thm:spec_sequence_implies_univesality} together with the observation that for any $\epsilon >0$, any free Borel dynamical system admits  a free Borel factor  having  Gurevich entropy less than $\epsilon$.
	To see this, note that the factor generated by a sequence of $(F_n,\epsilon_n)$-towers is always essentially free and has small entropy provided that the sequence $(|F_n|)_{n=1}^\infty$ grows rapidly enough and that the sequence $(\epsilon_n)_{n=1}^\infty$ decays fast enough. More generally, the existence of free factors having small entropy for actions of arbitrary countable groups follows rather directly from \cite[Theorem $1.1$]{MR3413491}.  
\end{proof}
We note that because there is no injectivity requirement, one can prove  Proposition \ref{prop:flex_factor} directly and much more easily than  Theorem \ref{thm:spec_sequence_implies_univesality}, along the lines of the proof of  Theorem \ref{thm:infinite_entropy}.

\subsection{Borel setup} 
It will be convenient for us to use the following  (essentially trivial)  ``universal Borel embedding'' result with respect to the shift over $(\{0,1\}^{\NN})^{\ZD}$:
\begin{prop}\label{prop:Borel_universaility_of_shift_on_cantor_Set}
Any standard Borel dynamical system is isomorphic to a Borel subsystem of the shift over the Cantor set:
Let $\Y=(Y,T)$ be a Borel $\ZD$ dynamical system.
Then there is a Borel  embedding $\rho:Y \to (\{0,1\}^{\NN})^{\ZD}$  that is equivariant in the sense that
$$\rho(T^{\mi}(y))_{\vec{0}} =  \rho(y)_{\mi} \mbox{ for every } \mi \in \ZD,~ y \in Y.$$
\end{prop}
\begin{proof}
Let $\mathcal{A}= \{A_1,\ldots, A_n,\ldots\} \subset \mathit{Borel}(Y)$ be a countable sequence of Borel sets that generate the Borel $\sigma$-algebra of $Y$.
Define $\rho:Y \to  (\{0,1\}^{\NN})^{\ZD}$ by
$$ \rho(y)_{\mi}(n)= 1_{A_n}(T^{\mi}(y)),~ \mi \in \ZD, n \in \NN.$$
Clearly $\rho$ is a Borel function. Injectivity of $\rho$ follows because the elements of $\mathcal{A}$ separate points.
\end{proof}

Proposition \ref{prop:Borel_universaility_of_shift_on_cantor_Set} is an easy and well known starting point for ``Borel dynamics''.
 It clearly  generalizes verbatim to actions of arbitrary countable groups. 
 For free $\ZZ$-actions, Proposition \ref{prop:Borel_universaility_of_shift_on_cantor_Set} is a direct consequence of the existence of a countable generator. This has a  short but  non-trivial proof due to Weiss \cite{MR737417}.

Using Proposition  \ref{prop:Borel_universaility_of_shift_on_cantor_Set}
 above, from now on we will identify $Y$ with a shift-invariant Borel subset of $(\{0,1\}^{\NN})^{\ZD}$, and assume that the $\ZD$-action $T$ on $Y$ is the restriction of the shift on $(\{0,1\}^{\NN})^{\ZD}$. Thus $Y$ inherits the relative topology from $(\{0,1\}^{\NN})^{\ZD}$, compatible with its Borel structure. We denote by \index{Definitions and notation introduced in Section 7!$C(Y, X)$}$C(Y,X)$ the space of continuous functions from $Y$ to $X$. The space $C(Y,X)$ is a standard Borel space. 
 Let \index{Definitions and notation introduced in Section 7!$\Mor_C(\Y, \X)$}$\Mor_C(\Y,\X)$ denote the space of continuous morphisms from $\Y=(Y,T)$ to $\X=(X^{\ZD},S)$. There is an obvious bijection between $\Mor_C(\Y,\X)$ and $C(Y,X)$ that gives $\Mor_C(\Y,\X)$  a standard Borel structure.
 We denote by \index{Definitions and notation introduced in Section 7!$\Clopen(Y)$}$\Clopen(Y)$ the collection of subsets of $Y$ 
 obtained by intersecting a clopen subset of $(\{0,1\}^{\NN})^{\ZD}$ with $Y$. This is a countable set (as there are countably many clopen subsets of $(\{0,1\}^{\NN})^{\ZD}$).
  We consider the space $\Prob_e(\Y) \subset \Prob\left((\{0,1\}^{\NN})^{\ZD}\right)$, again with the Borel structure that comes from the  weak-$*$ topology.
 Clearly, we can choose the Borel embedding of $Y$ into $(\{0,1\}^{\NN})^{\ZD}$ so that with respect to the inherited  topology on $Y$ each of the partitions $\mathcal{P}_k$ consists of clopen sets. We will assume this from now on.  This is to ensure  continuity of certain functions related to approximate embeddings that appear in the proof. 
We call a point $y \in Y$ \emph{generic}\index{Definitions and notation introduced in Section 7!generic points} if the sequence of probability measures
\begin{equation}\label{eq:generic_seq}
\frac{1}{|F_n|}\sum_{\mi \in F_n}\delta_{T^{\mi}(y)}
\end{equation}
converge in the weak-$*$ topology to an ergodic measure $\mu \in \Prob_e(\Y)$. 
Let $G(\Y) \subset Y$ denote the set of generic points for $\Y$.
In this case we refer to the measure
\begin{equation}
\index{Definitions and notation introduced in Section 7!$\mu_y$}\mu_y = \lim_{n \to \infty}\frac{1}{|F_n|}\sum_{\mi \in F_n}\delta_{T^\mi(y)}
\end{equation}
as the \emph{empirical measure} of $y$. 
Note that $Y$,  viewed as a subset of $\{0,1\}^{\ZD}$, is not assumed to be compact, so for general $y \in Y$ a limit point of the sequence of measures in \eqref{eq:generic_seq} need not be supported on $Y$. 
We also denote by 
\begin{equation}
\mathit{emp}:G(\Y) \to \Prob_e(\Y)\index{Definitions and notation introduced in Section 7!$\mathit{emp}:G(\Y)\to \Prob_e(\Y)$}
\end{equation}
the map that sends a generic point $y \in G(\Y)$ to its empirical measure $\mu_y$.
Namely
\begin{equation}
\mathit{emp}(y) = \mu_y
\end{equation}
For $\mu \in \Prob_e(\Y)$, let
\begin{equation}
G_\mu(\Y) = \mathit{emp}^{-1}(\{\mu\}).\index{Definitions and notation introduced in Section 7!$G_\mu(\Y)$}
\end{equation}
The set  $G_\mu(\Y)$ is the collection of $\mu$-generic points in $Y$.
We have $$G(\Y)= \bigcup_{\mu \in \Prob_e(Y,T)}G_\mu(Y,T).\index{Definitions and notation introduced in Section 7!$G(\Y)$}$$
For later reference, we record some Borel measurably results about generic points and empirical measures.
\begin{prop}\label{prop:measurability_of_generic_points}
	\begin{enumerate}
		
		\item The set $G(\Y)$ is a Borel subset of $Y$.
		\item The set $\Prob(\Y)$ is a Borel subset of $\Prob\left((\{0,1\}^\NN)^{\ZD} \right)$.
		\item The set $\Prob_e(\Y)$ is a Borel subset of $\Prob(\Y)$.
		\item For every $\mu \in \Prob_e(\Y)$ the set $G_\mu(\Y)$ is a Borel subset of $Y$.
		\item The function $\mathit{emp}:G(\Y) \to \Prob_e(\Y)$ is a Borel measurable function.
	\end{enumerate}
\end{prop}

These statements are all  standard and well known, so we omit the proof, referring for instance to the discussion of generic points in \cite{MR3077948}.

 We  now give a short proof of Proposition \ref{prop:borel_ZD_rokhlin_lemma} regarding the existence of ``Borel'' Rokhlin towers:
\begin{proof}[Proof of Proposition \ref{prop:borel_ZD_rokhlin_lemma}]
Choose $n \in \NN$ and $\epsilon >0$.
Using the ``usual'' version of Rokhlin's lemma (Proposition \ref{prop:rokhlin_lemma}), for every $\mu \in \Prob_e(Y,T) \subset \Prob_e\left((\{0,1\}^{\NN})^{\ZD}\right)$, let $Z'_{\mu} \subset Y$ be the base an $(F_n, \epsilon/2,\mu)$-tower.
By inner regularity of $\mu$, we can find a closed (hence compact) set $Z''_{\mu}\subset (\{0,1\}^{\NN})^{\ZD}$ 
so that $Z''_{\mu}  \subset Z'_\mu$ and 
$$\mu(Z'_\mu \setminus Z''_\mu) < |F_n|^{-1}\epsilon/2.$$
Then $Z''_{\mu}$ is an $(F_n, \epsilon,\mu)$-tower.
Because the sets $\{T^{-\mi}Z''_{\mu}\}_{\mi \in F_n}$ are pairwise disjoint compact subsets and the topology of $Y$  has  a clopen basis, we can find a clopen set $Z_\mu \subset (\{0,1\}^{\NN})^{\ZD}$ that contains $Z''_\mu$ and so that $(T^{-\mi}Z_{\mu})_{\mi \in F_n}$ are still pairwise disjoint. We conclude that for every $\mu \in \Prob_e(Y,T)$ there exists a clopen set  $Z_\mu \subset (\{0,1\}^{\NN})^{\ZD}$ that is  also the base of an $(F_n, \epsilon,\mu)$-tower. 
Our next goal is to show that we can furthermore choose $Z_\mu$ as above so that the  function $\mu \mapsto Z_\mu$ will be  Borel measurable as a function from $\Prob_e(\Y)$ to the clopen sets of $(\{0,1\}^{\NN})^{\ZD}$.
 Let $(E_k)_{k=1}^\infty$ be some enumeration of the clopen subsets of $ (\{0,1\}^{\NN})^{\ZD}$. For every $\mu \in \Prob_e(Y,T)$ let $Z_\mu = E_{k_0}$ if $E_{k_0}$ is the base of an $(F_n, \epsilon,\mu)$-tower and for every $k < k_0$ the clopen set $E_k$ is not the base of an $(F_n, \epsilon,\mu)$-tower.
Then using Proposition \ref{prop:measurability_of_generic_points} it is not difficult to check that for every $k \in \NN$ the set
$$U_k = \left\{\mu \in \Prob_e(\Y):~ Z_\mu = E_k \right\}$$
is a Borel subset of $\Prob_e(\Y)$. It follows that the resulting function $\mu \mapsto Z_\mu$ is indeed a Borel measurable function.
 Define 
$$Z = \bigcup_{k=1}^\infty \mathit{emp}^{-1}(U_k) \cap E_k.$$
From the fact that the $U_k$'s are Borel subsets of $\Prob_e(\Y)$, using   Proposition \ref{prop:measurability_of_generic_points} again it follows that $Z$ is a Borel subset of $Y$. One can check that 
$$Z= \bigcup_{\mu \in \Prob_e(\Y)}\left(Z_\mu \cap G_\mu(\Y) \right).$$

Thus is clear that for every $\mu \in \Prob_e(\Y)$,
$$\mu(\bigcup_{\mi \in F_n}T^{-\mi} Z)= \mu(\bigcup_{\mi \in F_n}T^{-\mi} Z_\mu) \ge 1 -\epsilon.$$
Thus $Z$ is the base of an $(F_n,\epsilon,\mu)$-tower for every $\mu \in \Prob_e(\Y)$.
\end{proof}

From now on we will  further assume that for every $n \in \NN$ the set  $Z_n \subset Y$ is clopen in $Y$. We can assume this without any loss of generality by modifying the embedding of $\Y$ into the shift over $\{0,1\}^\NN$.

For $y \in Y$ and $A \subset Y$ define:
\begin{equation}\label{eq:overline_d_def}
\index{Definitions and notation introduced in Section 7!$\udense_{\Y}(y, A)$}\udense_{\Y}(y,A)= \limsup_{n \to \infty}\frac{1}{|F_n|}\left|\left\{\mi \in F_n:~ T^{\mi}(y) \in A\right\}\right|,
\end{equation}
and
\begin{equation}\label{eq:underline_d_def}
\index{Definitions and notation introduced in Section 7!$\ldens_{\Y}(y, A)$}\ldens_{\Y}(y,A)= \liminf_{n \to \infty}\frac{1}{|F_n|}\left|\left\{\mi \in F_n:~ T^{\mi}(y) \in A\right\}\right|.
\end{equation}
Viewing  $\Prob_e(\Y)$ as a Borel subset of $\Prob(\{0,1\}^{\ZD})$, it is a standard Borel space.
So we can find a sequence  $(\mathcal{Q}_n)_{n=1}^\infty$\index{Definitions and notation introduced in Section 7!$(\mathcal{Q}_n)_{n=1}^\infty$} of finite Borel partitions that together generate the Borel $\sigma$-algebra of $\Prob_e(\Y)$, and in particular separate points.
We further assume that each partition in the sequence refines the previous one and that $|\mathcal{Q}_r| \le |C_{r}|$ for every $r$.

\subsection{Main lemmas and proof of Theorem \ref{thm:spec_sequence_implies_univesality}}
In the following we will frequently need the fact that if $n_i$ is a sequence of integers such that $n_1\ll n_2...$ then $\epsilon_{n_i}$ and $\delta_{n_i}$ are decreasing at an exponential rate. For instance this would imply that
$$\sum_{t=1}^\infty\frac38\epsilon_{n_t}<3/7\epsilon_{n_1}.$$

Our proof of Theorem \ref{thm:spec_sequence_implies_univesality} follows a similar structure as the proof of Proposition \ref{prop:flex_implies_ergodic_univesality} in the previous section, but will be somewhat more involved. We will use two main lemmas. The  easier one is a refinement of Lemma \ref{lem:approximate_embedding_exists}. It asserts that we can produce a  $(k,n,\epsilon,\mu)$-approximate embedding $\rho \in \Mor(\Y,\X)$, provided that $n$ is sufficiently big, as a Borel function of $k \in \NN$, $\mu \in \Prob_e(\Y)$ and $\epsilon >0$:
\begin{lem}\label{lem:borel_appproximate_embedding}
	
	There exists a Borel function 
	$$N_0 : \Prob_e(\Y) \times \NN \times (0,1) \to \NN,$$
	a Borel function 
	$$\Phi_0: \Prob_e(\Y) \times \NN \times (0,1) \times  \NN \to \Mor_C(\Y,\X),$$
	and a Borel function
	$$\overline{Z}: \Prob_e(\Y) \times \NN \times (0,1) \times  \NN \to \Clopen(Y),$$
	so that the following holds:
	
	Fix $\mu \in \Prob_e(\Y)$,  $k,n \in \NN$ and $\epsilon >0$ such that 
	$$n > N_0(\mu,k,\epsilon).$$\index{Definitions and notation introduced in Section 7!$N_0(\mu,k,\epsilon)$}
	Denote
	$$\rho = \Phi_0(\mu,k,\epsilon,n). $$\index{Definitions and notation introduced in Section 7!$\Phi_0(\mu,k,\epsilon,n)$}
	Then $\rho \in \Mor(\Y,\X)$ is a $(k,n,\epsilon,\mu)$-approximate embedding and we can choose 
	$$Z[\rho] = \overline{Z}(\mu,k,\epsilon,n).\index{Definitions and notation introduced in Section 7!$\overline{Z}(\mu,k,\epsilon,n)$}$$
\end{lem} 

The following lemma is an extension of Lemma \ref{lem:improve_approx_emb}. Recall that Lemma \ref{lem:improve_approx_emb} 
says that there is a procedure that given a $(k_0, n_0, \delta_0, \mu)$-approximate embedding $\rho$ and parameters $k,n,\delta$ produces a $(k, n, \delta, \mu)$-approximate embedding $\tilde \rho$ which is  ``close'' to $\rho$ in certain sense, as long as the parameters $n_0,k_0,\delta_0$ are ``sufficiently good''. In the lemma below the ``distance'' between $\rho$ and $\tilde \rho$ will be captured by an additional parameter $\gamma$. The lemma says that $\tilde \rho$ can be generated as a Borel function of the ``input parameters'' $(\mu,k,n,\delta,\gamma,\rho)$. Furthermore, the lemma below asserts that we can arrange that for a generic $y$, the image  $\tilde \rho(y) \in X^{\ZD}$ will  ``approximately encode the measure $\mu \in \Prob_e(\Y)$'', in the sense that it will be possible to recover the partition element $\mathcal{Q}_r(\mu)$ from $\tilde \rho(y)$, for $r$'s in a certain range. There is a technical requirement that the parameter $n$ takes values only in a certain subsets of the integers $\{ \tilde q(n) : n \in \mathbb{N}\}$ where $\tilde q:\mathbb{N} \to \mathbb{N}$ is some rapidly growing function.

\begin{lem}\label{lem:improve_approx_emb_with_measure_encode}

	Given  functions $q,\tilde q:\NN \to \NN$   so that 
	\begin{equation}\label{eq:q_r_rapid}
	1 \ll q(1) \ll \tilde q(1) \ll q(2) \ll \tilde q(2) \ll \ldots \ll q(n) \ll\tilde q(n) \ll q(n+1) \ll \ldots,
	\end{equation}  \index{Definitions and notation introduced in Section 7!$q(\N)$ and $\tilde q(\N)$}
	there exist   Borel functions
		\begin{equation}\label{eq:overline_Phi_type}
\overline{\Phi}:\Prob_e(\Y)\times \NN \times \NN \times (0,1) \times (0,1) \times \Mor_C(\Y,\X) \to
\Mor_C(\Y,X),
\end{equation}
 $$\overline{\overline{Z}}:\Prob_e(\Y)\times \NN \times \NN \times (0,1) \times (0,1)  \times \Mor_C(\Y,\X) \to
\Clopen(Y),
$$
 $$\tilde{k}_0:\Prob_e(\Y)\times (0,1) \to \NN,$$
$$\tilde{N}_0:\Prob_e(\Y)\times (0,1)  \times \NN \to \NN,$$
	 $$	\tilde{N}:\Prob_e(\Y) \times (0,1) \times \Mor_C(\Y,\X) \times (0,1) \times \NN \to \NN, $$
  and a partition 
   $\{C_{r,Q}\}_{Q \in \mathcal{Q}_r}$\index{Definitions and notation introduced in Section 7!$C_{r, Q}$} of $C_{q(r)}$ for every $r \in \NN$, so that  whenever $\mu \in \Prob_e(\Y)$, $\gamma \in (0,1)$, $\delta_0,\delta >0$, $k, k_0,n,n_0 \in \NN$ and   $\rho \in \Mor_C(\Y,\X)$ is a continuous $(k_0,n_0,\delta_0,\mu)$-approximate embedding so that
\begin{equation}\label{eq:cond_delta_0_n_0_k_0}
	\delta_0 \le \frac{1}{\tilde{N}_0(\mu,\gamma, k_0)},~ k_0 \ge \tilde{k}_0(\mu,\gamma), \mbox{ and }	
	n_0 \ge \tilde{N}_0(\mu,\gamma,k_0)
	\end{equation} \index{Definitions and notation introduced in Section 7!$\tilde{k}_0(\mu,\gamma)$}\index{Definitions and notation introduced in Section 7!$\tilde{N}_0(\mu, \gamma,k_0)$}
	and
	\begin{equation}\label{eq:cond_n_tilde_q}
	 n \ge \tilde{N}(\mu,\gamma,\rho,\delta,k) \mbox{ and } n\in \tilde{q}(\NN),\index{Definitions and notation introduced in Section 7!$\tilde{N}(\mu,\gamma,\rho,\delta,k) $}
	\end{equation}
	then 
	\begin{equation}\label{eq:tilde_rho_def}
	\tilde \rho = \overline{\Phi}(\mu,k,n,\delta,\gamma,\rho)\index{Definitions and notation introduced in Section 7!$\overline{\Phi}(\mu,k,n,\delta,\gamma,\rho)$}
	\end{equation}
	is a  $(k,n,\delta,\mu)$-approximate embedding
	so that 
	\begin{equation}\label{eq:D_rho_tilde_rho}
	\mu\left(D_{n_0,\frac{3}{8}\epsilon_{n_0}}[\rho,\tilde \rho] \mid Z_{n_0}\right)< \gamma,
	\end{equation}
and whenever $n_0 \ll q(r) \ll n$ we have 
	\begin{equation}\label{eq:y_C_r_mu_on_U_r}
	\mu \left(\left\{y \in Z_{q(r)}:~ d_X^{F_{q(r)}}(\tilde \rho(y),C_{r,\mathcal{Q}_r(\mu)}) \ge \frac{3}{8}\epsilon_{q(r)}\right\}\mid Z_{q(r)} \right)
<  \gamma,
	\end{equation}
	and	we can take
	\begin{equation}\label{eq:Z_tilde_rho}
	Z[\tilde \rho]= \overline{\overline{Z}}(\mu,k,n,\delta,\gamma,\rho).\index{Definitions and notation introduced in Section 7!$\overline{\overline{Z}}(\mu,k,n,\delta,\gamma,\rho)$}
	\end{equation}
\end{lem}

	The proof of Theorem  \ref{thm:spec_sequence_implies_univesality} is similar to the proof of Proposition \ref{prop:flex_implies_ergodic_univesality}. Like in the proof of Proposition \ref{prop:flex_implies_ergodic_univesality}, we produce a sequence of better and better approximate embeddings $\rho_{n,\mu}$ for each measure $\mu$, except that now our approximate embeddings are also a Borel function of the measure $\mu$. These are produced in an inductive manner where we start the induction using Lemma \ref{lem:borel_appproximate_embedding} and apply the induction step using Lemma \ref{lem:improve_approx_emb_with_measure_encode}. The convergence of the maps  $\rho_{n,\mu}$ to embeddings $\rho_\mu$ also follows as in Proposition \ref{prop:flex_implies_ergodic_univesality}. However now we need to further worry about the fact that for two distinct measures $\mu$ and $\nu$, the maps $\rho_{ \mu}$ and $\rho_{\nu}$ might be close to each other. For this we use the fact that in our induction step, our approximate embeddings $\rho_{n, \mu}$ also encode the measure. In particular if $\QQ_{r}(\mu)$ and $\QQ_r(\nu)$ for some $r$ are distinct, the embeddings $\rho_{n,\mu}$ and $\rho_{n,\nu}$ (implying that also $\rho_{\mu}$ and $\rho_{\nu}$) are far from each other.

\begin{proof}[Proof of Theorem \ref{thm:spec_sequence_implies_univesality} assuming Lemma \ref{lem:borel_appproximate_embedding} and Lemma  \ref{lem:improve_approx_emb_with_measure_encode}]

	We follow a similar scheme as in the proof of Proposition \ref{prop:flex_implies_ergodic_univesality}, this time using Lemma \ref{lem:borel_appproximate_embedding} instead of Lemma \ref{lem:approximate_embedding_exists} and Lemma \ref{lem:improve_approx_emb_with_measure_encode} instead of Lemma \ref{lem:improve_approx_emb}.
	
	We fix a rapidly decaying sequence $(\gamma_j)_{j=1}^\infty$ by $\gamma_j = \frac{1}{10^{j+1}}$. For every $j \in \NN$ and $\mu \in \Prob_e(\Y)$, define 
	$$ \overline{k}_{j,\mu}=\max\{\tilde {k}_0(\mu,\gamma_{j+1}),\overline{k}_{j-1, \mu}+1\} \mbox{ and } { \overline{N}_{j,\mu}=\tilde{N}_0(\mu,\gamma_{j+1},\overline{k}_{j,\mu})},$$
	$$\delta_{j,\mu}= \frac{1}{2{ \tilde{N}_0(\mu,\gamma_{j+1},\overline{k}_{j,\mu})}},$$
	where {$\tilde{k}_0:\Prob_e(\Y)\times(0,1) \to \NN$, $\tilde{N}_0:\Prob_e(\Y)\times(0,1) \times \NN \to \NN$} are the Borel functions that appear in the statement of Lemma \ref{lem:improve_approx_emb_with_measure_encode}.
	By induction on $j \in \NN$ define
	$n_{j,\mu} \in\NN$, $\rho_{j,\mu} \in \Mor_C(\Y,\X)$ and $Z_{j,\mu} \in \Clopen(Y)$ 
	for every $\mu \in \Prob_e(\Y)$	as follows:

	To start the induction, let 
	\begin{equation}
	\tilde {n}_{1,\mu}= \max\{N_0(\mu,\overline{k}_{1,\mu},\delta_{1,\mu}),{\overline{N}_{1,\mu}}\},
	\end{equation}
	\begin{equation}
	n_{1,\mu} = \min \left\{ \tilde q(r):~ r\in \NN ,~\tilde{q}(r) > \tilde n_{1,\mu}\right\},  
	\end{equation}
	\begin{equation}
	\rho_{1,\mu}= \Phi_0\left(\mu,\bar{k}_{1,\mu},\delta_{1,\mu},n_{1,\mu}\right)
	\end{equation}
	and
	\begin{equation}
	Z_{1,\mu}= \overline{Z}(\mu,\bar{k}_{1,\mu},\delta_{1,\mu},n_{1,\mu}),
	\end{equation}
	where $N_0,\Phi_0,\overline{Z}$ are the functions that appear in the statement of Lemma \ref{lem:borel_appproximate_embedding}.

	Assume by induction that these have been defined for some $j  \ge 1$ and all $\mu \in \Prob_e(\Y)$. Define:
	\begin{equation}
	\tilde n_{j+1,\mu}= \tilde {N}\left(\mu,\gamma_{j+1},\rho_{j,\mu},\delta_{j+1,\mu},\overline k_{j,\mu}\right),
	\end{equation}
	\begin{equation}
	n_{j+1,\mu}= \min \left\{ \tilde q(r):~ r\in \NN ,~\tilde{q}(r) \ge \overline N_{j+1,\mu},~\tilde{q}(r) \gg n_{j,\mu} \mbox{ and } 
	~\tilde{q}(r) \ge \tilde n_{j+1,\mu}\right\},    
	\end{equation}
	\begin{equation}
	\rho_{j+1,\mu}= \overline{\Phi}\left(\mu,\bar k_{j+1,\mu},n_{j+1,\mu},\delta_{j+1,\mu},\gamma_{j+1},\rho_{j,\mu}\right),
	\end{equation}
	\begin{equation}
	Z_{j+1,\mu}=\overline{\overline{Z}}\left(\mu,\bar k_{j+1,\mu},n_{j+1,\mu},\delta_{j+1,\mu},\gamma_{j+1},\rho_{j,\mu}\right),
	\end{equation}
	where $\tilde{N}$, $\overline{\Phi}$ and $\overline{\overline{Z}}$ are the functions that appear in the statement of Lemma \ref{lem:improve_approx_emb_with_measure_encode}.
	Then the application of  Lemmas \ref{lem:borel_appproximate_embedding} and \ref{lem:improve_approx_emb_with_measure_encode} ensures that for every $\mu \in \Prob_e(\Y)$, the function $\rho_{j,\mu} \in \Mor_C(\Y,\X)$ is a  continuous $(\overline k_{j,\mu},n_{j,\mu},\delta_{j,\mu},\mu)$-approximate embedding, and that
	\begin{equation}\label{eq:D_n_j_mu}
	\mu\left(D_{n_{j,\mu},\frac{3}{8}\epsilon_{n_{j,\mu}}}[\rho_{j,\mu},\rho_{j+1,\mu}] \mid Z_{n_{j, \mu}}\right) < \gamma_{j}.
	\end{equation}
	
	Then, by Lemma \ref{lem:convergence_set_for_rho_j_s}, for every $\mu \in \Prob_e(\Y)$ the sequence $(\rho_{j,\mu})_{j=1}^\infty$ converges 
	pointwise on the set
	\begin{equation}\label{eq:Y_mu_def}
	Y_\mu = \bigcap_{\mi \in \ZD}T^{-\mi}\bigcup_{j=1}^\infty\bigcap_{t \ge j}T^{(1-2\delta_{n_{t,\mu}})F_{n_{t,\mu}}}\left(Z_{t,\mu}\setminus D_{n_{t,\mu},\frac{3}{8}\epsilon_{n_{t,\mu}}}[\rho_{t,\mu},\rho_{t+1,\mu}]\right),  
	\end{equation}
	and on this set the limit is a Borel embedding into $X$.
	By Lemma \ref{lem:mu_Y_0_1}, it follows that  $\mu(Y_\mu)=1$ for all $\mu \in \Prob_e(\Y)$.
	For each $\mu \in \Prob_{e}(\Y)$ let us denote  the pointwise limit of $(\rho_{j,\mu})_{j=1}^\infty$ by   $\rho_\mu:Y_\mu \to X$.

	By Lemma \ref{lem:D_n_epsilon_D_m_epsilon}  and  \eqref{eq:D_n_j_mu} it follows that for every $r$ such that $q(r) \ll n_{j,\mu}$,
	$$
	\mu\left(D_{q(r),\frac{3}{8}\epsilon_{n_{j,\mu}}}[\rho_{j,\mu},\rho_{j+1,\mu}] \mid Z_{q(r)}\right) <\frac{1}{10^j} .
	$$
	
	So, using the triangle inequality, it follows that whenever $q(r) \ll n_{j,\mu}$ then 
	\begin{equation}\label{eq:D_q_r_rho_mu_j}
	\mu\left(D_{q(r),\frac{3}{7}\epsilon_{n_{j,\mu}}}[\rho_{j,\mu},\rho_{\mu}] \mid Z_{q(r)}\right) <\sum_{t=j}^\infty\mu\left(D_{q(r), \frac38\epsilon_{n_t,\mu}}[\rho_{t, \mu}, \rho_{t+1, \mu}]~|~Z_{q(r)}\right)< \frac{1}{10^{j-1}}=\gamma_{j-2}.
	\end{equation}
	
	Also,  by  the construction of  $\rho_{j,\mu}$ using Lemma \ref{lem:improve_approx_emb_with_measure_encode}
	if $n_{j-1,\mu} \ll q(r) \ll n_{j,\mu}$ then 
	$$\mu \left(\left\{ y \in Y :~ d_X^{F_{q(r)}}\left(\rho_{j,\mu}(y),C_{r,\mathcal{Q}_r(\mu)}\right) < \frac{3}{8}\epsilon_{q(r)}\right\} \mid Z_{q(r)} \right) > 1 - \gamma_j,$$
	where $\{C_{r,Q}\}_{Q \in {\QQ}_r}$	is the partition of $C_{q(r)}$ that appears in the statement of Lemma \ref{lem:improve_approx_emb_with_measure_encode}.

	Since $\frac37\epsilon_{n_{j, \mu}}+ \frac38\epsilon_{q(r)}  <  \frac{5}{12}\epsilon_{q(r)}$ it follows from \eqref{eq:D_q_r_rho_mu_j} that
	\begin{equation}\label{eq:mu_rho_mu__close_to_C_r}
	\mu \left(\left\{ y \in Y :~ d_X^{F_{q(r)}}\left(\rho_{ \mu}(y),C_{r,\mathcal{Q}_r(\mu)}\right) < \frac{5}{12}\epsilon_{q(r)}\right\} \mid Z_{q(r)} \right) > 1- 2\gamma_{j-2}.
	\end{equation}
	Recall the Borel set $\G(\Y)$ consisting of points of $Y$ which are generic for ergodic probability measures. We will be concerned with the subset of $\G(\Y)$ on which the maps $\rho_{j,\mu}$ converge:
	$$Y_{\rightarrow}=\{y\in \G(\Y)~:~ \rho_{j, \mu_y}(y)\text{ converges as } j\to \infty\}.$$
	For every $j \in \mathbb{N}$ the function  $y\to \rho_{j, \mu_y}(y)$ is a composition of Borel functions. Thus $y\to \rho_{j, \mu_y}(y)$ is also a Borel function and hence the set $Y_{\rightarrow}$ is measurable.
	We have already mentioned that the sequence $(\rho_{j,\mu})_{j=1}^\infty$ converges pointwise on $Y_\mu$, so 	$Y_\mu \cap G_\mu\subset Y_\to$ for every $\mu\in \Prob_e(\Y)$.
	Consequently 
	 $\mu(Y_\to)=1$ for all $\mu\in \Prob_e(\Y)$.
	Now consider the Borel maps
	$\rho_j: Y_\to \to X^{\ZD}$ as $\rho_j(y)= \rho_{j , \mu_y}(y)$
	and its limit
	$\rho: Y_\to \to X^{\ZD}$ as
	$\rho(y)=\lim_{j \to \infty} \rho_{j}(y)$.

	For $r \in \NN$ and $Q \in \mathcal{Q}_r$ let
	\begin{equation}
	A_{r,Q} = \left\{y \in Y_\to:~ d_X^{F_{q(r)}}\left(\rho(y),C_{r,Q}\right) < \frac{1}{2}\epsilon_{q(r)} \right\},
	\end{equation}
	
	and
	\begin{equation}
	A_{r} = \left\{y \in Y_\to:~ d_X^{F_{q(r)}}\left(\rho(y),C_{q(r)}\right) < \frac{1}{2}\epsilon_{q(r)} \right\}.
	\end{equation}
	Then the fact that $C_{q(r)}$ is $(\epsilon_{q(r)},F_{q(r)})$-separated implies that $\{A_{r,Q}\}_{Q \in \mathcal{Q}_r}$ are pairwise disjoint and $A_r = \biguplus_{Q \in \mathcal{Q}_r}A_{r,Q}$.
	Because $\rho:Y_\to \to X^{\ZD}$ is a Borel function it follows that $A_r, A_{r,Q} \subset Y_\to$ are Borel sets for every $r \in\NN$ and $Q \in \mathcal{Q}_r$.
	From \eqref{eq:mu_rho_mu__close_to_C_r} it follows that 
	\begin{equation}
	\lim_{r \to \infty}\mu(A_{r,\mathcal{Q}_r(\mu)} \mid Z_{q(r)})=1 ~\forall ~\mu \in \Prob_e(\Y).
	\end{equation}
	Since $\lim_{r \to \infty} |F_{q(r)}| \cdot \mu(Z_{q(r)})=1$, it follows that $\lim_{r \to \infty} |F_{q(r)}| \cdot \mu(A_{r,\mathcal{Q}_r(\mu)})=1$.
	The marker property \eqref{eq:marker_property} implies that for every $y \in Y$ the set
	$\{\mi \in \ZD:~ \rho(T^{\mi}y) \in A_r \}$ is $(1-\delta_{q(r)})F_{q(r)}$-spaced. 
	This means that for every $\mu \in \Prob_e(\Y)$ there exist $R \in \NN$ so that for every $r >R$
	$$
	\mu\left(A_{r,\mathcal{Q}_r(\mu)} \right) > \max_{Q \in \mathcal{Q}_r \setminus \{\mathcal{Q}_r(\mu)\}}\mu\left(A_{r,Q}\right).
	$$
	Let 
	$$G_r = \left\{y \in Y_\to :~ \udense_{\Y}(y,A_{r,Q_r(\mu_y)}) > \max_{Q \in \mathcal{Q}_r \setminus \{\mathcal{Q}_r(\mu_y)\}}\udense_{\Y}(y,A_{r,Q})\right\}.
	$$
	Then it follows that $G_r \subset Y$ is a Borel set, and that for every $\mu \in \Prob_e(\Y)$ there exists $R$ so that  $\inf_{r \ge R} \mu(G_r) =1$. Note that
	$$
	G_r \cap \mathit{emp}^{-1}(Q) =  \left\{y \in Y_\to :~ \udense_{\Y}(y,A_{r,Q}) > \max_{Q' \in \mathcal{Q}_r \setminus \{Q\}}\udense_{\Y}(y,A_{r,Q'})\right\} \mbox{ for } r \in \NN \mbox{ and } Q \in \mathcal{Q}_r.$$
	Since the sets $A_{r, Q}$ are measurable with respect to $\rho^{-1}\left(\Borel(X)\right)$ this shows that for every $r \in \NN$, the partition $G_r \cap \mathit{emp}^{-1}(\mathcal{Q}_r)$ is also measurable with respect to $\rho^{-1}\left(\Borel(X)\right)$. Further let $$Y_\infty' = \bigcup_{R \in \NN} \bigcap_{r >R}G_r.$$ It follows that if $y_1,y_2 \in Y_\infty' $ and $\mu_{y_1} \ne \mu_{y_2}$ then for some large enough $r$, $\QQ_r(\mu_{y_1})\neq \QQ_r(\mu_{y_2})$ and thus $\rho(y_1) \ne \rho(y_2)$. 
	Clearly $Y_\infty' \subset Y$ 
	is a $T$-invariant Borel subset and $\mu(Y_\infty')=1$ for every $\mu \in \Prob_e(\Y)$. Denote 
	\begin{equation}\label{eq:hat_Y_infty_def}
	\hat Y_\infty = \bigcup_{\mu \in \Prob_e(\Y)}\left(Y_\mu \cap G_\mu(\Y)\right).
	\end{equation}
	Let us check that $\hat Y_\infty$ is a Borel subset of $Y$. The above formula involves a possibly  uncountable union, so this does not directly follow from the observation that $Y_\mu \cap G_\mu(\Y)$ is a Borel subset for every $\mu \in \Prob_e(\Y)$.
	To see that $\hat Y_\infty \in \Borel(Y)$  we can use a similar argument used to show that $G(\Y)= \bigcup_{\mu \in \Prob_e(\Y)}G_\mu(\Y)$ is Borel:
	For $Z \in \Clopen(\Y)$, $n,\tilde n, j \in \NN$   define:
	\begin{equation}
	P(Z,n,\tilde n,j) =
	\left\{\mu \in \Prob_e(\Y):
	{Z}_{j,\mu} = Z,~ n_{j, \mu}=n \mbox{ and } n_{j+1, \mu}=\tilde n\right\}.
	\end{equation}
	Then using \eqref{eq:Y_mu_def} we can rewrite \eqref{eq:hat_Y_infty_def} as follows:
	
	$$\hat Y_\infty = 
	\bigcap_{\mi \in \ZD}T^{-\mi}
	\bigcup_{j=1}^\infty\bigcap_{t \ge j}\bigcup_{n,\tilde n, Z}T^{(1-2\delta_{n})F_{n}} \left(\left(Z \setminus  D_{n,\frac{3}{8}\epsilon_n}[\rho_t,\rho_{t+1}]\right) \cap \mathit{emp}^{-1}(P(Z,n,\tilde n,t)) \right). 
	$$
	The right-most union  is over $Z \in \Clopen(\Y)$ and $n,\tilde n \in \NN$.
	Since all the above unions and intersections are countable, it is clear that $\hat Y_\infty$ is a Borel set. Now let 
	$$Y_\infty = Y'_\infty \cap \hat Y_\infty.$$
	Then $Y_\infty \in \Borel(Y)$ is a $T$-invariant set and $\mu(Y_\infty)=1$ for every $\mu \in \Prob_e(\Y)$. Also if $y_1,y_2 \in Y_\infty$ and 
	$\rho(y_1)=\rho(y_2)$ then there exists $\mu \in \Prob_e(\Y)$ such that $y_1,y_2 \in G_\mu(\Y)$ and so 
	$\rho(y_i) = \rho_\mu(y_i) $ for $i=1,2$. Since $\rho_\mu:Y_\mu \to X$ is injective, it follows that $y_1 =y_2$.
	It follows that $\rho:Y_\infty \to X$ is a Borel embedding.
\end{proof}

\subsection{Completing the proof of Theorem \ref{thm:spec_sequence_implies_univesality}}
Our goal now is to prove Lemma \ref{lem:borel_appproximate_embedding} and Lemma \ref{lem:improve_approx_emb_with_measure_encode} above, in order to complete the proof of Theorem \ref{thm:spec_sequence_implies_univesality}. We will state and prove a few  additional auxiliary lemmas in the process. 
\begin{lem}\label{lem:smallest_Cov_N_is_Borel}
	For $\mu \in \Prob_e(\Y)$, $\epsilon >0$, a finite Borel partition $\mathcal{P}$ and $\theta \in (0,1)$ let
	$N(\mu,\epsilon,\theta,\mathcal{P})$ denote the smallest integer $N \ge 1$ that satisfies the conclusion of Lemma \ref{lem:tower_cov} in the sense that for every $n >N$ \eqref{eq:COV_small1}  holds. Then the function 
	$\mu \mapsto N(\mu,\epsilon,\theta,\mathcal{P})$ is Borel measurable.  Similarly, if $N'(\mu,\epsilon,\theta,\mathcal{P},\mathcal{Q})$ denotes the smallest integer $N \ge 1$ that satisfies the conclusion of Lemma \ref{lem:rel_tower_cov}, then the function $\mu \mapsto  N(\mu,\epsilon,\theta,\mathcal{P})$ is Borel measurable.
\end{lem}
\begin{proof}
	
	For any finite Borel partition $\mathcal{P}$, any Borel $A \subset \Y$  and any $m \in \NN$, the set of  $\mu \in \Prob_e(\Y)$ such that $\COV_{\mu,\epsilon,\mathcal{P}}(A) \le m$  is described by a finite system of inequalities on measures of atoms of the partition $\mathcal{P}$ intersected with $A$. This shows that   
	the function $\mu \mapsto \COV_{\mu,\epsilon,\mathcal{P}}(A)$ is a Borel function. From this it follows directly that  the function 
	$\mu \mapsto N(\mu,\epsilon,\theta,\mathcal{P})$ is also a Borel function. The proof of the second statement is similar.   
\end{proof}

\begin{proof}[Proof of Lemma \ref{lem:borel_appproximate_embedding}]
	The proof of Lemma \ref{lem:approximate_embedding_exists} implicitly describes functions $N_0$, $\Phi_0$ and $\overline{Z}$ as in the statement of Lemma \ref{lem:borel_appproximate_embedding}. The idea is to make sure that various arbitrary choices in the proof of Lemma \ref{lem:approximate_embedding_exists} can be specified  in a ``Borel measurable manner''. Here are the details:
	Define $N_0 : \Prob_e(\Y) \times \NN \times (0,1) \to \NN$ by
	$$N_0(\mu,k,\epsilon) = \min \left\{n\in \NN:~ \COV_{\mu,\epsilon,\mathcal{P}_k^{F_n}}(Z_n) < |C_n| \right\}.$$
	Because $$h(\C) > h(Y,T) \ge h_\mu(Y,T; \mathcal{P}_k),$$ it follows from Lemma \ref{lem:tower_cov} that $N_0(\mu,k,\epsilon)$ is well defined and from Lemma \ref{lem:smallest_Cov_N_is_Borel} above it follows that $N_0 : \Prob_e(\Y) \times \NN \times (0,1) \to \NN$ is Borel measurable.
	
	For every $k,n \in \NN$ choose an enumeration of the elements of $\mathcal{P}_k^{F_n}$
	$$F_{k,n}:\mathcal{P}_k^{F_n} \to \{1,\ldots,|\mathcal{P}_k^{F_n}|\}.$$
	For every  $\mu \in \Prob_e(\Y)$, $k,n \in \NN$ let $<_{\mu,k,n}$ be the linear order on $\mathcal{P}_k^{F_n}$ defined by
	$$P <_{\mu,k,n} Q \mbox{ iff } \left(\mu(P \mid Z_n) > \mu(Q \mid Z_n) \mbox{ or } F_{k,n}(P) < F_{k,n}(Q) \mbox{ and } \mu(P \mid Z_n) = \mu(Q \mid Z_n)\right).$$
	Let 
	$$F_{\mu,k,n}:\mathcal{P}_k^{F_n} \to \{1,\ldots,|\mathcal{P}_k^{F_n}|\}$$
	be the enumeration of the elements of $\mathcal{P}_k^{F_n}$ according to the linear order $<_{\mu,k,n}$. It is routine to check that  the map  $(\mu,k,n) \mapsto F_{\mu,k,n}$ is also Borel measurable.
	
	Given $\mu \in \Prob_e(\Y)$, $k,n \in \NN$ such that $n > N_0(\mu,k,\epsilon)$, define $\mathcal{G} \subset \mathcal{P}_k^{F_n}$ by
	$$\mathcal{G} = F_{\mu,k,n}^{-1}(\{1,\ldots,\min\{|\mathcal{P}_k^{F_n}|,|C_n|-1\}\}).$$
	Then by construction of $N_0(\mu,k,\epsilon)$, as in the proof of Lemma \ref{lem:approximate_embedding_exists}, 	$\mathcal{G}$ satisfies \eqref{eq:mu_G_big} and \eqref{eq:size_G_small}. 
	
	We specify
	$$ \overline{Z}(\mu,k,\epsilon,n) = \bigcup_{P \in \mathcal{G}}P \cap Z_n = \bigcup_{P \in F_{\mu,k,n}^{-1}(\{1,\ldots,|\G|\})}P \cap Z_n.$$
	It follows that $\overline{Z}: \Prob_e(\Y) \times \NN \times (0,1) \times  \NN \to \Clopen(Y)$ is Borel measurable. By \eqref{eq:mu_G_big} 
	$$ \mu \left( \overline{Z}(\mu,k,\epsilon,n) \right) > (1-\epsilon)\mu(Z_n).$$
	Choose some fixed enumeration 
	$$C_n = \{x_1,\ldots,x_{|C_n|}\}.$$
	Let $\Phi_{\mu,k,n}:\mathcal{P}_k^{F_n} \to C_n$ be given by
	$$\Phi_{\mu,k,n}(P) = \begin{cases}
	x_{F_{\mu,k,n}(P)} & \mbox { if } F_{\mu,k,n}(P) \le |C_n|\\
	x_{|C_n|} & \mbox{ otherwise}.
	\end{cases}
	$$
	Then by \eqref{eq:size_G_small}, $\Phi_{\mu,k,n}$ is injective on $\mathcal{G}$.
	Let $\phi: Y \to X$ be given by
	$$\phi(y)=\Phi_{\mu,k,n}(\mathcal{P}_k^{F_n}(y)).$$
	Then it follows that $\phi \in C(Y,X)$ and that
	$\rho_{\phi,n} \in \Mor_C(\Y,\X)$ is a $(k,n,\epsilon,\mu)$-approximate embedding,
	and we can choose
	$$Z[\rho_{\phi,n}] = \overline{Z}(\mu,k,\epsilon,n).$$
	In this case define: 	$\Phi_0(\mu,k,\epsilon,n)=\rho_{\phi,n}$ . 
	If $n \le N_0(\mu,k,\epsilon)$ we can arbitrarily define  $\Phi_0(\mu,k,\epsilon,n)=\rho_0$ for some fixed $\rho_0 \in \Mor_C(\Y,\X)$.
 This clearly  defines a Borel measurable function  $\Phi_0: \Prob_e(\Y) \times \NN \times (0,1) \times  \NN \to \Mor_C(\Y, \X)$ that satisfies the statement of the lemma.
\end{proof}

We now introduce some definitions that will be needed for a   technical ``Borel'' version of Lemma  \ref{lem:COV_rel_to_approximate_embedding}: Given $\gamma \in [0,1)$, $\delta,\eta >0$ and finite Borel partitions $\mathcal{P}$ and $\mathcal{Q}$ of $Y$ let
\begin{equation}\label{eq:N_C_def}
\tilde{N}_C \left(\mu,\gamma,\delta,\eta,\mathcal{P},\mathcal{Q}\right)=
\inf \left\{N \in \NN:~ \forall n >N, ~ 
\COV_{\mu,\delta,\mathcal{P}^{F_n \setminus \gamma F_n} \mid \mathcal{Q}^{(1-\delta_n)F_n \setminus \gamma F_n}}(Z_n) < e^{\eta |F_n|}
\right\}. 
\end{equation}\index{Definitions and notation introduced in Section 7!$\tilde{N}_C$}
Depending on the  parameters $\tilde{N}_C \left(\mu,\gamma,\delta,\eta,\mathcal{P},\mathcal{Q}\right)$ could be a finite natural number or $+\infty$.
For $A \in \Borel(Y)$ and an $n_0$-towerable map $\rho \in \Mor(\Y,\X)$ let
\begin{equation}\label{eq:N_U_def}
\tilde{N}_U\left(\mu,\gamma,\delta,\eta,\mathcal{P},\rho,A\right)= \sup  \left\{\tilde{N}_C(\mu,\gamma,\delta,\eta,\mathcal{P},\mathcal{P}_{\hat \rho}) :~ \hat{\rho} \in \Mor(\Y,\X) \mbox{ is symbolic and } D_{n_0,\frac{1}{2}\epsilon_{n_0}}[\rho,\hat{\rho}]\subseteq A\right\}.
\end{equation}\index{Definitions and notation introduced in Section 7!$\tilde{N}_U$}
Roughly speaking, in Lemma \ref{lem:COV_rel_to_approximate_embedding} we showed that given $\eta>0$ and $k_0$, if $\rho$ is a $(k_0, n_0, \frac{1}{N_0}, \mu)$-approximate embedding for large enough $n_0$ and $N_0$ then the log of the approximate covering number of $\mathcal P_{k_0}^{F_n}$ relative to $\mathcal P_\rho^{F_n}$ is an $\eta$ fraction of $|F_n|$ provided $n$ is large enough. In the course of our proof, we found that we needed $n_0, N_0> \overline{N}_R(\eta, k_0)$ where $\overline{N}_R$ is a Borel function. We point this out more explicitly in the following lemma with some additional technicalities coming in because we replace the map $\rho$ with a map $\tilde \rho$ which is close to $\rho$ outside a  set $A$ of small measure. 

\begin{lem}\label{lem:COV_rel_to_approximate_embedding_borel}
	There exists a Borel function 
	$$  \overline{N}_R: (0,1) \times \NN \to \NN$$
	so that for every  $\eta >0$, 
	and $k_0,n_0\in \NN$ such that $n_0 >\overline{N}_R(\eta,k_0)$, $\mu \in \Prob_e(\Y)$ and  
	any Borel set  $A \subset Y$ such that
	$$\mu(A \mid Z_{n_0})< \frac{1}{\overline{N}_R(\eta,k_0)},$$
any  $(k_0,n_0, \frac{1}{\overline{N}_R(\eta,k_0)},\mu)$-approximate embedding $\rho \in \Mor(\Y,\X)$,  $\gamma \in [0,1)$ and  $\delta >0$ we have
	$$ \tilde{N}_U(\mu,\gamma,\delta,\eta, \mathcal{P}_{k_0},\rho,A) < +\infty.$$
	Equivalently:
	There exists $N \in \NN$ so that for every $n >N$ and every  $\tilde \rho \in \Mor(\Y,\X)$ that satisfies $D_{n_0,\frac{1}{2}\epsilon_{n_0}}[\rho,\tilde \rho]  \subset A$ we have
	\begin{equation}\label{eq:COV_P_k_cond_tilde_rho_small}
	\COV_{\mu,\delta,\mathcal{P}_{k_0}^{F_n \setminus \gamma F_n }\mid \mathcal{P}_{\tilde \rho}^{(1-2\delta_n)F_n \setminus \gamma F_n}}(Z_n) < e^{ \eta |F_n|}.
	\end{equation}
\end{lem}
\begin{proof} 
Although the statement of this lemma is a bit more complicated compared to that of Lemma \ref{lem:COV_rel_to_approximate_embedding}, the proof is almost identical. We explain the requirement modifications.
As in the proof of Lemma \ref{lem:COV_rel_to_approximate_embedding},
	define $\overline{N}_R: (0,1) \times \NN \to \NN$  by \eqref{eq:N_r_def},	for $\eta \in (0,1)$ and $k_0 \in \NN$.
	It is clear that  $\overline{N}_R: (0,1) \times \NN \to \NN$ is a Borel function.
	 Fix $\eta \in (0,1)$ and $k_0 \in \NN$. Denote $N_0 = \overline{N}_R(\eta,k_0)$. 
	For $n_0>N_0$, let $A \subset Y$ with $\mu(A \mid Z_{n_0})  < \frac{1}{N_0}$,  $\rho \in \Mor(\Y,\X)$ and $\gamma \in [0,1)$ be as in the statement above and $\delta >0$.

	In contrast to the proof of Lemma \ref{lem:COV_rel_to_approximate_embedding} where $G_{n_0}$ has been defined by \eqref{eq:G_n_rho_def}, this time we define $G_{n_0} \subset Y$ by 
	\begin{equation}\label{G_n_rho_tilde_rho_def}
	G_{n_0} = T^{(1-2\delta_{n_0})F_{n_0}}\left(Z[\rho] \setminus  A\right).
	\end{equation}
	Then 
	$$\mu(Z_{n_0}\setminus (Z[\rho] \setminus A) \mid Z_{n_0}) <  \mu(Z_{n_0} \setminus Z[\rho] \mid  Z_{n_0}) + \mu(A \mid Z_{n_0}) <  \frac{2}{N_0},$$
	so by Lemma \ref{lem:mu_tower_subset}
	$$\mu(Y \setminus G_{n_0})   <  \epsilon_{n_0} +  \left(1-\frac{|(1-2\delta_{n_0})F_{n_0}|}{|(1+\delta_{n_0})F_{n_0}|}\right)+ \frac{2}{N_0}.$$ 
	In particular, by definition of $\t \epsilon_{N_0}$ in \eqref{eq:t_epsilon_def}, it follows that  $\mu(Y \setminus G_{n_0}) < \t \epsilon_{N_0}$.
		Let $\tilde \rho \in \Mor(\Y,\X)$ satisfy  $D_{n_0}[\rho,\tilde \rho] \subset A$.
		By Lemma \ref{lem:approximate_inverese}, for every $y \in G_{n_0}$, the value of	$\mathcal{P}_{k_0}(y)$ is determined by $\rho (y)\mid_{F_{2n_0}}$.
		This means that 
		$$\COV_{\mu,0,\mathcal{P}_{k_0} \mid \mathcal{P}_{\tilde \rho}^{F_{2n_0}}}(G_{n_0}) =1.$$
		
		From here we proceed exactly as in the proof of Lemma \ref{lem:COV_rel_to_approximate_embedding}, replacing $\rho$ by $\t\rho$ throughout the proof.
\end{proof}

 The following is a technical refinement of Lemma   \ref{lem:improve_approx_emb}, an auxiliary result towards Lemma \ref{lem:improve_approx_emb_with_measure_encode}.

\begin{lem}\label{lem:improve_approx_emb_robust}
	Suppose $(E_m)_{m=1}^\infty$ is a sequence of Borel subsets $E_m \in \Borel(Y)$ and $E_m \subseteq Z_m$ so that 
	\begin{equation}\label{eq:lim_ucap_E_m}
	\forall \mu \in \Prob_e(\Y),~ \lim_{m \to \infty}\mu(E_m \mid Z_m) =0.
	\end{equation}
	Then there exists Borel functions 
		$$\overline{k}_0:\Prob_e(\Y)\times (0,1) \to \NN,$$
    $$  \overline{N}_0:\Prob_e(\Y)\times (0,1) \times \NN \to \NN,$$
	$$\tilde{N}:\Prob_e(\Y)\times (0,1)  \times \Mor_C(\Y,\X) \times (0,1) \times \NN \to \mathbb{N},$$ 
	$$\tilde\Phi:\Prob_e(\Y)\times \NN \times \NN \times (0,1) \times (0,1) \times \Mor_C(\Y,\X) \to \Mor_C(\Y,\X)$$ 
	and
	$$\tilde{Z}:\Prob_e(\Y)\times \NN \times \NN \times (0,1) \times (0,1) \times \Mor_C(\Y,\X) \to  \Clopen(Y),$$
	so that the  following holds:	
	Suppose that $\mu \in \Prob_e(\Y)$, $\delta_0, \gamma\in (0,1)$, 
	and that $\rho \in \Mor_C(\Y,\X)$ is a continuous $(k_0,n_0,\delta_0,\mu)$-approximate embedding where
	$$ k_0 \ge \overline{k}_0(\mu,\gamma),$$
	\begin{equation}\label{eq:n_0_k_0_big2}
	n_0 \ge {\overline{N}_0(\mu,\gamma,k_0)},
	\end{equation}
	{ $$ \delta_0 < \frac{1}{\overline{N}_0(\mu,\gamma,k_0)},$$}
	that $\hat \rho \in \Mor_C(\Y,X)$ is $\hat n$-towerable for some 
	$\hat{n} \in \NN$ so that $n_0 \le \hat{n} \ll n$ and satisfies 
	\begin{equation}\label{eq:D_n_0_rho_hat_rho}
	D_{n_0,\frac{1}{2}\epsilon_{n_0}}[\rho,\hat \rho] \subseteq E_{n_0}.
	\end{equation}
	Fix $\delta\in(0,1)$, $k\in \N$ and let
	\begin{equation}\label{eq:n_big_tilde_N}
	n > \tilde N(\mu,\gamma,\rho,\delta,k).
	\end{equation}
	Denote 
	$$\tilde \rho  = \tilde \Phi(\mu,k,n,\gamma,\delta,\hat \rho) \in \Mor_C(\Y,\X).$$
	Then
	\begin{equation}\label{eq:tilde_phi_approx_emb}
	\tilde \rho \mbox{ is a }  (k,n,\delta,\mu) \mbox{-approximate embedding},
	\end{equation}
	and we can take
	\begin{equation}
	Z[\tilde 	\rho] = \tilde Z(\mu,k,n,\gamma,\delta,\hat \rho)
	\end{equation}
	and
	\begin{equation}\label{eq:tilde_rho_traces_rho1}
	D_{\hat n,\frac{3}{8}\epsilon_{\hat n}}[\hat \rho,\tilde  \rho] \subset  Y \setminus \left(T^{(1-2\delta_n)F_n \setminus \gamma F_n} Z_n \right).
	\end{equation}
\end{lem}
	
 Since the statement above is somewhat lengthy and convolved, we give an informal interpretation.   Recall that  given a $(k_0, n_0, \delta_0, \mu)$-approximate embedding $\rho$ for large enough $k_0, n_0$ and small enough $\delta_0$, Lemma \ref{lem:improve_approx_emb} says we can construct a much better $(k, n, \delta, \mu)$-approximate embedding $\tilde \rho$ for large enough $n$ which is equal to $\rho$ on a large part of the space. The following Lemma will show how all of this can be done as Borel function of various parameters. It will further show that the ``procedure''  is ``robust'' in the sense  that $\rho$ can be replaced by another morphism $\hat \rho$ which is obtained by perturbing it on small part of the space denoted by $E_{n_0}$  (in the sense that \eqref{eq:D_n_0_rho_hat_rho} holds), in such a way that the parameters are controlled in a uniform (and Borel) manner. The reader should think of $\hat \rho$ as a morphism which further encodes the measure $\mu$.
We note that Lemma \ref{lem:improve_approx_emb} is a particular case of Lemma \ref{lem:improve_approx_emb_robust} by setting $\hat \rho = \rho$ and $\hat n = n_0$  .

\begin{proof}
	As in the proof of Lemma \ref{lem:borel_appproximate_embedding}, the basic idea is to follow the steps of the proof of Lemma \ref{lem:improve_approx_emb}, taking care to specify all the ``choices'' so it is is clear that the functions constructed are all Borel measurable. 
	Additionally, there is the issue of making the function $\tilde \Phi$ ``works properly''even if we replace $\rho$ by suitable $\hat \rho$. For this we 	
	use the full strength of Lemma \ref{lem:COV_rel_to_approximate_embedding_borel}.
	Here are the details:
	
	Fix a sequence $(E_m)_{m=1}^\infty$ of Borel subsets that satisfies \eqref{eq:lim_ucap_E_m}. As in the proof of Lemma \ref{lem:improve_approx_emb}, fix $\hat{h} \in (h(\Y),h(\C))$.
	We can take  function $\overline{k_0}:\Prob_e(\Y) \times (0,1) \to \NN$ to be the one given by \eqref{eq:k_0_def} in the proof of Lemma \ref{lem:improve_approx_emb},
	namely, the smallest positive integer $k$ such that the entropy of $(Y,T,\mu)$ given $\mathcal{P}_k^{\ZD}$ is less than $\alpha(\gamma)/8$, where $\alpha(\gamma)$ is given by \eqref{eq:alpha_def}. The function $\mu \mapsto h_\mu(Y,T\mid \mathcal{P}_k^{\ZD})$ is a Borel function on $\Prob_e(\Y)$. The function  $\alpha:(0,1) \to \mathbb{R}$ given by  \eqref{eq:alpha_def} is clearly a Borel function. It follows that the function $\overline{k_0}:\Prob_e(\Y) \times (0,1) \to \NN$ is Borel measurable.
	
{
 For $\mu \in \Prob_e(\Y)$, $\gamma \in (0,1)$ and $k_0 \ge \overline{k_0}(\mu,\gamma)$, let
	\begin{equation}\label{eq:ovreline_N_0_prime_def}
	\overline{N}_0'(\mu,\gamma,k_0) = \max \left\{\overline{N}_{S}(\alpha(\gamma)^d \hat{h}),\overline{N}_R\left(\alpha(\gamma)/4,k_0\right), K_0(\gamma)\right\},
	\end{equation}
}
	where $\overline{N}_{S}:(0,1) \to \NN$ is given by \eqref{eq:N_S_def}, $\overline{N}_R:(0,1)\times \NN \to \NN$ is given by \eqref{eq:N_r_def} and $K_0(\gamma)$ is the smallest integer $K_0$ that appears in the statement of Lemma \ref{lem:C_gamma_n}. An inspection of the proof of Lemma \ref{lem:C_gamma_n}  confirms  that the integer $K_0$ depends only on $\gamma$ and that $\gamma \mapsto K_0(\gamma)$ is a a Borel function.
	Define { $\overline{N}_0:\Prob_e(\Y)\times (0,1) \times \NN  \to \NN$} by
	\begin{equation}\label{eq:ovreline_N_0_def}
	\overline{N}_0(\mu,\gamma,k_0)= \min \left\{ n \ge \overline{N}_0'(\mu,\gamma, k_0):~ \mu\left(E_n \mid Z_n\right) < \frac{1}{\overline{N}_R\left(\alpha(\gamma)/4,k_0\right)} \right\}.
	\end{equation}
	This is a well defined natural  number by \eqref{eq:lim_ucap_E_m}.
	It follows that $\overline{N}_0$ is a Borel function.
	
	Now for $\mu\in\Prob_e(Y, T)$ we fix $n_0, \hat n, k_0$ and $\delta_0$ as in \eqref{eq:n_0_k_0_big2} and $ \rho, \hat \rho  \in \Mor_C(\Y,\X)$ as in the statement of the lemma. Fix $\delta >0$, $k\in \N$ and let $\delta' = 10^{-10}\delta^4$. Let
	$\tilde N(\mu,\gamma,\rho,\delta,k)$ be the smallest natural number $N$  satisfying  $N \gg n_0$ 
	and so that 
	for every $n >N$  \eqref{eq:gamma_n_big}, \eqref{eq:C_n_gamma_big} (the role of $k$ which appears in this inequality is played by $n_0$ here), \eqref{eq:COV_P_k_0_gamma_F_n},  \eqref{eq:rel_COV_P_k_F_n_P_k_0_F_n_Z_n}   
	and \eqref{eq:epsilon_n_smaller_eta} hold, and in addition 
	$$N \ge \tilde{N}_U(\mu,\gamma,\delta',\alpha(\gamma)/4,\mathcal{P}_{k_0},\rho,E_{n_0}),$$
	where the right hand side is the function defined by \eqref{eq:N_U_def}.
	The fact that the right hand side is finite follows by Lemma \ref{lem:COV_rel_to_approximate_embedding_borel}
	because {$\mu\left(E_{n_0} \mid Z_{n_0} \right) < \frac{1}{\overline{N}_R(\alpha(\gamma)/4,k_0)}$ and $n_0 >\overline{N}_R(\alpha(\gamma)/4,k_0)$}.
	This means that 
	 \begin{equation}\label{eq:COV_P_k_0_rel_hat_rho}
	\COV_{\mu,\delta',\mathcal{P}_{k_0}^{F_n \setminus \gamma F_{n}}\mid \mathcal{P}_{\hat \rho}^{(1-2\delta_n)F_n \setminus \gamma F_n}}(Z_n) < e^{\alpha(\gamma)/4 |F_n|}
	\end{equation}
 holds for any $\hat{n}$-towerable $\hat \rho \in \Mor_C(\Y,\X)$ satisfying \eqref{eq:D_n_0_rho_hat_rho} with $\hat{n} \ll n$
 (this is the analog of \eqref{eq:COV_P_k_0_rel_rho} with $\rho$ replaced with $\hat \rho$).
	
	It follows that $\tilde N$ is a Borel function for the set of parameters where it has been defined. 
	Extend it to a Borel function
	$$\tilde{N}:\Prob_e(\Y)\times (0,1)  \times \Mor_C(\Y,\X) \times (0,1) \times \NN \to \mathbb{N}.$$ 
	
	Let us now describe the functions 
	$$\tilde\Phi:\Prob_e(\Y)\times \NN \times \NN \times (0,1) \times (0,1) \times \Mor_C(\Y,\X) \to \Mor_C(\Y,\X),$$ 
	and
	$$\tilde{Z}:\Prob_e(\Y)\times \NN \times \NN \times (0,1) \times (0,1) \times \Mor_C(\Y,\X) \to  \Clopen(Y).$$
	
	It is enough to define $\tilde \Phi(\mu,k,n,\gamma,\delta,\hat \rho) \in C(\Y,\X)$ and $\tilde{Z}(\mu,k,n,\gamma,\delta,\hat \rho)  \in \Clopen(Y)$ for parameters satisfying the assumptions of the lemma. So assume that  
	$n \ge \tilde{N}(\mu,\gamma,\rho,\delta, k)$, $\rho  \in \Mor_C(\Y,\X)$ as before, and that $\hat \rho \in \Mor_c(\Y,X)$ is $\hat n$-towerable for some $n_0 \le \hat{n} \ll n$ and satisfies \eqref{eq:D_n_0_rho_hat_rho}.
	Then, repeating the exact same steps as in the proof of Lemma \ref{lem:improve_approx_emb}, except that this time we use \eqref{eq:COV_P_k_0_rel_hat_rho} instead of \eqref{eq:COV_P_k_0_rel_rho}, and replace $\rho$ by $\hat \rho$ throughout, we conclude  that
	\begin{equation}
	\COV_{\mu,\delta/3,\mathcal{P}_{k}^{F_n} \mid \mathcal{P}_{\hat \rho}^{(1-2\delta_n)F_n \setminus \gamma F_n}}(Z_n) < e^{\alpha(\gamma)/2 |F_n| + \hat h |\gamma F_n|}.
	\end{equation}
	This is the counterpart of \eqref{equation: covering is good} as in the proof of Lemma \ref{lem:improve_approx_emb}. Let $K_{n,\hat n,\beta(\gamma)} \subset (1-\delta_n)F_n$ be given by \eqref{eq:K_n_n_0_theta_def} and
	$$ \hat{X}_n = \left\{w \in X^{(1-2\delta_n)F_n \setminus \gamma F_n}:~ \COV_{\mu(\cdot \mid \hat \rho^{-1}([w])),\delta/3,\mathcal{P}_k^{F_n}}(Z_n) \le
	|C_{\hat n}^{K_{n,\hat n, \beta(\gamma)}}|
	\right\}.
	$$
	Then the analog of \eqref{eq:mu_hat_X_n1} holds with $\rho$ replaced by $\hat \rho$. Enumerate the elements of $C_{\hat n}^{K_{n,\hat n, \beta(\gamma)}}$:
	$$C_{\hat n}^{K_{n,\hat n, \beta(\gamma)}}= \{w_1,\ldots,w_{M}\},$$
	where
	$$M = \left|C_{\hat n}^{K_{n,\hat n, \beta(\gamma)}} \right|.$$
	As in the proof of Lemma \ref{lem:borel_appproximate_embedding}, 
	for every $k,n \in \NN$ let 
	$$F_{k,n}:\mathcal{P}_k^{F_n} \to \{1,\ldots,|\mathcal{P}_k^{F_n}|\}$$
	be an enumeration of the elements of $\mathcal{P}_k^{F_n}$.
	For every  $\mu \in \Prob_e(\Y)$, $k,n \in \NN$ and $w\in X^{(1-2\delta_n)F_n \setminus \gamma F_n}$ let $<_{\mu,k,n,w}$ be the linear order on $\mathcal{P}_k^{F_n}$ define by $P <_{\mu,k,n,w} Q$ if and only if
	$$
	\mu(P \mid Z_n \cap \hat \rho^{-1}([w])) > \mu(Q \mid Z_n\cap \hat \rho^{-1}([w]))$$
	or
	$$F_{k,n}(P) < F_{k,n}(Q) \mbox{ and } \mu(P \mid Z_n\cap\hat \rho^{-1}([w])) = \mu(Q \mid Z_n\cap\hat \rho^{-1}([w])).
	$$
	Let 
	$$F_{\mu,k,n,w}:\mathcal{P}_k^{F_n} \to \{1,\ldots,|\mathcal{P}_k^{F_n}|\}$$
	be the enumeration of the elements of $\mathcal{P}_k^{F_n}$ according to the linear order $<_{\mu,k,n,w}$.
	As in the  proof of Lemma \ref{lem:borel_appproximate_embedding}, $(\mu,k,n,w) \mapsto F_{\mu,k,n,w}$ is also Borel measurable.
	
	Let $\mathcal{G}_w \subset \mathcal{P}_k^{F_n}$
	$$\mathcal{G}_w = F_{\mu,k,n,w}^{-1}(\{1,\ldots,\min\{M,|\mathcal{P}_k^{F_n}|\}\}).$$
	For $w \in \hat{X}_n$, as in the proof of Lemma \ref{lem:improve_approx_emb}, we see that \eqref{eq:mu_G_w_big} holds  with $\rho$ replaced by $\hat \rho$, and also that \eqref{eq:G_w_small_C_eta} holds with $n_0$ replaced by $\hat n$.
	For $w \in \hat{X}_n$ let 
	$\Phi_w:\mathcal{P}_k^{F_n} \to C_{\hat n}^{K_{n,\hat n, \beta(\gamma)}}$ be given by 
	$$\Phi_w(P)=\begin{cases}
	w_{F_{\mu,k,n,w}(P)}&\text{ if }F_{\mu,k,n,w}(P)\leq |C_{\hat n}^{K_{n,\hat n, \beta(\gamma)}}|  \\
	w_{|C_{\hat n}^{K_{n,\hat n, \beta(\gamma)}}|}&\text{ otherwise.}
	\end{cases}$$ Then $\Phi_w$ is injective on $\mathcal{G}_w$.
	From here we proceed exactly as in the proof of Lemma \ref{lem:improve_approx_emb}, replacing $\rho$ with $\hat \rho$ and $n_0$ with $\hat n$ in the appropriate places: For $y \in Y$  define 
	$\overline{W}_y \in \Inter_n(\C)$  by  replacing $\rho$ with $\hat \rho$ and $n_0$ with $\hat n$ in \eqref{eq:overline_W_y_def}.
	
	Then let
	\begin{equation}\tilde \Phi(\mu,k,n,\gamma,\delta,\hat \rho) 
	= \rho_{\tilde \phi,n},
	\end{equation}
	where
	$\tilde \phi$ be given by \eqref{eq:tilde_phi_def},  and $\rho_{\tilde \phi,n}$ is given by \eqref{eq:rho_Phi_n_def}. 
	Then indeed $\rho= \rho_{\tilde \phi,n}\in \Mor_C(\Y,\X)$, and we can extend $\tilde \Phi$ to a Borel function arbitrarily on other parameter values.  
	Let $Y_0 \subset Z_n$ be the set given by \eqref{eq:Y_0_def}. Then it follows that $Y_0$ is clopen in $Y$. 
	We set $\tilde Z({\mu,k,n,\gamma,\delta, \hat\rho}) = Y_0$.
	From here, the proof concludes exactly as in the proof of Lemma \ref{lem:improve_approx_emb}.
\end{proof}

	The following lemma describes how given an $n$-towerable map $\rho$, we can modify it on a small part of the space to get a $q(r)$-towerable map $\tilde \rho$ such that $\tilde \rho$ is close to $\rho$ on most of the space and on a small portion it  is close to a chosen element $x\in C_r$ where $r\ll n\ll q(r)$ in a Borel fashion. The reader should think of $\rho$ as an $(n, k, \delta, \mu)$-approximate embedding and $x\in C_r$ as a point chosen so as to record the partition element $\mu$ belongs to.

\begin{lem}\label{lem:trace_and_recover_Q}
	Let $q:\NN\to \NN$ be an increasing function so that $q(r) \ll q(r+1)$ for every $r \in \NN$ and $1 \ll q(1)$. Then there exists  a sequence of Borel functions
		$$\tilde \Phi_r: C_r \times \Mor_C(\Y,\X) \to \Mor_C(\Y,\X)$$ 
	with the following property:
    Suppose  $x \in C_r$ and  $\rho \in \Mor_C(\Y,\X)$   is a continuous $n$-towerable with  $r \ll n \ll q(r)$.
    Denote 
    $$
    \tilde \rho = \tilde \Phi_r(x,\rho) \in \Mor_C(\Y,\X).$$
    Then $\tilde \rho$ is a   $q(r)$-towerable continuous morphism so that

	\begin{equation}\label{eq:tilde_phi_traces_phi_on_W_r_}
	D_{n,\frac{3}{8}\epsilon_n}[\rho,\tilde \rho] \subseteq K_{r,n},
	\end{equation}
	and
	\begin{equation}\label{eq:mu_circ_rho_tilde_phi_in_U_x_r}
	\forall y \in Z_{q(r)},~ d_X^{F_r}\left(\tilde \rho(y),x\right) < \frac{11}{40}\epsilon_{r}
	\end{equation}
	where:
	\begin{equation}\label{eq:K_r_def}
	K_{r,n} = Y \setminus T^{\hat F_{r,n}}Z_{q(r)}
	\end{equation}\index{Definitions and notation introduced in Section 7!$K_{r,n}$}
	and
	\begin{equation}\label{eq:F_r_def}
	\hat F_{r,n} = (1-\delta_{q(r)})F_{q(r)} \setminus (1+\delta_{n})F_{2n}.
	\end{equation}\index{Definitions and notation introduced in Section 7!$\hat F_{r,n}$}
	
\end{lem}

\begin{proof}
	Let $q:\NN \to \NN$ as in the statement of the lemma.

 	Fix $r \in \NN$.  We will now construct   $\tilde \Phi_r:C_r \times \Mor_C(\Y,\X) \to \Mor_C(\Y,\X)$ as in the statement of the lemma.
	It suffices to define 
	$\tilde \Phi_r(x,\rho)$ for $x \in C_r$ and $n$-towerable   $\rho \in \Mor_C(\Y,\X)$  with $r \ll n \ll q(r)$.
	Let $W_{x,n} \in C_r^{\{\vec{0}\}}$ be given by $(W_{x,n})_{\vec{0}} =x$, and let $\tilde x \in C_n$ be given by $\tilde x = \Ext_n(W_{x,n})$. Consequently $d_{X}^{F_r}(\tilde x, x)< \frac14 \epsilon_{r}$.
	
	For $y \in Y$ define
	$$K_y = \left\{\mi \in \hat F_{r,n} :~T^{\mi}(y) \in Z_{n} \right\} \cup \{\vec{0}\}\subset (1-\delta_{q(r)})F_{q(r)} ,$$
	where $\hat F_{r,n} \subset F_{q(r)}$ is defined by \eqref{eq:F_r_def}.
	Then $K_y \subset (1-\delta_{q(r)})F_{q(r)}$ is $(1+\delta_n)F_n$-spaced. Define $W_y \in C_{n}^{K_y}$ by
	$$(W_y)_{\mi} =
	\begin{cases}
	\tilde x & \mi = 0\\
	S^{\mi}(\rho(y)) & \mi \in K_y \setminus \{0\}.
	\end{cases}
	$$
 Define $\tilde \phi:Y \to X$ by
	$$\tilde \phi(y) = \Ext_{q(r)}(W_y).$$
	Let
	$$\tilde \rho =  \rho_{\tilde \phi, q(r)},$$
 where $\rho_{\tilde \phi, q(r)}$ is defined in \eqref{eq:rho_Phi_n_def}. By definition, $\tilde \rho$ is $q(r)$-towerable. Let $y'\in Z_n$. If $y'\in T^{\hat{F}_{r,n}}(Z_{q(r)})$ then we have 
$$d^{F_n}(\tilde{\rho}(y'), \rho(y'))< 1/4 \epsilon_{n}.$$
Then $\tilde{\rho}$ satisfies \eqref{eq:tilde_phi_traces_phi_on_W_r_}  (we are replacing $\frac14$ for the worse constant $\frac38$, see Remark \ref{remark:constant_half}). Also, by the triangle inequality we have for every $y \in Y$,
	\begin{equation}\label{eq:tilde_psi_code_Psi}
	d_X^{F_r}\left(\tilde \phi(y), x \right) <d^{F_r}_X(\tilde\phi(y), \tilde x)+d^{F_r}_X(\tilde x, x)<1/4 \epsilon_r+ 1/4 \epsilon_{q(r)}<\frac{11}{40}\epsilon_r.
	\end{equation}
	By definition of $\rho_{\tilde \phi, q(r)}$ this implies 
	\eqref{eq:mu_circ_rho_tilde_phi_in_U_x_r}.
	
	For $\rho \in \Mor_C(\Y,\X)$ as above and $x \in C_r$, define
	$$\tilde \Phi_r(x,\rho) = \rho_{\tilde \phi, q(r)}.$$
So far, we have defined $\tilde{\Phi}_r(x, \rho)$ for $x\in C_r$ and $\rho\in Mor_C(\mathcal Y, \mathcal X)$ which are $n$-towerable. The set of such morphisms forms a Borel set and we extend the definition of $\tilde \Phi_r$ arbitrarily to a measurable function on the entire domain $C_r\times Mor_C(\mathcal Y, \mathcal X)$. This completes the construction.  
\end{proof}

	The proof of Lemma \ref{lem:improve_approx_emb_with_measure_encode} below applies both Lemma \ref{lem:trace_and_recover_Q} and Lemma  \ref{lem:improve_approx_emb_robust}: Given a   $(k_0, n_0, \delta_0, \mu)$-approximate embedding $\rho$ we first use Lemma  \ref{lem:trace_and_recover_Q}  to produce a morphism $\hat \rho$ which is close to $\rho$ on most of the space and encodes some information about the measure $\mu$.  Then we use Lemma \ref{lem:improve_approx_emb_robust} which given $\hat \rho$ helps construct an ``improved''  $(k, n, \delta, \mu)$-approximate embedding $\tilde \rho$, which is still close to $\rho$ on most of the space, and still retains information about  the measure $\mu$  encoded in $\hat \rho$. While the idea is simple, there is copious book-keeping involved.

\begin{proof}[Proof of Lemma \ref{lem:improve_approx_emb_with_measure_encode}]

	Let $q,\tilde q:\NN \to \NN$  be  functions that satisfy \eqref{eq:q_r_rapid}, as in the statement of the lemma.
	For $n,n_0 \in \NN$ denote:
	\begin{equation}\label{eq:r_1_r_2_def}
	  r_1(n_0) =\min \{r \in \NN:~ n_0 \ll q(r)\}
	\mbox{ and }
	r_2(n) =  \max \{r \in \NN:~ q(r) \ll  n\}.
	\end{equation}
	Let $(\tilde \Phi_r)_{r=1}^\infty$ be the functions given by Lemma \ref{lem:trace_and_recover_Q}. For $r \in \NN$ let 
	$$ \t E_r =  \left(Y\setminus T^{(1-\delta_{q(r)})F_{q(r)}}Z_{q(r)}\right) \cup T^{F_{q(r)}}\left(K_{r+1,q(r)}\right),$$
		where $K_{r,n}$ is given by \eqref{eq:K_r_def}.
	For every $n_0 \in \NN$ define
		\begin{equation}
	E_{n_0} =  K_{r_1(n_0),n_0} \cup 
	\bigcup_{r = r_1(n_0) }^{\infty}\t E_r.
	\end{equation}

	By definition of $K_{r,n}$, if $1\ll n_0$, for any $\mu \in \Prob_{e}(\Y)$,
	\begin{equation}\label{eq:mu_K_r_1}
		\mu\left( K_{r_1(n_0),n_0} \mid Z_{n_0}\right) \le 2(\epsilon_{q(r_1(n_0))}+ 6d\delta_{q(r_1(n_0))}+ 2\frac{|F_{n_0}|}{|F_{q(r_1(n_0))}|})|F_{n_0}| \le \epsilon_{n_0} + \delta_{n_0},
	\end{equation}
where in the last inequality we used that $n_0 \ll q(r_1(n_0))$. Note that 
	$$T^{F_{q(r)}}(K_{r+1,q(r)})\subseteq (Y\setminus T^{(1-2\delta_{q(r+1)})F_{q(r+1)}}Z_{q(r+1)}) \cup 
	T^{3(1+\delta_{q(r)})F_{q(r)}}Z_{q(r+1)}.$$
	An estimate similar  to that of \eqref{eq:mu_K_r_1} shows that  for any $r \ge r_1(n_0)$, $1\ll k \ll q(r)$ and  $\mu \in \Prob_{e}(\Y)$ we have
	\begin{equation}\label{eq:mu_E_0_k2}
	\mu\left(\t E_r\right) < \left(\epsilon_{k} + \delta_{k}\right)\mu\left( Z_{k} \right).
	\end{equation}
	In particular, by union bound it follows that
	$$\mu\left(E_{n_0} \mid Z_{n_0} \right) \le 2\epsilon_{n_0}+2\delta_{n_0}.$$
		This shows that the sequence $(E_m)_{m=1}^\infty$ satisfies 
	\eqref{eq:lim_ucap_E_m}.
		 Let 
	$N_0(\gamma)$ be the smallest $n \in \NN$ so that 
	$$\epsilon_{n}+\delta_{n} < \gamma/4.$$
	It follows that for every $n_0 \ge N_0(\gamma)$,
	\begin{equation}\label{eq:mu_E_n_0_n_gamma_small}
	\mu \left( E_{n_0} \mid Z_{n_0}\right) < \frac{1}{2}\gamma.
	\end{equation}

	Let {$\tilde{N}_0:\Prob_e(\Y)\times (0,1)\times \NN \to \NN$} be defined by
{	$$\tilde{N}_0(\mu,\gamma,k_0) = \max\{\overline{N}_0(\mu,\frac{1}{8}\gamma,k_0),N_0(\gamma), \frac{1}{\gamma},\tilde q(1)\} \mbox{ for } \gamma \in (0,1),\mu\in \Prob_e(\Y), k_0 \in \NN$$}
	where { $\overline{N}_0:\Prob_e(\Y)\times (0,1) \times \NN \to \NN$ } is  the Borel function obtained  by  applying Lemma \ref{lem:improve_approx_emb_robust} with the sequence $(E_m)_{m=1}^\infty$ above.

	Let $\tilde{k}_0:\Prob_e(\Y)\times (0,1) \to \NN$ 
	be defined by
	$$\tilde{k}_0(\mu,\gamma) = \overline{k}_0(\mu,\frac{1}{8}\gamma)\mbox{ for } \gamma \in (0,1),\mu\in \Prob_e(\Y),$$ where $\overline{k}_0:\Prob_e(\Y)\times (0,1) \to \NN$ is
	the  Borel function obtained by applying Lemma \ref{lem:improve_approx_emb_robust} with the sequence $(E_m)_{m=1}^\infty$ above.

	Also , let $\tilde{N}:\Prob_e(\Y)\times (0,1)  \times \Mor_C(\Y,\X) \times (0,1) \times \NN \to \mathbb{N}$
	be obtained  by  applying Lemma \ref{lem:improve_approx_emb_robust} with the sequence $(E_m)_{m=1}^\infty$ above.
	Recall that for every $r \in \NN$, $\mathcal{Q}_r$ is a finite Borel partition of $\Prob_e(\Y)$, and that according to our assumption $|\mathcal{Q}_r| \le |C_r|$. For each $r \in \NN$, let  $\tilde f_r:\mathcal{Q}_r \to C_r$ be an injective function. Recall that for every $r \in \NN$ the set $C_r \subset X$ is $(\epsilon_r,F_r)$-separated so we can choose a partition $\{C_{r,Q}\}_{Q \in \mathcal{Q}_r}$ of $C_{q(r)}$ so that for every  $Q \in \mathcal{Q}_r$ 
	$$
	\{x \in C_{q(r)}:~ d_X^{F_r}(x,\tilde f_r(Q)) < \frac{1}{2}\epsilon_r\} \subseteq C_{r,Q}.
	$$
	
	We will now define $\overline{\Phi}(\mu,k,n,\delta,\gamma,\rho)$ assuming that  $\rho \in \Mor_C(\Y,X)$ is a continuous  $(k_0,n_0,\delta_0,\mu)$-approximate embedding with $k_0,n_0 \in \NN$ and $\delta_0$ satisfying \eqref{eq:cond_delta_0_n_0_k_0} and $n \in\NN$ satisfying \eqref{eq:cond_n_tilde_q}.
	Let $r_1=r_1(n_0)$, $r_2=r_2(n)$.
	If  $r_1 > r_2$, 
	let $\hat{\rho} = \rho$ and $\hat n = n_0$.	Otherwise for $0 \le i \le r_2 -r_1$, 
	$$\rho_{i+1}=\tilde \Phi_{r_1+i}\left(\tilde f_{r_1+i}(\mathcal{Q}_{r_1+i}(\mu)),\rho_i\right),$$
	where $(\tilde \Phi_r)_{r=1}^\infty$ are the functions given by Lemma \ref{lem:trace_and_recover_Q} and $\rho_0=\rho$. Also, let $\hat n= q(r_2)$,  $\hat \rho  = \rho_{r_2 -r_1+1}$.
	
	The properties of the functions $\tilde \Phi_r$ given by Lemma \ref{lem:trace_and_recover_Q} ensure that
	\begin{equation}\label{eq:D_rho_rho_1}
	D_{n_0,\frac{3}{8}\epsilon_{n_0}}[\rho,\rho_1] \subseteq K_{r_1,n_0}
	\end{equation}
	and 
	\begin{equation}\label{eq:D_r_i_i_plus_i}
	\forall r_1 < r \le r_2,~  D_{q(r-1),\frac{3}{8}\epsilon_{q(r-1)}}[\rho_{r-r_1},\rho_{r-r_1+1}] 
	\subseteq K_{r,q(r-1)}.
	\end{equation}
	Applying  \eqref{eq:D_n_n_0} and \eqref{eq:D_triangle} together with \eqref{eq:D_rho_rho_1} and \eqref{eq:D_r_i_i_plus_i} it follows by triangle inequality that
	\begin{equation}\label{eq:D_n_rho_hat_rho}
	D_{n_0,\frac{3}{7}\epsilon_{n_0}}[\rho,\hat \rho] \subseteq D_{n_0, \frac{3}{8}\epsilon_{n_0}}[\rho, \rho_1]\bigcup _{r=r_1+1}^{r_2} D_{n_0, \frac{3}{8}\epsilon_{q(r-1)}}[\rho_{r-r_1}, \rho_{r-r_1+1}]  \subseteq  E_{n_0}.
\end{equation}
	Using \eqref{eq:mu_E_n_0_n_gamma_small} it follows that
	$$ \mu \left(D_{n_0,\frac{3}{7}\epsilon_{n_0}}[\rho,\hat \rho] \mid Z_{n_0} \right) < \gamma/2.$$
	Let
	$$\tilde {\rho} = \tilde \Phi\left(\mu,k,n,\frac{1}{8}\gamma, \delta, \hat \rho\right),$$
	where $\tilde \Phi$ is the function obtained by applying Lemma \ref{lem:improve_approx_emb_robust} with the sequence $(E_m)_{m=1}^\infty$ defined as above.
	Define:
	$$\overline{\Phi}\left(\mu,k,n,\delta,\gamma,\rho \right)= \tilde{\rho},$$
	and
	$$\overline{\overline{Z}}\left(\mu,k,n,\delta,\gamma,\rho \right)= \tilde Z\left(\mu,k,n,\frac{1}{8}\gamma, \delta, \hat \rho\right),$$
	where  $\tilde Z$ is the Borel function described in the statement of Lemma \ref{lem:improve_approx_emb_robust}.
	This completes the definition of $\overline{\Phi}\left(\mu,k,n,\delta,\gamma,\rho \right)$ 
	and $\overline{Z}\left(\mu,k,n,\delta,\gamma,\rho \right)$
	for relevant input parameters.
	
	So we assume now that $\rho \in \Mor_C(\Y,X)$ is a continuous $(k_0,n_0,\delta_0,\mu)$-approximate embedding such that \eqref{eq:cond_delta_0_n_0_k_0} and \eqref{eq:cond_n_tilde_q} hold.
		Since $\tilde \rho$ was obtained by applying the function $\overline{\Phi}$ of Lemma \ref{lem:improve_approx_emb_robust}, 
	$$ D_{\hat n, \frac{3}{8}\epsilon_{\hat n}}[\hat{\rho}, \tilde \rho]  \subseteq  Y \setminus \left(T^{(1-2\delta_{n})F_{n} \setminus \frac{\gamma}{8} F_{n}} Z_{ n} \right)= A\cup B $$
	where
	$$A= Y \setminus \left(T^{(1-2\delta_n)F_n}Z_n\right)  \text{ and } B=T^{\frac{\gamma}{8} F_{n}}Z_n.$$
	Because $1 \ll \hat n\ll n $, by \eqref{eq:n_ll_m}, using Lemma \ref{lem:mu_tower_subset} to bound $\mu (A)$ and Lemma \ref{lem:eq:mu_tower_subset_cond} to bound $\mu(B|Z_{\hat n})$ and the fact that $\hat n> N_0(\gamma)$, we have that,  
	$$ \mu \left( D_{\hat n, \frac{3}{8}\epsilon_{\hat n}}[\hat{\rho}, \tilde \rho]  \mid Z_{\hat n}\right) <\left( \epsilon_n+ 6d\delta_{n}\right) |F_{\hat n}|+ \frac{3}{2}(\gamma/8)+ 2 \hat{n}/n<\frac{1}{10}\epsilon_{\hat n}+ \frac{1}{10}\delta_{\hat n} + \frac32(\gamma/8)+ \delta_n/1000   <\gamma/5.$$	
Because $1\ll n_0 \ll \hat n$ it follows using Lemma \ref{lem:D_n_epsilon_D_m_epsilon} that 
$$ \mu \left( D_{ n_0, \frac{3}{8}\epsilon_{\hat n}}[\hat{\rho}, \tilde \rho]  \mid Z_{ n_0}\right) < \gamma/2.$$	
	Using \eqref{eq:D_triangle} once again we conclude that \eqref{eq:D_rho_tilde_rho} holds.
	
	Apply the triangle inequality to conclude that  whenever $n_0 \ll q(r) \ll n$ 
	$$ \left\{y \in Z_{q(r)} :~ d_X^{F_{q(r)}}\left(\tilde \rho(y),C_{r,\mathcal{Q}_r}\right) \ge 
	\frac{3}{8}\epsilon_{q(r)} 
	\right\}
	\subseteq $$
	$$\subseteq \bigcup_{r' =r+1}^{r_2} \left\{
	y \in Z_{q(r)} :~ d_X^{F_{q(r)}}\left( \rho_{r'-r_1-1}(y), \rho_{r'-r_1}(y)\right) \ge 
	\frac{3}{8}\epsilon_{q(r')} 
	\right\} \cup 
	\left\{
	y \in Z_{q(r)} :~ d_X^{F_{q(r)}}\left(\hat \rho(y),\tilde \rho(y)\right) \ge \frac{3}{8} \epsilon_{\hat n}
	\right\}.
	$$
	The inequality  \eqref{eq:y_C_r_mu_on_U_r}  now follows in a very similar fashion as did \eqref{eq:D_rho_tilde_rho}. We don't repeat the argument.
This completes the proof.
\end{proof}

\section{Universality of generic homeomorphisms on manifolds and realization of infinite entropy transformations via Lebesgue measure}\label{sec:universlity_generic}

In this short section we apply  Theorem \ref{thm:infinite_entropy} together with the results of  Guih\'eneuf-Lefeuvre \cite{MR2931648,MR3834653} to prove  Theorem \ref{thm:inf_universal_generic} (full $\infty$-universality of a generic homeomorphism on  a manifold $M$). Subsequently, we deduce  Theorem \ref{thm:lebesgue_universal}  (realizing an arbitrary measure preserving transformations via a homeomorphism  that preserves Lebesgue measure).	 

\begin{proof}[Proof of Theorem \ref{thm:inf_universal_generic}]
	Let $M$ be a finite dimensional compact topological manifold of dimension  at least $2$ and $\mu \in \Prob(M)$ a fully supported measure. 
	Let $\mathit{Homeo}(M,\mu)$ denote the space of homeomorphisms of $M$ that preserve the measure $\mu$.
	The group $\mathit{Homeo}(M,\mu)$  is a Polish group, and in particular a Baire space.
	By a result of  Pierre-Antoine Guih\'eneuf and Thibault Lefeuvre \cite[Theorem 3.17]{MR2931648}  infinite entropy is generic in $\mathit{Homeo}(M,\mu)$. By a more recent result of  Guih\'eneuf and Lefeuvre  \cite[Corollary 1.4]{MR3834653}    specification is also generic in $\mathit{Homeo}(M,\mu)$. If $(X,S)$ has specification, 
	then by Proposition \ref{prop:specification_implies_flexible}  it has a flexible marker sequence which is dense and has the same entropy. Theorem \ref{thm:infinite_entropy} now implies the result.
\end{proof}

We now show how to deduce Theorem \ref{thm:lebesgue_universal} from  Theorem \ref{thm:inf_universal_generic}. This reduction follows an argument of Lind and Thouvenot \cite{MR0584588}, who used it to show  that any ergodic $(Y,\mu,T)$ that has finite entropy is isomorphic to a Lebesgue measure-preserving homeomorphism of the torus.
\begin{proof}[Proof of Theorem \ref{thm:lebesgue_universal}:]
	By Theorem \ref{thm:inf_universal_generic} there exists a fully $\infty$-universal homeomorphism $\t \psi$.
	Let $(Y,\mu,T)$ be a free measure preserving $\Z$-system.
	Let $\nu$ be a fully-supported $\tilde \psi$-invariant measure on the torus so that $(\mathbb{T}_k,\nu,\tilde \psi)$ is isomorphic to     $(Y,\mu,T)$.
	An old  result of Oxtoby and Ulam \cite[Corollary 1]{MR0005803} implies that  there is a homeomorphism of the torus $g_\nu$ such that the pushforward of $\nu$ via $g_\nu$ is the Lebesgue measure.
	It follows that $ \psi= g_\nu \circ \tilde \psi \circ g_\nu^{-1}$ preserves Lebesgue measure and that $(Y,\mu,T)$ is isomorphic to $(\mathbb{T}_k,m_{\mathbb{T}_k},\psi)$ as a measure preserving system. 
\end{proof}

\section{Universality for graph homomorphisms}\label{section:universal_hom}

The following sections will be concerned with \emph{symbolic dynamical systems} or \emph{subshifts}.
We briefly recall some notation and basic definitions.

For a finite subset $\A$ (the ``alphabet''), the \emph{$\ZD$-full shift} over $\A$ is the topological $\ZD$ dynamical system  $(\A^{\ZD},S)$, where $S$ is the shift action  given by
$$S^{\mi}(x)_\mj= x_{\mi + \mj}.$$ 
A \emph{subshift} is a topological subsystem of the full-shift. 

If $X \subset \A^{\ZD}$ is a subshift, for every finite  $F \subset \ZD$  let
\begin{equation}
\L(X,F) = \left\{x\mid_F:~ x \in X \right\} \subset \A^F. \index{Definitions and notation introduced in Section 9 and Section 10!$\L(X, F)$}
\end{equation}
The elements of $\L(X,F)$ are called \emph{admissible $F$}-configurations for $X$.

Let $X$ be a $\ZD$ subshift, and  let  $g:\NN \to \NN$ be a function so that $\lim_{n\to \infty}\frac{g(n)}{n}=0$.

We say that  $\C = (\tilde C_n)_{n=1}^\infty$ with $\tilde C_n \subset \L(X,F_n)$ is a  \emph{flexible sequence of patterns}\index{Definitions and notation introduced in Section 9 and Section 10!flexible sequence of patterns} for $X$ with respect to $g$ if  for every  $k,n \in \NN$
any $F_{k+g(k)}$-spaced subset  $K$ contained in $F_{n-g(n)-k}$  and $W \in (\t C_k)^K$ there exists $w \in \t C_n$ so that 
\begin{equation}\label{eq:x_interpulate2}
S^{\mi}(w)\mid_{F_{k}}=W(\mi) \mbox{ for all } \mi \in K.
\end{equation}
We say that $\t \C$ as above is a 
 \emph{flexible marker sequence of patterns} \index{Definitions and notation introduced in Section 9 and Section 10!flexible marker sequence of patterns} if in addition
 \begin{equation}\label{eq:marker_property_subshift}
 \forall x \in X \mbox{ and } n \in \NN, \mbox{ the set } \left\{\mi \in \ZD~:~ S^{\mi}(x)\mid_{F_n}\in \t C_n \right\}
 \mbox{ is } F_{n-g(n)} \mbox{-spaced}. 
 \end{equation}
 
 For a flexible sequence of patterns $\tilde \C$ as above let  
 \begin{equation}
 h(\tilde C)= \limsup_{n \to \infty}\frac{1}{|F_n|}\log |C_n|.\index{Definitions and notation introduced in Section 9 and Section 10!$h(\tilde C)$ for a flexible sequence of patterns $\tilde C$}
 \end{equation}
 It is not difficult  to check that any $\ZD$ subshift $X$ that admits a flexible marker sequence of patterns $\t \C= (\t C_n)_{n=1}^\infty$ also admits a flexible marker sequence $\C=(C_n)_{n=1}^\infty \in X^{\NN}$ so that $h(\C)=h(\t \C)$.
 The idea is that if we choose a suitable function $\alpha:\NN\to \NN$ so that $\lim_{n\to \infty}\alpha(n)=+\infty$ and $\lim_{n\to \infty}\frac{\alpha(n)}{n}=0$, and choose for any $w \in \t C_{n +\alpha(n)}$ some $x(w) \in X$ so that $x(w)\mid_{F_{n+\alpha(n)}} = w$, then the sequence $\C = (C_n)_{n=1}^\infty$ defined by 
 	$$ C_n = \left\{x(w):~ w \in \t C_{n+ \alpha(n)}\right\}$$
 	will be a flexible marker  sequence for $(X,S)$ (with respect to  suitable sequences $(\delta_n)_{n=1}^\infty$ and $(\epsilon_n)_{n=1}^\infty$).
 
 By Theorem \ref{thm:spec_sequence_implies_univesality}, a $\ZD$ subshift $(X,S)$ that admits a flexible marker sequence of patterns $\t \C= (\t C_n)_{n=1}^\infty$ is almost Borel $h(\C)$-universal. If furthermore for every $n \in \NN$ and   $w \in \L(X,F_n)$ there exist $m \ge n$ and  $\t w \in \t C_m$ such that $\t w \mid_{F_n} = w$ then $X$ is fully ergodic $h(\t C)$-universal.
 
Let us also remark that for a slightly modified definition for a ``flexible sequence of patterns'', we can show that the existence of a flexible sequence of patterns implies the existence of a flexible marker sequence of patterns of equal entropy. The proof for this is quite similar to that of Proposition \ref{prop:specification_implies_flexible}. This slightly modified definition still implies universality and holds for the systems appearing in our main applications below (hom-shifts and rectangular tilings).  For the sake of presentation, we choose to bring a direct construction for existence of flexible marker sequence of patterns in the application below rather than proving this general implication.
   
Let us  briefly introduce graph homomorphisms and hom-shifts. For background and more we refer for instance to \cite{MR3702862,MR1719348,MR3743365}. 
In the category of graphs, a \emph{homomorphism} is a function between the vertex sets of two graphs that maps adjacent vertices to adjacent vertices. Note that graphs are one dimensional simplicial complexes and a graph homomorphism is just a simplicial map. If $\G$ and $\H$ are graphs, we denote by $Hom(\G, \H)$ the set of all graph homomorphisms from $\G$ to $\H$. For a finite graph $\H$ we let $Hom(\ZD,\H)$ \index{Definitions and notation introduced in Section 9 and Section 10!$Hom(\ZD, \H)$} denote the space of graph homomorphisms from the standard Cayley graph on $\ZD$ to $\H$. As $\ZD$ acts on its Cayley graph by graph automorphisms, it also acts on $Hom (\ZD,\H)$. Viewed as a closed shift-invariant subset of $V_\H^{\ZD}$,   $Hom(\ZD,\H)$ is a subshift. Here graphs will be  undirected and simple in the sense that there is at most one edge between any pair of vertices, but we allow self-loops. Hom-shifts are shift spaces that arise as the space of graph homomorphisms from $\ZD$ to a fixed graph $\H$. For instance, the space of proper $n$-colorings is the set of graph homomorphisms from $\ZD$ to a complete graph on $n$ vertices.

Equivalently, hom-shifts are nearest neighbor $\ZD$ subshifts of finite type that are symmetric with respect to permutations and reflections along the cardinal directions (they are ``isotropic and symmetric''). 
They are very special because of their inherent symmetry, and contain various interesting examples  of shift spaces.
 
Our goal is to prove the following:
\begin{thm}\label{thm: universality of hom-shifts}
	If $\H$ is a finite connected graph which is not bipartite,
	then the subshift $\Hom(\ZD,\H)$ admits a 
	flexible marker sequence of patterns $\t \C= (\t C_n)_{n=1}^\infty$ with $h(\t \C)= h(\Hom(\ZD,\H))$ so that $\bigcup_{n=1}^\infty \t C_n$ is dense in $X$. Thus, $\Hom(\ZD,\H)$ is almost Borel universal. Furthermore, it is fully ergodic universal. 
\end{thm}  

To prove Theorem \ref{thm: universality of hom-shifts}, we will identify a flexible marker sequence of patterns in  $\Hom(\ZD,\H)$ as follows:
We will identify a subset $F \subset \ZD$ with the corresponding induced subgraph of the standard Cayley graph of $\ZD$.
Let $\H =(V_\H,E_\H)$ be a finite connected graph. For $\mi \in \ZD$ let $\parity(\mi) \in \{0,1\}$ denote the parity of the sum of the coordinates of $\mi$.
Given  neighboring vertices $v_0,v_1 \in V_\H$ define for $n \ge 1$
\begin{equation}\label{eq:checkerboard_patterns_tag}
C_{n}^{(v_0,v_1)} = \left\{ a\in \Hom(F_{n},\H):~ a_{\mi} = v_{\parity(\mi)}~ \mbox{ for all } \mi \in \left(F_{n}\setminus F_{n-1}\right)  \right\}.\index{Definitions and notation introduced in Section 9 and Section 10!$C_n^{(v_0, v_1)},$ a set of patterns for hom-shifts}
\end{equation} 

 We assume  that $|E_\H|>1$ (otherwise $\Hom(\ZD,\H)$ is either empty or consists of a single point). Because $\H$ is connected this implies that there exists $v_0 \in V_\H$ that is incident to at least two edges.
So we can find  $v_1,v_2 \in V_\H$ such that $v_1 \ne v_2 $ and $(v_0,v_1),(v_0,v_2) \in E_\H$. 
For $n \ge 1$ let
\begin{equation}\label{eq:checkerboard_patterns_marker}
\tilde C_{n+1} = \left\{ a\in C_{n+1}^{(v_0,v_1)}:  a\mid_{F_n} \in C_{n}^{(v_0,v_2)}  \right\}.\index{Definitions and notation introduced in Section 9 and Section 10!$\tilde C_n$, a set of patterns for hom-shifts}
\end{equation}

In other words the patterns in $\t C_{n}$ have two layers of ``checkerboard boundaries'', so that for each of the two layers there is a different ``color'' for the odd sites.
For this reason when $d\ge 2$ it is not possible for two patterns in $\t C_n$ to overlap non-trivially except at most on the boundary.
In other words, for any $x \in \Hom(\ZD, \H)$ the set 
\begin{equation}\label{eq:t_C_marker}
\left\{\mi \in \ZD:~ S^{\mi}(x)\mid_{F_{n+1}} \in \t C_{n+1} \right\} \mbox{ is } F_n \mbox{-spaced}.
\end{equation}

Our goal is to show that the sequence $\t \C = (\t C_{n+1})_{n=1}^\infty$ is a flexible marker sequence of patterns with $h(\tilde \C)=h(X,S)$.
We will do this by showing that  for any neighboring vertices $v_0, v_1$ the sequence $\C^{(v_0,v_1)}= (C_n^{(v_0,v_1)})_{n=1}^\infty$ is a flexible sequence of patterns for $\Hom(\ZD,\H)$.
It is clear that $|\t C_{n+1}| = |C_n^{(v_0,v_2)}|$ because every $a \in  C_n^{(v_0,v_2)}$ is the restriction of precisely one $\t a \in \t C_{n+1}$. In particular $h(\t \C)=h(\C^{(v_0,v_2)})$.

\begin{prop}\label{prop:checkerboard_flexible}
		Let $\H$ be a connected graph which is not bipartite.
	For any $(v_0,v_1) \in E_\H$ the sequence $\t \C= (\t C_n)_{n=1}^\infty$ given by \eqref{eq:checkerboard_patterns_marker} is flexible  marker sequence of patterns in $\Hom(\ZD, \H)$ and 
	\begin{equation}\label{eq:h_C_eq_h_Hom}
	h(\t \C)=  h(\Hom(\ZD,\H)).
	\end{equation}
	Moreover, there exists $N\in \N$ such that	for any $a \in \L(\Hom(\ZD,\H),F_n)$ there exists $w \in \t C_{(d+1)n+N+2}$ so that $w \mid_{F_n}=a$.
\end{prop}
We prove Proposition \ref{prop:checkerboard_flexible} in several steps.

\begin{lem}\label{basic_hom_contraction_ZD}
There exists a graph homomorphism  $\tau:\ZD \to \ZD$  from the Cayley graph of $\ZD$ to itself  so that:
\begin{enumerate}
	\item For any 
	$\mi \in \ZD \setminus \{0\}$,
 $\|\tau(\mi)\|_1 < \|\mi\|_1$.  
\item 
$\tau(\vec{0})=\vec{e}_1$.
\item For any $\mi \in \ZD$, $\tau(\mi)$ is adjacent to $\mi$. 
\item For any $n \in \NN$, if 
$\mi,\mj  \in \ZD$ are adjacent vertices with $\mi \not\in F_n$ and $\mj \in F_n$ then $\tau(\mi)=\mj$.
\end{enumerate}
\end{lem}
\begin{proof}
	We define $\tau:\ZD\to \ZD$ as follows: $\tau(\vec{0})=\vec{e}_1$, and if $\mi = (i_1,\ldots,i_d) \ne \vec{0}$ 
	let 
	$$\xi(\mi) = \min \{ 1\le t \le d:~ |i_t|= \|\mi\|_\infty \}$$
	and let
	$$
	\tau(\mi) =\begin{cases}
	\mi - \vec{e}_{\xi(\mi)} & i_{\xi(\mi)} >0\\
	\mi + \vec{e}_{\xi(\mi)} & i_{\xi(\mi)} <0.	
	\end{cases}
	$$
	
	Suppose $\mi,\mj  \in \ZD$ are adjacent vertices then there exists $1\le t \le d$ so that $\mi = \mj \pm \vec{e}_t$.
	Without loss of generality assume $\|\mi\|_1 > \|\mj \|_1$.
	If $\xi(\mi)\ne t$ then $\|\mj \|_\infty = \|\mi \|_\infty$ and $\xi(\mi)= \xi(\mj)$ so $\tau(\mi)=\tau(\mj) \pm \vec{e}_t$.
	If $\xi(\mi)= t$ then $\tau(\mi)=\mj$, and so $\tau(\mi)=\tau(\mj) \pm \vec{e}_{\xi(\mj)}$.
	In either case, we see that $\tau(\mi)$ and $\tau(\mj)$ are adjacent vertices. This shows that $\tau:\ZD \to \ZD$ is indeed a graph homomorphism. 
	Properties $(1),(2),(3)$ follow directly from the definition.
	Let us check property $(4)$:
	If $\mi,\mj  \in \ZD$ are adjacent vertices and in addition 
	$\mi \not\in F_n$ and $\mj \in F_n$ then it follows that $\mj = \mi \pm \vec{e}_{\xi(\mi)}$ and    $\mj = \tau(\mi)$.

\end{proof}

\begin{lem}\label{lem:hom_ZD_B_n}
For any $n \in \NN$ there exists a graph homomorphism  $\tau_n:\ZD \to \ZD$  from the Cayley graph of $\ZD$ to itself so that the following holds.
\begin{itemize}
	\item $\tau_n(\mi)=\mi$ for all $\mi  \in F_n$.
	\item $\tau_n(\ZD)= F_n$.
	\item If $\|\mi\|_1 \ge 2nd$ then either $\tau_n(\mi) =\vec{0}$ or $\tau_n(\mi)=\vec{e}_1$, according to the parity of $\mi$.
\end{itemize}
\end{lem}
\begin{proof}
	Let $\tau:\ZD \to \ZD$ be as in the statement of Lemma \ref{basic_hom_contraction_ZD}.
	Define $\tau_n:\ZD \to \ZD$  as follows:
	\begin{equation}
	\tau_n(\mi) = \begin{cases}
	\mi & \mi \in F_n\\
	\tau^{2k}(\mi) & \mi \in \tau^{-k}(F_n)\setminus \tau^{-(k-1)}(F_n).
	\end{cases}
	\end{equation}
	
Let us check that $\tau_n$ is a graph homomorphism: 
Let us say that $\mi \in \ZD\setminus F_n$ is of level $k >0$ if $\mi \in \tau^{-k}(F_n)\setminus \tau^{-(k-1)}(F_n)$ (if $\mi \in F_n$ we say it is of level $0$).
Suppose $\mi,\mj \in \ZD$ are adjacent. If both $\mi$ and $\mj$ are of the same level $k >0$,
then $\tau_n(\mi)=\tau^{2k}(\mi)$ and $\tau_n(\mj)=\tau^{2k}(\mj)$ and so $\tau_n(\mi)$ is adjacent to $\tau_n(\mj)$ because $\tau^{2k}$ is a graph homomorphism.
Otherwise, without loss of generality $\mi$ is of level $k >0$ and $\mj$ is of level $k-1$. 
Since $\tau$ is a graph homomorphism, $\tau^{k-1}(\mi)\notin F_n$ is adjacent to $\tau^{k-1}(\mj)\in F_n$. But by the properties of $\tau$ we have that $\tau^{k}(\mi)=\tau^{k-1}(\mj)$ and consequently $\tau_n(\mi)=\tau^{2k}(\mi)$ is adjacent to $\tau_n(\mj)=\tau^{2k-2}(\mj)$.
This shows that $\tau_n$ is a graph homomorphism. 

It is clear from the definition that $\tau_n(\mi)=\mi$ for $\mi \in F_n$.
Because $\tau (F_n) \subseteq F_n$ it follows that $\tau_n(\ZD) = F_n$.
 For any $\mj \in F_n$, and $k \ge 2dn$,  either $\tau^{k}(\mj)=\vec{0}$ or  
 $\tau^k(\mj)=\vec{e}_1$. Also, if $\|\mi\|_1 \ge 2nd$ then 
there exists $k \ge dn$ so that $\mi \in \tau^{-k}(F_n)\setminus \tau^{-(k-1)}(F_n)$. This shows that if $\|\mi\|_1 \ge 2nd$ then either $\tau_n(\mi) =\vec{0}$ or $\tau_n(\mi)=\vec{e}_1$, according to the parity of $\mi$.
\end{proof}

\begin{lem}\label{lem:C_v_0_v_1_w_0_w_1}
If $\H$ is a connected graph which is not bipartite then there exists $N \in \NN$ so that for any $(v_0,v_1),(w_0,w_1) \in E_\H$, $n \in \N$, $k \ge N$ and $a \in C_n^{(v_0,v_1)}$ there exists $\t a \in C_{n+k}^{(w_0,w_1)}$ such that $\t a\mid_{F_n}= a$. 
\end{lem}
\begin{proof}
	The fact that $\H$ is connected and not bipartite implies that  there exists $N \in \NN$ so that for every $n \geq N$ there exists a path of length precisely  $n$ between any two vertices of $\H$ (a path of length $m$ is a graph homomorphism from $\{0,1,\ldots, m\}$ to $\H$).
	Suppose $n \in \NN$, $k \ge N+1$, $(v_0,v_1),(w_0,w_1) \in E_\H$ and $a \in C_n^{(v_0,v_1)}$.
	We split the proof into two cases depending on the parity of $k$.
	
	Let $k\in \N$ be even. Then we can choose a path $v_0, v_1, v_2, \ldots, v_{k+1}$  such that $v_k=w_0$ and $v_{k+1}=w_1$ and define $\t a \in C_{n+k}^{(w_0,w_1)}$ by the implicit conditions
	\begin{eqnarray*}
		\t a|_{F_n}&=&a\\
		\t a|_{F_{n+t}}&\in& C_{n+t}^{(v_t, v_{t+1})}\text{ if }1\leq t\leq k\text{ is even}\\
		\t a|_{F_{n+t}}&\in& C_{n+t}^{(v_{t+1}, v_{t})}\text{ if }1\leq t\leq k\text{ is odd}.
	\end{eqnarray*}
	The case of odd  $k$ is identical, except that we set $v_k=w_1$ and $v_{k+1}=w_0$.
\end{proof}

\begin{lem}\label{lem:C_k_full_support}
	Let $\H$ be a connected graph which is not bipartite and $(v_0, v_1)\in E_{\H}$.
	 There exists $N\in \N$ such that for all $n\in \N$, $k\geq N+d$ and $a\in \Hom(F_n, \H)$ there exists $\t a\in C_{2dn+k}^{(v_0, v_1)}$ for which $\t a|_{F_n}=a$. 
\end{lem}
\begin{proof}
	Given $a \in \Hom(F_n,\H)$ we can define $\hat a \in \Hom(F_{2nd},\H)$
	by 
\begin{equation}
\hat a_\mi = a_{\tau_n(\mi)},
\end{equation}
where $\tau_n:\ZD \to \ZD$ is the graph homomorphism given by Lemma \ref{lem:hom_ZD_B_n}. Note that $a_{\tau_n(\mi)}$ is well defined for any $\mi \in \ZD$ because $\tau_n(\mi) \in F_n$. Because it is a composition of graph homomorphisms, $\hat a \in  \Hom(F_{2nd},\H)$. Because the restriction of $\tau_n$ is the identity it follows that  $\hat a\mid_{F_n}=a$. By the last property of $\tau_n:\ZD \to \ZD$ in Lemma \ref{lem:hom_ZD_B_n} it follows that there exists $(w_0,w_1) \in E_\H$ 
so that for any $\mi \in F_{2nd}\setminus F_{2nd-1}$ either $\hat a_\mi = w_0$ or $\hat a_\mi = w_1$, according to the parity of $\mi$. 
In other words, $\hat a \in C_{2nd}^{(w_0,w_1)}$ for some $(w_0,w_1) \in E_\H$.
Let $N$ be the integer given by  Lemma \ref{lem:C_v_0_v_1_w_0_w_1}, and suppose $k \ge N$.
Apply Lemma \ref{lem:C_v_0_v_1_w_0_w_1} to find $\t a \in C_{2nd + k}^{(v_0,v_1)}$ so that $\t a \mid_{F_{2nd}} = \hat a$.
\end{proof}

Note that Lemma \ref{lem:C_k_full_support} in particular tells us that $\L(\Hom(\ZD,\H),F_n)= \Hom(F_n,\H)$. 
\begin{lemma}\label{lem:checkerboard_flexbile}
	Let $\H$ be a connected graph which is not bipartite.
	There exists $N \in \NN$ so that for any $(v_0,v_1),(w_0,w_1) \in E_\H$,  $k, n \in \NN$, any $F_{k+N}$-spaced subset $K \subset F_{n-N-k-2}$ and any $W \in (C_{k}^{(v_0,v_1)})^K$ there exists 
	$w \in C_{n}^{(w_0,w_1)}$ so that 
	$$ S^{\mi}(w)\mid_{F_k} = W(\mi) \mbox{ for all } \mi \in K.$$	
\end{lemma}
\begin{proof}
Let $N \in \NN$ be as given Lemma \ref{lem:C_v_0_v_1_w_0_w_1}. Suppose  $K \subset F_{n-N-k-2}$ is $F_{k+N}$-spaced and  $W \in (C_{k}^{(v_0,v_0)})^K$.
Apply Lemma \ref{lem:C_v_0_v_1_w_0_w_1} to obtain $\hat W \in (\Hom(F_{k+N},\H))^K$ so that for any $\mi \in K$, 
$\hat W(\mi) \in C_{k+N}^{(w_{\parity(\mi)},w_{1-\parity(\mi)})}$ and $\hat W(\mi)\mid_{F_k} = W(\mi)$.
Now define $w \in C_{n}^{(w_0,w_1)}$ as follows:
$$
w_\mj = \begin{cases}
\hat W(\mi)_{\mj-\mi} & \mbox{ if } \mi \in K \mbox{ and } \mj \in \mi + F_{k +  N}\\
w_{\parity(\mj)} & \mbox{ if } \mj \not\in \bigcup_{\mi \in K}(\mi + F_{k +N}).
\end{cases}
$$
Then $w \in C_{n}^{(w_0,w_1)}$ has the desired properties.

\end{proof}
Given Lemma \ref{lem:checkerboard_flexbile} it is easy to prove that $\t \C$ is a flexible sequence of patterns, and the only  somewhat non-trivial part is the ``entropy estimate''  $h(\t \C)= h(\Hom(\ZD,\H))$. For this purpose it will be convenient to introduce another sequence of patterns:
\begin{equation}\label{eq:almost_checkerboard_patterns}
\hat C_{n} = \left\{ a\in \Hom(F_{n},\H):~  \mbox{ If }\mi,\mj \in \left(F_{n}\setminus F_{n-1}\right) \mbox{ and }  (\mi- \mj) \in 2\ZD \mbox{ then } a_{\mi} = a_\mj \right\}.\index{Definitions and notation introduced in Section 9 and Section 10!$\hat C_n$,  a set of patterns for hom-shifts}
\end{equation}
By definition  $\hat C_{n}$ is precisely the set of graph homomorphisms $w \in \Hom(F_{n},H)$ such that $w\mid_{F_n \setminus F_{n-1}}$ is the composition of some $\t w \in \Hom(\{0,1\}^d,\H)$ with natural graph homomorphism 
$\mi \mapsto (\mi \mod 2\ZD)$  from the standard Cayley graph of $\ZD$ to the standard Cayley graph of $\ZD/2\ZD$.
Clearly, $ \t C_n \subset \hat C_n$. We will  prove  that $h(\t \C)= h(\Hom(\ZD,\H))$ by proving that $h(\t \C)=h(\hat \C)$ and $h(\hat \C)=h(\Hom(\ZD,\H))$.
The fact that $h(\t \C)=h(\hat \C)$ is a direct consequence of the following:
\begin{lem}\label{lem:ext_C_n_to_tilde_C_n} 
	If $\H$ is a connected graph  
	then 
	for every $n \in \NN$, $k \ge 2d$ and  $a \in \hat C_n$ there exists $(v_0,v_1)\in E_\H$ and $\tilde a \in C_{n+k}^{(v_0,v_1)}$ such that
	$$\tilde a \mid_{F_n} = a.$$
\end{lem}
\begin{proof}
	If $a \in \hat C_n$ there exists $a' \in \Hom(\{0,1\}^d,\H)$ such that for every $\mi \in F_n \setminus F_{n-1}$ we have $a_\mi = a'_{\hat \tau(\mi)}$ where $\hat \tau:\ZD \to \{0,1\}^d$ is given by $\hat \tau(\mi) = \mi \mod 2\ZD$.
	Note that $\hat \tau$ is a graph homomorphism.
	Let $\tau :\ZD \to \ZD$ be as in Lemma 
	\ref{basic_hom_contraction_ZD}.
	Define $\t a :\ZD\to \H$ by 
	$$\t a_{\mi}=
	\begin{cases}
	a_\mi & \mi \in F_n\\
	a_{\tau^k(\hat \tau(\tau^{k}(\mi))}' & \mi \in \tau^{-k}(F_{n}) \setminus \tau^{-(k-1)}(F_n).
	\end{cases}
	$$ 
	It follows by the properties of $\tau$ that $\t a\in \Hom(\ZD,\H)$.
	Exactly as in the proof of Lemma \ref{lem:C_k_full_support} it follows that whenever $k \ge 2d$,   $\t a \mid_{C_{n+k}} \in C_{n+k}^{(v_0,v_1)}$, where 
	either $(v_0,v_1)= (a'_{\vec{0}},a'_{\vec{e}_1})$ or $(v_0,v_1)= (a'_{\vec{e}_1},a'_{\vec{0}})$, according to the parity of $k$.
	It is clear from the definition that $\t a \mid_{F_n} =a$.
\end{proof}

To see that $h(\hat \C)= h(\Hom(\ZD,\H))$ we prove the following:
\begin{prop}\label{prop: estimate of entropy}
	There exists a constant $c >0$ so that for every $n \in \NN$
	\begin{equation}
	\frac{|\hat C_n|}{|\Hom(F_n,\H)|} \ge e^{-c n^{d-1}}.
	\end{equation}
	In particular,
	\begin{equation}
	h(\C) = h(\Hom(\ZD,\H),S).
	\end{equation}	
\end{prop}
Proposition \ref{prop: estimate of entropy} is essentially the statement that the probability that a uniform random graph homomorphism from the box $F_n \subset \ZD$ to a finite graph $\H$ takes only two values on the ``boundary'' of the box $F_n$ is at most exponentially small in the size of the boundary.

We include a proof since we did not find one in the literature.
As pointed out to us by Yinon Spinka, in the case where $\H$ is the complete graph on $q$ vertices and so $\Hom(\ZD,\H)$ corresponds to $q$-colorings of $\ZD$, a short proof for Proposition \ref{prop: estimate of entropy} follows from the fact that for any finite  bipartite graph $\G$, the probability that a uniform random $q$-coloring assigns only two colors to a subset $W \subset V_\G$ is at least $q^{-|W|}$. For $q=3$, this follows from arguments in \cite[Lemma 5.1]{MR3374637} but the same idea works for $q>3$. For $3$-colorings it also follows from \cite[Proposition 2.1]{benjaminimossel2000}. Our proof of  Proposition \ref{prop: estimate of entropy} is based on an idea that is  sometimes called \emph{reflection positivity}. We mostly follow Biskup \cite{Biskup2009}, with suitable adaptations to our setting. 
   
Let us make a  brief detour to introduce a  notion of ``reflection'' and ``reflection positivity'' for ``discrete random fields'' (namely, random functions on discrete graphs).

\begin{defn}
	Let $\G=(V,E)$ be a (discrete, finite or countable) graph. An automorphism $R \in \Aut(\G)$ is called a \emph{reflection} if:
	\begin{enumerate}[(i)]
		\item $R$ is an involution: $R^2 = \mathit{Id}$.
		\item The complement of the fixed points of $R$ in $V$  has precisely two connected components  $V_1,V_2$ that are mapped bijectively onto each other. We refer to these as the  \emph{sides} of the reflection. 
	\end{enumerate}
   A reflection also induces a self map on $\A^V$. We say that  $\mu \in \Prob(\A^V)$ has \emph{reflection positivity} with respect to $R$ if $R$ preserves $\mu$ (that is $ \mu \circ R^{-1}= \mu$) and for every $W \subset V_1$ and $a \in \A^{W}$,
   \begin{equation}\label{eq:RP_sets}
   \mu \left([a]_W \cap [R(a)]_{R(W)} \right) \ge \mu \left([a]_W \right)^2. 
   \end{equation}
\end{defn}

A \emph{Markov random field} with respect to a graph $\G$ is a Borel probability measure $\mu \in \Prob(\A^{V_\G})$ with the following conditional independence property:
Whenever $A_1,A_2,B \subset V_\G$ and $B$ disconnects  $A_1$ and $A_2$ then 
the sigma-algebras generated by the restrictions to $A_1$ and $A_2$ are independent conditioned on the sigma-algebra generated by the restriction to $B$.
\begin{prop}\label{prop:MRF_RP}
	If $\mu \in \Prob(\A^{V_\G})$ is a Markov random field with respect to $\G$ and $R \in \Aut(\G)$ is a reflection that preserves $\mu$, then $\mu$ is reflection positive with respect to $R$.
\end{prop}
\begin{proof} Let $\mu$ and $R$ be as in the statement of the proposition, let $F \subset V_\G$ be the fixed points of $R$, and let $V_1 \subset V_\G$ be one of the sides for $R$.
	 Choose $W \subset V_1$, $a \in \A^W$ and $b \in \A^F$.
	 Because $\mu$ is a Markov random field, and $F$ disconnects  $W$ and $R(W)$
	$$\mu([a]_W \cap [R(a)]_{R(W)} \mid [b]_{F})=
	\mu([a]_W  \mid [b]_{F}) \cdot
	\mu([R(a)]_{(R(W)} \mid [b]_{F})=
	\mu([a]_W  \mid [b]_{F})^2.
	$$
	The last equality follows because of invariance of $\mu$ under $R$. Taking expectation over $b \in \A^F$ with respect to $\mu$  and applying the Cauchy-Schwarz inequality we conclude that \eqref{eq:RP_sets} holds.	
\end{proof}

For $n \in \NN$ let $T_n$ denote the Cayley graph of the group $\left( \ZZ / 2n\ZZ \right)^d$ with respect to the standard generators; this is a ``discrete torus''.
\begin{lem}\label{lem:many_periodic_conf_hom}
	There exists a constant $c>0$ so that for every $n \in \NN$,
	\begin{equation}
	\frac{|\Hom(T_n,\H)|}{|\Hom(F_n,\H)|} \ge e^{-c n^{d-1}}.
	\end{equation}
\end{lem}

\begin{proof}
	Let $\mu$ denote the uniform measure on $\Hom(F_n,\H)$. For all reflections $R$ along coordinate hyperplanes,  $\mu$ is an $R$-invariant Markov random field. By Proposition \ref{prop:MRF_RP}, it is reflection positive with respect to all such reflections $R$. 
	We can find $ a\ \in \V_\H^{(F_n \setminus F_{n-1}) \cap (\ZZ_+)^d}$  such that 
	$$\mu([a]_{(F_n \setminus F_{n-1}) \cap (\ZZ_+)^d})\geq |\H|^{-|(F_n \setminus F_{n-1}) \cap (\ZZ_+)^d|}.$$
By successive applications of ``reflection positivity''  along the $d$ hyperplanes corresponding to the cardinal directions, we get a pattern $\t a$ on $F_n\setminus F_{n-1}$ which is periodic, meaning, $\t a_\mi =\t a_\mj$ whenever $\mi -\mj \in 2n\ZD$ and so that 
$$\mu([\t a]_{F_n \setminus F_{n-1}} )\geq \left(\mu([a]_{F_n \setminus F_{n-1}) \cap (\ZZ_+)^d}) \right)^{2^{d}}\geq|\H|^{-2^d|F_n \setminus F_{n-1}) \cap (\Z_+)^d|}.$$
This gives us the required result because there is a natural bijection between the periodic patterns and elements of $Hom(T_n, \H)$.
\end{proof}
\begin{remark}
Lemma \ref{lem:many_periodic_conf_hom} has the following dynamical consequence about $\Hom(\ZD,\H)$: The subshift  $\Hom(\ZD,\H)$ has ``many'' periodic points in the sense that 
\begin{equation}
h(\Hom(\ZD,\H),S)= \lim_{n \to \infty}\frac{\log |P_{2n}(\Hom(\ZD,\H),S)|}{(2n)^d},
\end{equation}
where $P_{2n}(\Hom(\ZD,\H),S)$ is the set of points in $\Hom(\ZD,\H),S)$ that are stabilized under the subaction of $(2n\ZZ)^d$.
See for instance \cite{symmtricfriedlan1997}. 
\end{remark}
The conceptual reason  for introducing  $\Hom(T_n,\H)$ as an auxiliary object  is that $T_n$ has  additional symmetries coming from reflection and also the action of $(\ZZ/2n\ZZ)^d$.
\begin{lem}\label{lemma: checkerboard on one side}
	There exists a constant $c>0$ so that for every $n \in \NN$,
\begin{equation}
\frac{|\{a\in Hom(T_n, \H)~:~a_{(0, \mi)}=a_{(0,\mj)}\text{ whenever }\mi-\mj\in (2\Z)^{d-1}\}|}{|\Hom(T_n,\H)|} \ge e^{-c n^{d-1}}.
\end{equation}	
\end{lem}
\begin{proof}
Let $\nu$ denote the uniform measure on $Hom(T_n, \H)$ and observe that $\nu$ is invariant and reflection positive with respect to reflections of the type $R_{r,k}:T_n\to T_n$ given by
\begin{equation}
R_{r,k}(\mi):= \mi - (2i_k-2r)\m e_k \text{ for $1\leq k\leq d$ and $0\leq r\leq n-1$}.\label{equation: reflections on torus}
\end{equation}
Again we begin with a pattern $a$ on $\{0\}\times \{0,1\}^{d-1}$, chosen such that
$$\nu([a]_{\{0\}\times \{0,1\}^{d-1}})\geq |\H|^{-2^{d-1}}.$$
Successive reflections of $a$ by
$$R_{1, 2},  R_{2, 2},  R_{4, 2}, \ldots, R_{2^{\lfloor\log_2{n}\rfloor}, 2}, \   R_{1, 3},  R_{2, 3},  R_{4, 3}, \ldots, R_{2^{\lfloor\log_2{n}\rfloor}, 3}, \ldots, R_{2^{\lfloor\log_2{n}\rfloor}, d}$$ and applying reflection positivity gives us a pattern $\t a$ on $\{0\}\times [0,n-1]^d$ such that 
$$\t a_{\mi}=a_{\mj}\text{ whenever }\mi-\mj\in (2\Z)^d$$
and 
$$\nu([\t a]_{\{0\}\times [0,n-1]^d})\geq |\H|^{-(2n)^{d-1}}.$$
(In fact, we might get a slightly bigger pattern but we do not need it for this proof.) Finally by successively reflecting $\t a$ ($R_{r,k}$ for $r=0$ and $2\leq k \leq d$ in \eqref{equation: reflections on torus}) and applying reflection positivity we get a pattern $a'$ on $\{0\}\times [0,2n-1]^d$ such that 
$$a'_{\mi}=a_{\mj}\text{ whenever }\mi-\mj\in (2\Z)^d$$
and 
$$\nu([a']_{\{0\}\times [0,2n-1]^d})\geq |\H|^{-(4n)^{d-1}}.$$
This completes the proof.
\end{proof}

We can now complete the proof:
\begin{proof}[Proof of Proposition \ref{prop: estimate of entropy}]
Again, let $\mu$ be the uniform measure on $Hom(F_n, \H)$. The measure $\mu$ is invariant and reflection positive with respect to ``diagonal'' reflections of the type 
$R_{k, \pm}: F_n\to F_n$ given by 
$$R_{k, \pm}(\mi):=\mi-(i_1\m e_1+i_k\m e_k)\pm (i_k\m e_1 + i_1 \m e_k)\text{ for }2\leq k \leq d.$$

By Lemmas \ref{lem:many_periodic_conf_hom} and \ref{lemma: checkerboard on one side} we have a pattern $a$ on $\{n\}\times[-n,n]^{d-1}$ such that $a_\mi=a_\mj$ whenever $\mi-\mj\in (2\Z)^d$ and 
$$\nu([a]_{\{n\}\times[-n,n]^{d-1}})\geq e^{-C n^{d-1}}$$
for some constant $C$ independent of $n$. Now we apply successive reflections to $a$. It may so happen that a certain reflections might result in patterns which are larger than the side associated with the next reflection. In this case, we just restrict that pattern to the side which contains $\{n\}\times[-n,n]^{d-1}$ before continuing.

By successive reflections of $a$ along the diagonals by
$$R_{2, +},R_{3, +}, \ldots, R_{d, +}, R_{2, -},R_{3, -}, \ldots, R_{d, -}$$
 and applying reflection positivity we get a checkerboard boundary pattern $\t a$ such that 
$$\mu([\t a]_{F_n \setminus F_{n-1}})\geq e^{-2^{2(d-1)}C n^{d-1}}.$$
This completes the proof. 
\end{proof}

\begin{proof}[Proof of Proposition \ref{prop:checkerboard_flexible}]
The result that $\C$ is a flexible sequence of patterns follows  immediately from Lemma \ref{lem:checkerboard_flexbile}. The fact that it is furthermore  a flexible marker sequence of patterns follows from \eqref{eq:t_C_marker}.
By Lemma \ref{lem:ext_C_n_to_tilde_C_n} combined with Lemma \ref{lem:C_v_0_v_1_w_0_w_1}  it follows that  there exists $N \in \NN$ such that $|\hat C_n| \le |\t C_{n+N}|$ so  
\begin{equation}\label{eq:h_t_C_h_top}
h(\t \C) \ge \limsup_{n \to \infty}\frac{\log |\hat C_n|}{|F_n|}.
\end{equation}
As we explained, from Lemma  \ref{lem:C_k_full_support} it follows that  $\L(\Hom(\ZD,\H),F_n)=\Hom(F_n,\H)$.
By Proposition \ref{prop: estimate of entropy} the right hand side of \eqref{eq:h_t_C_h_top} is equal to the topological entropy of the hom-shift $\Hom(\ZD,\H)$.
	The statement beginning with  ``Moreover'' follows  by applying Lemma \ref{lem:C_k_full_support} together with Lemma \ref{lem:C_v_0_v_1_w_0_w_1}.
	
\end{proof}

\begin{remark}
	If $\H$ is not bipartite then $\Hom(\ZD,\H)$ is not universal because any invariant measure admits a set of measure $\frac{1}{2}$ which is invariant under the  $(2\Z)^d$ subaction. Nevertheless, in this case it is still true that the  $(2\Z)^d$ subaction on $\Hom(\ZD,\H)$  is universal, as can be shown by suitably adapting our  proof of  Proposition \ref{prop:checkerboard_flexible}. Even without our new result, ergodic universality of the $(2\Z)^d$ subaction on $\Hom(\ZD,\H)$ can be deduced by applying  the  \c{S}ahin-Robinson universality result \cite{MR1844076} to a certain sequence of closed subsystems that are strongly irreducible and have dense periodic points with entropy arbitrarily close to $Hom(\Z^d, \H)$. As this is not the focus of this paper, we do not include a  proof. 
\end{remark}

\begin{remark}
	Using the entropy formula for actions of finite-index subgroups, it is not difficult to see that universality (in the ergodic or almost Borel sense) implies universality with respect to the action of any finite-index subgroup.
\end{remark}
	 As stated in the introduction Theorem \ref{thm: universality of hom-shifts} has consequences  related  to some  problems in  ``Borel graph theory'' discussed for instance  in \cite{gao2018continuous} and  references within.
	A \emph{Borel graph} is a pair  $\G=(V,E)$, where $V$ is a Borel space and $E \subset V \times V$ is a Borel subset of $V \times V$. 
	A Borel probability measure $\mu$ on $V$ is called $\G$-invariant if $\mu \circ \phi^{-1} = \mu$ for any Borel bijection $\phi:V \to V$ that satisfies $(v,\phi(v)) \in E$ for every vertex $v\in V$.
	A subset of the vertices of $\G$ is called $\G$-null if it has zero measure with respect to any
	measure which is $\G$-invariant. 
	A \emph{Borel graph homomorphism}  from $\G$ to another Borel graph $\H$ is a Borel map from the vertices of $\G$ to the vertices of $\H$ that takes edges into edges.
	A \emph{Borel almost graph homomorphism} from $\G$ to $\H$ is a Borel map from the vertices of $\G$ to the vertices of $\H$ that is a graph homomorphism away from a $\G$-null set.
	The Borel  chromatic number of a Borel graph $\G$  is the smallest $k$ so that there exists a Borel function from $Y$ to $\{1,\ldots,k\}$ which is a proper coloring of the graph $\G_{T_1,\ldots,T_d}$.
	The almost Borel chromatic number of a Borel graph $\G$  is the smallest $k$ so that there exists a Borel function from $Y$ to $\{1,\ldots,k\}$ which is a proper coloring of the graph $\G_{T_1,\ldots,T_d}$, away from a $\G$-null set.
	Given a Borel space $Y$ and Borel bijections  $T_1,\ldots,T_d:Y \to Y$, let $\G_{T_1,\ldots,T_d}$ denote the  graph on $Y$ that has edges of the form $(y,T_j^{\pm 1}(y)) \in Y \times Y$ where  $y \in Y$ and $1 \le j \le d$. It follows that $\G_{T_1,\ldots,T_d}$ is a Borel graph.
	It also follows that a Borel subset $Y_0 \subset Y$ is $\G_{T_1,\ldots,T_d}$-null if and only if $\mu(Y_0)=0$ for any $T$-invariant Borel probability measure $\mu$ on $Y$.
	\begin{cor}\label{cor:Borel_3_coloring}
		Let $T_1,\ldots,T_d:Y \to Y$ are $d$-commuting  Borel bijections of a standard Borel space $Y$ that generate a free $\ZD$-action, and let $\H$  be a finite connected graph  that  is non-bipartite. Then there exists a Borel almost graph homomorphism from $\G_{T_1,\ldots,T_d}$ to $\H$. In particular the almost Borel chromatic number of $\G_{T_1,\ldots,T_d}$  is either $2$ or $3$. It is equal to $2$ if and only if there exists a  two set Borel partition ${Y_0,Y_1}$ of $Y$ modulo a null set so that $T_i(Y_j)=Y_{1-j}$ for $i=1,\ldots,d$ and $j=0,1$.
		 
	\end{cor}

\begin{proof}
	Let $T_1,\ldots,T_d:Y \to Y$ be as above, and let $(Y,T)$ be the  Borel $\ZD$ dynamical system generated by $T_1,\ldots,T_d$.
	Let $\H=(V_\H,E_\H)$ be a finite connected graph  that  is non-bipartite.
	By Theorem \ref{thm: universality of hom-shifts}, $\Hom(\ZD,\H)$ admits a flexible sequence so  
	by Proposition \ref{prop:flex_factor}, there exists a $T$-invariant null set $Y_0 \subset Y$ and a Borel factor map $\pi:Y\setminus Y_0 \to \Hom(\ZD,\H)$.
 	The set $Y_0$ is $\G_{T_1,\ldots,T_d}$-null.
 	Choose an arbitrary $v_0 \in V_{\H}$. 
 	Define a  map $\phi:Y \to V_{\H}$  by 
 	$$\phi(y)=\begin{cases}
 	\pi(y)_0 & y \in Y \setminus Y_0\\
 	v_0 & y\in Y_0.
 	\end{cases}$$
 	It follows directly that $\phi$ is a Borel almost graph homomorphism.
 	Applying this to the case where $\H$ is the complete graph on three vertices, we see that the almost Borel chromatic number of $\G_{T_1,\ldots,T_d}$  is at most $3$. It is clear that the almost Borel chromatic number of $\G_{T_1,\ldots,T_d}$  is at least $2$, with equality if and only if the graph is ```almost Borel bipartite'' in the sense that there exists  a  two set Borel partition ${Y_0,Y_1}$ of $Y$ modulo a null set so that $T_i(Y_j)=Y_{1-j}$ for $i=1,\ldots,d$ and $j=0,1$. 	  
\end{proof}
	More generally, Theorem \ref{thm: universality of hom-shifts} implies that for  any $d$-commuting  Borel bijections  $T_1,\ldots,T_d:Y \to Y$ and any connected finite graph $\H$ that  is non-bipartite there exists a Borel map from $Y$ to the vertices of $\H$ which is a graph homomorphism on a full subset of $Y$.
	We remark that in order to deduce this corollary alone we did not need to prove that $h(\t \C)$ is equal to the topological entropy of $\Hom(\ZD,\H)$, nor did we need the full strength of Theorem \ref{thm:spec_sequence_implies_univesality}.  As mentioned in the introduction,  Gao, Jackson, Krohne and Seward formulated a much stronger version of Corollary \ref{cor:Borel_3_coloring} in \cite{gao2018continuous}: The Borel chromatic number of such graphs is at most 3. 

\section{Universality of dimers and rectangular tilings}\label{section:for Tilings and Flexibility for Dominoes}

In this section we use our main result to prove universality for dimers and more generally for   rectangular tilings in $\ZD$.
Let  $\T$ be a finite collection of finite subsets of $\ZD$, which we refer to as prototiles. A $\T$-tiling of $\ZD$ is a partition of $\ZD$ into pairwise disjoint translates of elements of $\T$. We denote the space of all $\T$ -tilings by $X^\T$ and refer to it as the \emph{tiling space} corresponding to $\T$. \index{Definitions and notation introduced in Section 9 and Section 10!$X^\T$ for a set of prototiles $\T$}

In this section we consider rectangular prototiles. To every $\mi = (i_1,\ldots,i_d) \in \NN^d$ we associate the rectangular prototile  
\begin{equation}\label{eq:rect_prototile}
T_\mi = \{1,\ldots, i_1\} \times \ldots \times \{1,\ldots,i_d\}.\index{Definitions and notation introduced in Section 9 and Section 10!$T_\mi$}
\end{equation}

We call a tiling space with rectangular prototiles is called a rectangular tiling shift. Rectangular tilings spaces are naturally in  one-to-one correspondence with finite subsets of $\NN^d$. Given a finite subset $F \subset (\NN)^d$, we denote by $\T_{(F)}$ the tiling set corresponding to the prototiles
$$\T_{(F)} = \{T_\mi:~ \mi \in F\}.\index{Definitions and notation introduced in Section 9 and Section 10!$T_{(F)}$}$$
We refer to  $X_{(F)}=X^{\T_{(F)}}$\index{Definitions and notation introduced in Section 9 and Section 10!$X_{(F)}$, the rectangular tiling shift for the set of tiles $\T_{(F)}$} as the rectangular tiling shift corresponding to $F$.

A particularly interesting and well studied instance of a rectangular tiling shift is that of \emph{dimers} or \emph{domino tilings} in $\ZD$ where 
$$F  = \{\vec{d}+\vec{e}_1,\ldots,\vec{d}+\vec{e}_d\}, \mbox{ with } \vec{d}= (\underbrace{1,\ldots,1}_d)=\sum_{t=1}^d \vec{e}_t.$$

Dimer tilings also correspond to perfect matchings in the standard Cayley graph of $\ZD$. 
There are significant and deep results about dimers in $\ZZ^2$, in particular the topological entropy of the corresponding tiling space is known and  much more is known about the measure of maximal entropy. We refer for instance to the celebrated  Cohn-Kenyon-Propp variational principle for domino tilings   \cite{MR1815214}.

More general  rectangular tiling shift have been considered by Einsedler who studied their shift cohomology  \cite{MR1836431} and by Pak, Sheffer and Tassy who studied some algorithmic aspects of rectangular tilings \cite{MR3530972}.

Call a  set $F \subset \NN^d$ \emph{coprime} \index{Definitions and notation introduced in Section 9 and Section 10!coprime sets $F\subset \NN^d$} if projecting it onto each coordinate yields a coprime set.
So $F \subset \NN^d$ is coprime if $\gcd(\{ i_t:~ \mi \in F\})=1$ for every $1 \le t \le d$, where $i_t= \mi \cdot \vec{e}_t$ is the projection of $\mi$ onto the $t$'th coordinate.

We recall the following result, now known as the \emph{$\ZD$-Alpern Lemma} \cite{MR661814}: Let $F \subset (\NN)^d$ be a coprime finite set, and let $(Y,\mu,T)$ be an free measure preserving $\ZD$-system. Then there exists a Borel $T$-invariant subset $Y_0 \subset Y$ such that $\mu(Y_0)=1$ and an equivariant Borel map $\pi:Y_0 \to X_{(F)}$.
The above result is due to Alpern proved this for $d=1$, and to Prikhod'ko \cite{MR1716239} and  \c{S}ahin \cite{MR2573000} for arbitrary $d$. 
We note that  in  \cite{MR1716239,MR2573000} it is further proved that given any strictly positive probability distribution $(p_T)_{T \in F}$ on $F$ the map $\pi$ can be chosen so that almost surely with respect to $\mu \circ \pi^{-1}$, the proportion of tiles of type $T$ is $p_T$. This also applies when $F$ is infinite \cite{MR1716239,MR2573000}. 

\begin{thm}\label{theorem: flexible tiling}
	If $F \subset (\NN)^d$ is coprime and $|F|>1$ then $X_{(F)}$ admits a flexible marker sequence of patterns and is thus $h$-universal for some $h >0$.
\end{thm}

The case $|F|=1$ where $F$ is coprime corresponds to the trivial one point system. For $n \in \NN$ we denote a $\ZD$-box of side-length $n$ by 
\begin{equation}\label{eq:B_n_def}
B_n = \{1,\ldots,n\}^d\index{Definitions and notation introduced in Section 9 and Section 10! $B_n$}
\end{equation}
(we need this notation since the boxes $F_n$ have odd side-lengths, and we will need even ones as well here.)

Roughly speaking, the flexible marker sequence of patterns will consist of perfect tilings of $\T_F$-tiling of a translate of  boxes whose side length is divisible by  certain integers, and with a specific tiling by a ``marker pattern'' near the boundary.

 An immediate corollary of Theorem \ref{theorem: flexible tiling}  is  the following ``almost-Borel'' $\ZD$-Alpern's Lemma:
\begin{cor}
	Let $F \subset (\NN)^d$ be a coprime finite set, and let $(Y,T)$ be a free Borel $\ZD$ dynamical system. Then there exists a full Borel $T$-invariant subset $Y_0 \subset Y$ and an equivariant Borel map $\pi:Y_0 \to X_{(F)}$.
\end{cor} 
The proof of this corollary follows as did Corollary \ref{cor:Borel_3_coloring}. Theorem \ref{theorem: flexible tiling} additionally says that if $(Y,T)$ has sufficiently low entropy it is possible to make the equivariant tiling  $\pi:Y_0 \to X_{(F)}$  injective. We remark that although Theorem \ref{theorem: flexible tiling} does not directly recover the part of Alpern's lemma about specifying a probability distribution for the prototiles, it is possible to extract this part of the result by formulating a result about the possibility to control the push forward of a measure $\mu$ when embedding $(Y,T,\mu)$ into a flexible system $(X,S)$.

Notice that for $d=1$, if $\T$ is a coprime tile set then $X^\T$ is a mixing shift of finite type
so Alpern's Lemma follows from the mixing SFT version of Krieger's  generator theorem (stated in  \cite{MR0422576}, see   \cite[Theorem 28.1]{MR0457675} for a detailed proof).

For domino tilings in $\ZZ^2$ we can say a little more:
\begin{thm}\label{thm: domino universal}
 The  domino tiling in $\ZZ^2$ admits a flexible marker sequence of patterns $\t \C$ 
 such that $h(\t \C)$ is equal to the topological entropy and so that every admissible pattern appears in some element of $\t \C$. Thus the subshift of domino tilings in $\ZZ^2$ is fully ergodic universal and almost Borel universal.
\end{thm}

In the case when $(\underbrace{1,\ldots,1}_d)\in F$  the corresponding space $X_{(F)}$ is strongly irreducible and has dense periodic points.
In this case universality of  $X_{(F)}$ in the ergodic sense follows from an earlier result of  \c{S}ahin and Robinson \cite{MR1844076}. This is not the case for  domino tilings in $\ZD$ or for more general rectangular tiling spaces.  See for instance \cite{MR1658619}.

\begin{remark}
If $F$ is not coprime and $h(X_{(F)})>0$, the subshift $X_{(F)}$ fails to be $h$-universal for any $h>0$ due to periodicity issues, but it is still $h$-universal for some $h>0$ with respect to a subaction of the finite-index subgroup of $\ZD$  corresponding to the g.c.d of the side lengths of the tiles. This follows from the fact when $F$ is not coprime, we can restrict to tilings of $\mathbb{Z}^d$ by $F$  where the vertices of the tiles sit on a sublattice of $\mathbb{Z}^d$, and so up to choosing a base point one can reduce to the coprime case.
\end{remark}

One of the ingredients of the proof of Theorem \ref{theorem: flexible tiling} is the following lemma.

\begin{lemma}\label{lemma:extension of rectangles}
Let $N,n, n',M\in \N$ and $\mi \in \Z^d$ such that the translate of $B_{nM}+F_N$ centered at  $\mi$ is contained in $B_{(n+n')M}$. Then $B_{(n+n')M} \setminus (\mi + B_{nM})$ can be partitioned into rectangular shapes of which one of the sides is greater than or equal to $N$ and the rest are multiples of $M$.
\end{lemma}

\begin{proof}
 The proof for higher dimensions follows by induction on $d$.
For  $d=1$, $B_{(n+n')M} \setminus (\mi + B_{nM})$ is a disjoint union of two intervals of length greater than or equal to $N$, so this gives the required partition.

For the induction step, suppose the result is known for dimensions less than $d$.   Extend two opposite faces of $\mi + B_{nM}$ to get a partition of $B_{(n+n')M}$ into three parts. The parts which do not contain $\mi + B_{nM}$  have one side of length greater than or equal to $N$ and the rest are multiples of $M$. Now the partition element which contains $\mi + B_{nM}$ is the product of a $d-1$ dimensional  instance of the induction with  $\{1,\ldots,nM\}$. This can be partitioned as required by the induction hypothesis.
\end{proof}

Given a partition as in Lemma \ref{lemma:extension of rectangles}, we will need to tile each such rectangular shape of the partition with elements of our coprime tile set. This is given by the following lemma.

\begin{lemma}\label{lemma: tiling simple rectangles}
Suppose  $F \subset \NN^d$ is finite and  coprime. Let $M$ be the product of all the side lengths of the rectangles in $\T_F$. 
Then each of the following conditions on $\mi \in \NN^d$ is sufficient so that $\T_F$ can tile  $T_\mi$:
\begin{enumerate}
\item
$\mi \in M\NN^d$.
Equivalently, all of the side lengths of $T_\mi$ are positive integer multiplies of $M$.
\item
There exists $m \in \NN$ and $1\le t \le d$ so that
$\mi \in m \vec{e}_t + M \NN^d$.
Equivalently, one of its side lengths of $T_\mi$ is greater than or equal to $M$ and the rest are multiples of $M$.
\end{enumerate}
Furthermore, in the first case the number of possible tilings of $T_\mi$ is at least
$|F|^{M^{-d}|T_\mi|}$.
\end{lemma}

\begin{proof}
Let $M$ be as above.
It is clear that $T_\mj$ can tile $B_M$ for  every $\mj \in F$ by the obvious grid tiling. 
Thus, if $\mi \in M\NN^d$,  $T_\mi$ can be tiled by translates of $B_M$, each of which can be tiled by $\T_F$ in at least $|F|$ different ways.
It follows that $T_\mi$ can be tiled by $\T_F$, and there are  at least $|F|^{M^{-d}|T_\mi|}$ possible tilings.
 
Now assume  that $\mi \in  M\NN^d + m \vec{e}_1$, where $m \in \NN$. By our assumption that $F$ is coprime,
$$\gcd\left( \left\{ j_1:~\mj \in F\right\}\right) =1,$$
 where $j_1= \mj \cdot \vec{e}_1$ is the first coordinate of $\mj$.
Let $i_1 = \mi \cdot \vec{e}_1$. Our assumption is that $i_1 > M$. Thus 
there exists   $c \in \NN^F$ so that $i_1 = \sum_{\mj \in F} c_\mj j_1$, where again  $j_1= \mj \cdot \vec{e}_1$ is the first coordinate of $\mj$ (this is related to the well-known Diophantine Frobenius problem or the coin problem. For more about this have a look at \cite{MR0006196}).
For $\mj \in F$ let $\mj^* = (c_\mj j_1, i_2,\ldots,i_d)$, where $i_t= \mi \cdot \vec{e}_t$ is the $t$-th coordinate of $\mi$.
By our assumption $i_2,\ldots,i_d \in M \NN$. It follows that $T_\mj$ can tile $T_{\mj^*}$. By construction, $T_\mi$ can be tiled by $\{T_{\mj^*}:~\mj \in F\}$ by stacking them next to each other in direction $\vec{e}_1$. 
\end{proof}

Let 
\begin{equation}\label{eq:C_n_tiling_def}
C_n = \left\{a \in \L(B_{nM}, X_{(F)}):~ a \mbox{ is a perfect } \T_F  \mbox{ tiling of } B_{nM} \right\}.\index{Definitions and notation introduced in Section 9 and Section 10!$C_n$, a flexible sequence of patterns for rectangular tiling shifts}
\end{equation}

\begin{lemma}\label{lem:tiling_felxible}
Suppose $k<n$. Consider a $(B_{kM}+F_M)$-separated set $K \subset B_{nM}$ so that the translate of $\mi + B_{kM}+F_M$ is contained in $B_{nM}$ for all $\mi \in K$. Then for any  $w \in C_k^K$ there exists $a \in C_n$ such that $S^{\mi}(a)\mid_{B_{kM}}= w_\mi$ for all $\mi \in K$.
\end{lemma}
\begin{proof}
    To prove the lemma it is enough to show that for any $k<n$ and any $(B_{kM}+F_M)$-separated set $K \subset B_{nM}$ so that the translate of $\mi + B_{kM}+F_M$ is contained in $B_{nM}$ for all $\mi \in K$, there is a perfect tiling of $B_{nM} \setminus (K+B_{kM})$ by the tiles of $\T_F$.
    
	The proof will involve an appropriate division of 
	$B_{nM}\setminus K+ B_{kM}$ into rectangular shapes which can be tiled by rectangles in $\T_F$ by Lemma \ref{lemma: tiling simple rectangles}.
	
	\begin{enumerate}
		\item
		Let $\P_1$ be the partition of $B_{nM}$ into translates of $B_M$. 
		\item
		For all $\mi \in K$, we let $\tilde B_\mi$ be the union of all partition elements of $\P_1$ which intersect $\mi + B_{kM}+F_M$.
		\item
		By Lemma \ref{lemma:extension of rectangles}, for every $\mi \in K$, $\t B_\mi \setminus \mi +B_{kM}$ is partitioned into rectangular shapes of which one of the sides is greater than or equal to $M$ and the rest are multiples of $M$. The partition elements of $\P_1$ give a partition of  
		$$B_{nM}\setminus \biguplus_{\mi \in K}\tilde  B_\mi$$
		into translates of $B_M$. By Lemma \ref{lemma: tiling simple rectangles}, each of these can be tiled by $\T_F$. 
	\end{enumerate}
This completes the proof.
\end{proof}
Now we can proceed towards the proof of the theorem.
\begin{proof}[Proof of Theorem \ref{theorem: flexible tiling}]

Recall that $C_n$ has been defined by \eqref{eq:C_n_tiling_def}. Let $\t C_n \subset \L(X^\T, F_n)$ consist of the patterns $a \in  \L(X^\T, F_n)$ such that 
\begin{enumerate}
	\item There exists $\t a \in C_{M \lceil\frac{2n+1}{M}\rceil}$ so that the ``centered'' $F_n$ pattern in $\t a$ is equal to $a$. In other words, $a$ can be extended to a perfect tiling of a slightly bigger box whose side-length is divisible by $M$. If  $M\lceil \frac{2n+1}{M}\rceil$ is even, we choose an ``approximate center'' for  $B_{M\lceil \frac{2n+1}{M}\rceil}$.
	\item The restriction of $a$  to the  ``centered'' $B_{M\lfloor \frac{2n+1}{M}\rfloor}$  of $F_n$ is also  a perfect $\T_F$ tiling, where the two `outermost layers'' of thickness $M$ are each tiled by a single tile $T_{\mi^{(1)}}$ and $T_{\mi^{(2)}}$ respectively, where $\mi^{(1)} \ne \mi^{(2)}$ and $\mi^{(1)},\mi^{(2)} \in F$. 
\end{enumerate}
  By Lemma \ref{lem:tiling_felxible} the first property implies that $\t C_n$ is a flexible sequence of patterns (In general, whenever we have a flexible sequence of patterns ``along a subsequence with small gaps'' we can obtain a new flexible sequence along the integer). The second property implies the marker property (two such patterns cannot overlap too much).
	\end{proof}

Now we will prove the universality for the domino tiling shifts.

\begin{proof}[Proof of Theorem \ref{thm: domino universal}]

Let $D= \{(1,2),(2,1)\} \subset \NN^2$, and let $X_{(D)}$  be the domino tiling shift.
 As in the proof of the previous theorem we let $M=4$ be the product of all the side lengths of rectangles in $\T_D$.
We already proved in the Theorem \ref{theorem: flexible tiling} that the sequence $\t \C= (\t C_n)_{n=1}^\infty$ is a flexible marker sequence of patterns. The additional claim is that $h(\t \C)$ is equal to the topological entropy of $X_{(D)}$.
This is essentially the statement that the number of perfect matchings of $B_{2n}$ is $e^{|B_{2n}|h(X_{(D)})+o(n^2)}$ as $n \to \infty$.  
This  follows almost directly from well known results about dimers in $\ZZ^2$, and in particular  from \cite[Theorems 4.1 and 10.1]{MR1815214}. This seems to be an overkill for what we require. As an alternative, 
one can use older and more basic classical results by Kasteleyn about the  the number of tilings of a $2$-dimensional discrete torus $\ZZ^2/(2n\ZZ^2)$ and of $B_{2n}$, in combination with a ``reflection positivity argument'' of the sort used in the previous section: Observe that the uniform measure on tilings by dominoes of $B_{4n}$ is ``reflection positive'' with respect to reflection along the hyperplanes parallel to the coordinate hyperplanes passing through $(2n+\frac12,2n+\frac12,\ldots, 2n+\frac12)$. Thus it follows as in the proof of Lemma \ref{lem:many_periodic_conf_hom} that 
$$\lim_{n\longrightarrow \infty}\frac{\log{(\text{number of perfect matchings of }T^2_{4n})}}{|T^2_{4n}|}=h(X_{(D)}).$$
Kasteleyn in \cite[the discussion following Equation (27)]{KASTELEYN19611209} further proved that
$$\limsup_{n\longrightarrow \infty}\frac{\log{(\text{number  of perfect matchings of }T^2_{4n})}}{|T^2_{4n}|}=\limsup_{n\longrightarrow \infty}\frac{\log{|C_n|}}{|B_{4n}|},$$
where $T^2_{4n}=\ZZ^2/4n\ZZ^2$.
By these two equations together we conclude that
$$\limsup_{n\longrightarrow \infty}\frac{\log{|C_n|}}{|B_{4n}|}=h(X_{(D)}).$$

The part about ``full universality'' follows from \cite[Theorem 2.1]{MR1836431} which shows that any finite patch of a domino tiling of $\Z^2$ can be extended to a tiling of a square.
 \end{proof}

\section{A Fully Ergodic Universal and almost Borel universal System with a Factor which is Not Ergodic Universal}\label{sec:Not_fully_universal}
Lind and Thouvenot \cite{MR0584588} asked if a factor of a fully ergodic universal dynamical system must also be ergodic universal. 
Here
we provide a negative answer by describing a fully ergodic universal $\ZZ$ subshift that admits a factor which is not ergodic universal.
The construction has some similar features to   Haydn's construction of a subshift with multiple MME's \cite{haydn1998multiple} and also 
to related methods of Quas and \c{S}ahin \cite{quas2003entropy} that can be adapted to produce a $\ZZ^2$-SFT with similar properties.

Let $\A_{+,0} \subset 2\ZZ \cap (0,\infty)$ be a finite set of positive even integers.
Denote $\A_{+,1} =\A_{+,0} +1$, $\A_{-,0} = -A_{+,0}$, $\A_{-,1}=-A_{+,1}$.
Let $M = |\A_{+,0}|=|\A_{+,1}|=|\A_{-,0}|=|\A_{-,1}|$, and suppose $M >4$.
Denote:
\begin{equation}
\A_+ = \A_{+,1} \uplus A_{+,0}, ~ \A_- = \A_{-,1} \uplus A_{-,0} 
\end{equation}
and 
\begin{equation}
\A = \A_{+} \uplus \A_{-}\uplus \{0\}.
\end{equation}
We now describe a subshift $ \hat X \subset \A^{\ZZ}$.
The rules for the subshift  $X$ are:
\begin{enumerate}
	\item A negative even integer must be directly followed by a negative odd integer.
	\item A negative odd integer must be directly  followed by a non-zero even integer.
	\item A positive even integer must be directly followed by a positive odd integer.
	\item A positive odd integer must be directly  followed by non-negative even integer. 
	\item The symbol $0$ must be directly  followed by a non-positive even integer.
	\item Whenever $0$ is directly followed by a sequence of $n_+$ positive integers and then a sequence of $n_-$ negative integers and then another $0$, then $n_+=n_-$.   

\end{enumerate}
In other words, $\hat X$ consists of concatenations of blocks of the form 
$$0^k e^{-}_1o^{-}_1\ldots e^{-}_n o^{-}_n
  e^{+}_1 o^{+}_1\ldots e^{+}_n o^{+}_n0^j,$$
with $e^{-}_1,\ldots,e^{-}_n \in A_{-,0}$, $o^{-}_1,\ldots,o^{-}_n \in A_{-,1}$, $e^{+}_1,\ldots, e^{+}_n \in A_{+,0}$, $o^{+}_1,\ldots,o^{+}_n \in A_{+,1}$, and their limits. Let $\Phi:\hat X \to \ZZ^\ZZ$ be given by 
\begin{equation}
\Phi(x)_i = \min\{x_i,0\}.
\end{equation}
Let $X'= \Phi(\hat X)$.
\begin{prop}\label{prop:fully_universal_with_non_fully_universal_factor}
	The subshift $\hat X \subset \A^\ZZ$ is fully ergodic universal but the factor $X'= \Phi(\hat X)$ is not universal.
\end{prop}
\begin{proof}
	Let
	$$\A_2 = \{ n \in \NN~: 2n \in \A_{+,0} \} \cup \{0\}.$$
	Let $Y \subset \A_2^{\Z}$ be the subshift consisting of strings where the number of symbols between two consecutive  occurrences of  $0$ is divisible by $4$. Then $Y$ is a mixing sofic shift.  
	Let $\Psi:\hat X \to \A_2^{\ZZ}$ be given by
	\begin{equation}
	\Psi(x)_i = \lfloor |x/2| \rfloor.
	\end{equation}
        One can check that $\Psi(\hat{X})=Y$.
	Furthermore, if $y \in Y$ and $y_i =0$ for some $i \in \ZZ$, then $y$ has a unique preimage in $\hat X$ under $Y$.
	Let 
	$$\hat{X}_0 = \{x \in \hat X:~ \exists i \in \ZZ~ x_i =0 \}$$
	and 
	$$Y_0 = \{y \in  Y:~ \exists i \in \ZZ~ y_i =0 \}.$$
	So $\Psi$ induces a Borel isomorphism between $\hat{X}_0$ and $Y_0$.
	$Y_0 \subset Y$  and $\hat X_0 \subset \hat X$ are both dense.  
	Also, 
	$\mu(Y_0) =1$ for every ergodic fully-supported invariant probability measure on $Y$.
	Similarly, $\mu(\hat X_0)=1$ for every ergodic fully-supported invariant probability measure on $\hat X$.
	Also, $\Psi:\hat X \to Y$ is finite-to-one (actually at most $4$ to $1$), so $h(\hat X,S)=h(Y,S)$.
	Now since  $Y$ is mixing sofic shift, it is fully ergodic universal. This shows that every free $\ZZ$-action with entropy smaller than 
	$h(\hat X,S)$ can be realized as an invariant probability measure on $\hat X_0$ with full support, so $\hat X$ is fully ergodic universal.
	
	Let us prove that  $X'$ is not ergodic universal, note that the set of points $x \in X'$ with no occurrences of $0$ is precisely the subshift
$$ X'':= \{ x \in \A_-^\ZZ :~  x_i + x_{i+1} = 1\Mod 2\ \  \forall\  i \in \ZZ\}.$$
This is a non-mixing SFT with period $2$. It follows that any $\mu \in \Prob_e(X',S)$ with $\mu([0]_0)=0$ has a non-ergodic square.
Also, 
	\begin{equation}
		h(X') \ge h(X'') = \log M.
	\end{equation}
	Now if $\mu \in \Prob_e(X',S)$ and $\mu([0]_0)>0$ then by ergodicity $\mu$-almost every $x \in X'$ is (up to translation) an infinite concatenation of blocks of the form
	$$x_1, x_2, x_3, \ldots, x_n, 0^m$$    
	with $0 < n < m$ and $x_1,\ldots,x_n \in \A_-$.
	It follows that for  $\mu$-almost every $x \in X'$ the frequency of $0$'s is at least $\frac{1}{2}$ and so by the ergodic theorem
	$$\mu([0]_0) \ge \frac{1}{2}.$$
	So
	$$h_\mu(X',S) \le \frac{1}{2}\log(M)+\log(2).$$
	We conclude that no probability preserving system with ergodic square and entropy between  $\frac{1}{2}\log(M)+\log(2)$ and $\log(M)$ can be modeled as an invariant measure on  $X'$.
	This concludes the proof.
\end{proof}

\section{Mixing properties of subshifts of the $3$-colorings}\label{section:SI_subhshifts}
Let us recall the uniform filling property for $\ZD$ subshifts \cite{MR1844076}:
\begin{defn}
	A subshift $X \subset \A^{\ZD}$ has the uniform filling property (UFP) if there exists $M \in \NN$ such that for all $x,y\in X$ and $n\in \N$ there exists $z \in X$ such that $z\mid_{F_n}= x|_{F_n}$ and $z\mid_{\ZD \setminus F_{n+M}} = y\mid_{\ZD \setminus F_{n+M}}$.
\end{defn}
As we mentioned earlier, Robinson and \c{S}ahin \cite{MR1844076} have shown that if a $\ZD$ subshift of finite type $X$  has the \emph{uniform filling property (UFP)} (plus a condition on periodic points) then it is universal. 

As an early attempt to resolve the universality of the $3$-coloring subshift $X_3$ we  checked if it could be the case that for every $\epsilon >0$ there exists a subshift $Y \subset X_3$ with the UFP such that $h(Y')>h(Y)-\epsilon$. The result of  \cite{MR1844076} could then have been applied to prove that any such $Y$ is universal; this would show that $X_3$ is universal.  Pavlov used this argument  to prove universality for subshifts of finite type with ``nearly full entropy'' \cite{MR3162822}. Note that this does not automatically imply  full ergodic universality.
If we allow ourselves to restrict to the subaction of the shift by $(2\ZZ)^d$, this argument works for the $X_3$, and in fact for any mixing hom-shift.
Eventually we realized that this  approach cannot be used to prove universality for $3$-colorings, because of the following result:

\begin{prop}\label{prop:No_SI_in_C3}
	Suppose $d \ge 2$ and that  $Y$ is a  subshift of the $3$-colorings subshift  in $\ZD$. Then $Y$ does not have the UFP.
\end{prop}

\begin{remark}
	In response to a question posed by M. Boyle, Quas and \c{S}ahin \cite{quas2003entropy} constructed a topologically mixing $\ZZ^2$ SFT $\overline{X}$ and a number $h_0 \in (0,h(\overline{X},S))$ such that if $Y \subset \overline{X}$ has UFP then $h(Y,S) < h_0$.
	 Proposition \ref{prop:No_SI_in_C3} shows that a  somewhat stronger phenomena holds for the $3$-coloring subshift.
\end{remark}

The main tool we use in this proof is the so called \emph{height cocycle} or \emph{height functions} associated to $3$-colorings. A (real-valued) cocycle for a topological $\ZD$ dynamical system $(X,S)$ is a function $c:X \times \ZD \to \RR$ satisfying the relation
	\begin{equation}
	\label{eq:cocycle}c(x, \mi+ \mj)= c(x, \mi)+ c(S^{\mi}(x),\mj) \mbox{ for every } x \in X,~ \mi,\mj \in \ZD.
	\end{equation}
	
	Note that \eqref{eq:cocycle} implies that $c(x,\vec{0})=0$. The height cocycle $c:X_3 \times \ZZ^2 \to \ZZ$ is uniquely defined by the following properties:

	$$ x_\mi - x_\mj = c(x,\mi)-c(x,\mj) \mod 3 \mbox{ whenever } \mi \mbox{ is  adjacent  to } \mj,$$
	and
	$$ |c(x,\vec{e}_j) | =1 \mbox{ for all } x \in X_3 \mbox { and } 1\le j \le d.$$
	In other words, for every $x \in X_3$,  $\tilde x= (c(x,\mi)_{\mi \in \ZD}) \in \ZZ^{\ZD}$ is the unique graph homomorphism from the standard Cayley graph of  $\ZD$ to the standard Cayley graph of $\ZZ$ that satisfies $x_\mi- x_{\m 0} =\tilde x_\mi \mod 3$ for every $\mi \in \ZD$ and $\tilde x_{\m 0}=0$. We will use the following two facts about these cocycles. If for some connected set $A\subset \Z^d$, $x,x' \in X_3$ are such that $x|_A=x'|_A$ then for all $\mi, \mj\in A$
	\begin{equation}
c(S^{\mi}(x), \mj-\mi)=c(S^{\mi}(x'), \mj-\mi).\label{equation:equality of heights}
\end{equation}
Lastly 
	\begin{equation}\label{eq:c_lip}
	|c(x,\mi)| \le \|\mi\|_1\ \forall\  x \in X_3\text{ and } \mi \in \ZD.
	\end{equation}
	See  \cite{MR3552299,galvin_homomorphism_cube_2003,schmidt_cohomology_SFT_1995} for details and further references.

\begin{proof}[Proof of Proposition \ref{prop:No_SI_in_C3}]
	
	Call a subshift $Y \subset X_3$ \emph{quasiflat} if 
	\begin{equation}\label{eq:quasiflat}
	\sup \{ c(y',\mi)-c(y,\mi):~  \mi \in \ZD,~ y,y' \in Y\} < \infty.
	\end{equation}
	
	We claim that if $Y$ is not quasiflat, then $Y$ does not have UFP.
	Indeed, suppose that $Y$ is not quasiflat. Fix $M \in \NN$. Then there exists $y^{(1)},y^{(2)} \in Y$, $n \in \NN$, $\mi,\mj \in F_n \setminus F_{n-1}$ such that
	$$c(S^{\mi}(y^{(1)}),\mj-\mi)- c(S^{\mi}(y^{(2)}),\mj-\mi) > 4M.$$
	
	Now find $\mi',\mj' \in F_{n+M}\setminus F_{n+M-1}$ such that $\|\mi -\mi'\|_1 \leq M$ and $\|\mj-\mj'\|_1 \leq M$.
	Then it follows from \eqref{eq:cocycle} and \eqref{eq:c_lip} that
	$$\left|c(S^{\mi}(y^{(k)}),\mj-\mi) - c(S^{\mi'}(y^{(k)}),\mj'-\mi')\right| \leq 2M \mbox { for } k=1,2.$$
	Thus
	$$c(S^{\mi}(y^{(1)}),\mj-\mi)- c(S^{\mi'}(y^{(2)}),\mj'-\mi') > 2M.$$
	Suppose there exists $y \in Y$ such that $y\mid_{F_n}=y^{(1)}\mid_{F_n}$ and $y\mid_{\ZD \setminus F_{n+M-1}}=y^{(2)}\mid_{\ZD \setminus F_{n+M-1}}$.
	Then by \eqref{equation:equality of heights}
	$$c(S^{\mi}(y),\mj-\mi)=c(S^{\mi}(y^{(1)}),\mj-\mi)$$
	and
	$$c(S^{\mi'}(y),\mj'-\mi')=c(S^{\mi'}(y^{(2)}),\mj'-\mi').$$
    We conclude that 
	$$c(S^{\mi}(y),\mj-\mi)- c(S^{\mi'}(y),\mj'-\mi') > 2M,$$
    contradicting   \eqref{eq:cocycle} and \eqref{eq:c_lip}.
    This shows that a subshift $Y$ that is not quasiflat does not have the UFP.

    Now suppose $Y \subset X_3$ is quasiflat. We will show that in this case that $Y$ cannot even be topologically mixing, and in particular does not have UFP.
    
Since $Y$ is a quasiflat, for every $y \in Y$ the map $\phi_y:\ZD \to \RR$ given by  
    $\phi_y(\mi)= c(y,\mi)$ is a \emph{quasimorphism} on the group $\ZD$, in the sense that
    $$\sup_{\mi,\mj \in \ZD}|\phi_y(\mi+\mj)-\phi_y(\mi)-\phi_y(\mj)| < \infty.$$
    Then by a well known simple  argument the homogenized quasimorphism \cite{MR2026941,MR2527432} $\overline{\phi}:\ZD \to \RR$ given by 
    $$ \overline{\phi}(\mi) = \lim_{n \to \infty}\frac{\phi_y(n \cdot \mi)}{n}$$
    is a group homomorphism that satisfies $|\overline{\phi}(\mi)| \leq \|\mi\|_1$.
    Since $Y$ is quasiflat it follows that $\overline{\phi}$ is independent of the choice of $y$ and 
    $$D_Y = \sup_{y \in Y, \mi \in \ZD}|c(y,\mi)- \overline\phi(\mi)| < \infty.$$
    We can find  $y \in Y$ and $\mi_0\in \ZD$ such that
    $$|c(y,\mi_0)- \overline\phi(\mi_0)| > D_Y - \frac{1}{100}.$$ 
    Let $n = \|\mi_0\|_1$ and 
    $$L_Y = \{\mj \in \ZD:~ |\overline{\phi}(\mj)| \le 2 \mbox{ and } \mj \mbox{ is odd} \}.$$
Since $\overline \phi$ is a group homomorphism there exists $\m s\in \R^d$ such that $\overline{\phi}(\mi)=\langle \mi, \m s\rangle$. By taking integer approximations of the zeros of the inner product $\langle \cdot , \m s\rangle$ we get that $L_Y$ is an infinite set.  Since $Y$ is mixing, for large enough $\m j_0\in L_Y$ there exists $z\in Y$ such that
$$z|_{F_n}=y|_{F_n}\text{ and } S^{\m j_0}(z)|_{F_n} = y|_{F_n}.$$ 

In particular $z_{\m 0}=z_{\mj_0}$ and $z_{\mi_0}=z_{\mi_0+\mj_0}$ so $c(z,\mj_0) =0  \mod 3$ and $c(S^{\mi_0}(z), \mj_0)= 0 \mod 3$.
Since $\mj_0$ is odd it follows that $c(z,\mj_0) \ne 0$ and $c(S^{\mi_0}(z), \mj_0) \ne 0$.
Thus
$|c(z, \mj_0)|, |c(S^{\mi_0}(z), \mj_0)|\geq 3$. Assume without the loss of generality that
$$c(z,\mi_0)- \overline\phi(\mi_0) > D_Y - \frac{1}{100};$$
the proof is similar in the other case. Since $F_n$ is connected we have that either
$$c(z, \mj_0)=c(S^{\mi_0}(z),\mj_0)\geq 3\text{ or }c(z, \mj_0)=c(S^{\mi_0}(z),\mj_0)\leq -3.$$

If the former is true then
$$c(z, \mi_0+\mj_0)=c(z, \m j_0)+c(S^{\mj_0}(z), \mi_0) \geq 3+\overline{\phi}(\mi_0)+D_Y-\frac{1}{100}$$
while the latter implies
$$c(S^{\mj_0}(z), \mi_0-\mj_0)=c(S^{\mj_0}(z), -\mj_0)+ c(z,  \m i_0) \geq 3+\overline{\phi}(\mi_0)+D_Y-\frac{1}{100}.$$
But  the choice of $D_Y$  and that $\mj_0\in L_Y$ shows
\begin{eqnarray*}
c(z, \mi_0+\mj_0) \leq& \overline\phi(\mi_0+\mj_0)+D_Y&\leq 2+ D_Y +\overline\phi(\mi_0)\text{ and }\\
c(S^{\mj_0}(z), \mi_0-\mj_0) \leq& \overline\phi(\mi_0-\mj_0)+D_Y&\leq 2+ D_Y +\overline\phi(\mi_0)
\end{eqnarray*}
 contradicting both cases.
\end{proof}

\section{Further questions and comments}
We conclude with some comments and further questions:
\begin{enumerate}
		\item Ergodic vs. almost Borel universality: 
		Is there a general condition under which  ergodic universality of  a compact Borel dynamical system $(X,S)$ implies almost Borel universality?
	Specifically: Suppose $(X,S)$ and $(Y,T)$ are compact dynamical systems and that for every $\mu \in \Prob(Y,T)$ the system  $(Y,\mu,T)$ can realized as an invariant measure on $(X,S)$. Is there an almost-Borel embedding of  $(Y,T)$ into $(X,S)$? It is clearly not sufficient to  require only  that any ergodic $\mu \in \Prob(Y,T)$ has an isomorphic copy in  $\Prob(X,S)$. For example, take $(X,S)$ to be uniquely ergodic and let $(Y,T)$ consist of two disjoint copies of $(X,S)$.
	\item Borel vs. ``almost Borel'' universality: Is it possible to strengthen Theorem \ref{thm:spec_sequence_implies_univesality} and prove that systems satisfying the assumptions  contain a Borel copy of any free system of sufficiently low entropy? In other words, to what extent is it necessary to disregard a null set? Mike Hochman proved that any Borel $\ZZ$ dynamical system with no invariant probability measure admits a $2$-set generator \cite{MR3880210}, and thus deduced Borel universality for mixing $\ZZ$-SFTs and more. Mike Hochman and Brandon Seward  informed us that they have managed to extend this result to Borel actions of arbitrary countable groups.
	\item
	 Dimers in higher dimensions and rectangular tiling shifts: Are $\ZD$-dimers fully ergodic universal when $d >2$? We proved that they are $t$-universal for some $t>0$. It suffices to show that the number of perfect domino tilings of an $F_n$ is $e^{h|F_n|+o(|F_n|)}$ as $n \to \infty$. For this we invoked  some classical results based on ``hard'' computations, in particular  on Kasteleyn's formula for the number of perfect matchings of a finite planar graph. Is there a ``soft method'' to deduce a similar result in greater generality?
		\item Universality for $\mathbb{R}^d$-actions: To what extend do our results and methods apply to  $\mathbb{R}^d$ actions? Does specification imply universality for these systems?
		 Quas and Soo obtained results along theses lines for certain $\mathbb{R}$-actions \cite{MR3008405}.
	See also the Kra-Quas-\c{S}ahin version of Alpern Lemma for $\mathbb{R}^d$-actions \cite{MR3324927}.
	\item Does specification imply universality for actions of more general countable groups? 
	In this paper we do not directly deal with action of groups beyond $\ZD$, but  most of the ergodic theoretic machinery used in our proof (Rokhlin towers, Shannon-McMillan theorem) is available for countable amenable groups.
	In view of Seward's version of Krieger generator theorem for arbitrary countable groups \cite{MR3904452}, it is tempting to ask the question beyond the amenable setting.
	\item Realizing measure preserving actions and Borel actions as  continuous actions  on a manifold: We now know that any free measure preserving $\ZZ$-action is isomorphic to some continuous homeomorphism of the $2$-torus (as a measure preserving dynamical system, with respect to Lebesgue measure). What about actions of more general groups? For instance, what about $\Z^2$-actions? Actions of the free group? $\RR$-flows? (Note that there are no free continuous $\mathbb{R}^{d+1}$-flows on a $d$-dimensional manifold).
 
\item Non compact models: Some natural (non-compact) Polish dynamical systems have been shown to be universal. For instance, the space of entire functions on $\mathbb{C}$ is an example of a non-compact Polish $\mathbb{R}^2$ dynamical system that is universal (in the ergodic sense) \cite{MR1422707}. Can this be interpreted in the context of our results?

	\item Universality of algebraic actions: 
	Which algebraic actions (continuous actions on a compact group that preserve the group structure) are universal? For $\ZZ$-actions we know that ergodicity suffices.
\end{enumerate}
\printindex

\bibliographystyle{abbrv}
\bibliography{universal_models_ZD}
\end{document}